\documentclass[11pt]{amsart}

\date{\today}

\usepackage[textwidth=20mm]{todonotes}

\usepackage{amssymb}

\usepackage{mathtools} %
\usepackage{comment}
\usepackage{hyperref}
\usepackage{marginnote}
\usepackage{color}
\allowdisplaybreaks
\mathtoolsset{showonlyrefs}

\let \div \unefined

\let \Im \undefined

\DeclareMathOperator{\Im}{Im}

\DeclareMathOperator{\Div}{div}
\DeclareMathOperator{\div}{div}

\DeclareMathOperator{\WF}{WF}
\DeclareMathOperator{\curl}{curl}

\DeclareMathOperator{\supp}{supp}

\DeclareMathOperator{\diag}{diag}

\DeclareMathOperator{\grad}{grad}
\DeclareMathOperator{\adj}{adj}

\let \sectionsymbol \S

\newcommand{\R}{\mathbb{R}}

\newcommand{\C}{\mathbb{C}}
\renewcommand{\S}{\mathbb{S}}

\newcommand{\p}{\partial}
\newcommand{\n}{\nabla}

\newcommand{\oM}{\overline{M}}

\newcommand{\rg}{\rangle}
\renewcommand{\lg}{\langle}

\newcommand{\foliation}{\mathsf{x}}
\newcommand{\bo}{\partial M}

\newcommand{\In}{\text{in}}
\newcommand{\Out}{\text{out}}
\newcommand{\Ev}{{\text{ev}}}

\newcommand{\calE}{\mathcal{E}}
\newcommand{\calF}{\mathcal{F}}

\newcommand{\calH}{\mathcal{H}}
\newcommand{\calL}{\mathcal{L}}
\newcommand{\calM}{\mathcal{M}}

\newcommand{\calU}{\mathcal{U}}

\newcommand{\bv}{\mathbf{v}}
\newcommand{\bu}{\mathbf{u}}
\newcommand{\bw}{\mathbf{w}}

\newcommand{\bz}{\mathbf{z}}

\renewcommand{\a}{\alpha}
\renewcommand{\b}{\beta}

\renewcommand{\d}{\delta}
\let\epsilon\varepsilon
\newcommand{\e}{\epsilon}

\newcommand{\G}{\Gamma}

\renewcommand{\r}{\rho}
\renewcommand{\t}{\tau}

\renewcommand{\l}{\lambda}
\newcommand{\s}{\sigma}

\newcommand{\w}{\omega}

\newtheorem{theorem}{Theorem}

\newtheorem{definition}{Definition}
\numberwithin{definition}{section}

\newtheorem{proposition}[definition]{Proposition}

\newtheorem{lemma}[definition]{{Lemma}}

\theoremstyle{definition}
\newtheorem{remark}[definition]{Remark}

\newtheorem{corollary}[definition]{Corollary}

\newtheorem{claim}{Claim}

\newtheorem{assumption}[definition]{Assumption}

\numberwithin{equation}{section}

\newcommand{\bal}{\begin{aligned}}
\newcommand{\eal}{\end{aligned}}

\let \o \undefined
\def \o#1{\overline{#1}}

\let\td\undefined
\def \td#1{\widetilde{#1}}
\let\implies\Rightarrow

\newcommand{\be}[1]{\begin{equation}\label{#1}}
\newcommand{\ee}{\end{equation}}

\begin{document}
\title{The solid-fluid transmission problem}

\author[N. Eptaminitakis] {Nikolas Eptaminitakis}
\address{Department of Mathematics, Purdue University, West Lafayette, IN 47907}

\author[P. Stefanov]{Plamen Stefanov}
\address{Department of Mathematics, Purdue University, West Lafayette, IN 47907}
\thanks{P.S. partly supported by  NSF  Grant DMS-1900475}

\begin{abstract}We study microlocally the transmission problem at the interface between an isotropic linear elastic solid and a linear inviscid fluid.
We set up a system of  evolution equations describing the particle displacement and velocity in the solid, and pressure and velocity in the fluid, coupled by suitable transmission conditions at the interface. We show well posedness for the coupled system and study the problem microlocally, constructing a parametrix for it using geometric optics. This construction describes the reflected and transmitted waves, including mode converted ones, related to incoming waves from either side. We also study formation of surface Scholte waves. Finally, we prove that under suitable assumptions, we can recover the s- and the p-speeds, as well as the speed of the liquid, from boundary measurements. 
\end{abstract}

\maketitle

\section{Introduction}

The analysis of waves meeting an interface between a solid and liquid body is of great interest in seismology, where it is of importance to understand the behavior of seismic waves in the interior of the Earth.
It is well known that the Earth's outer core is liquid, and of course the same is true of the oceans, whereas the crust, mantle and inner core are solid.
Earthquakes occur in the crust or upper mantle, so it is desirable to investigate their behavior when they encounter a liquid medium. 
The purpose of the present work is to study microlocally the transmission problem at the  interface between an isotropic linear elastic solid and a compressible inviscid fluid.
We assume that the interface is smooth and that the Lam\'e parameters $\l_{\rm s}$, $\mu_{\rm s}$ in the solid, the bulk modulus $\l_{\rm f}$ in the fluid, and the two respective densities $\r_{\rm s}$ and $\r_{\rm f}$   are spatially varying. We  construct and justify a parametrix (an approximate solution up to a smooth error) for a coupled system describing the pressure and particle velocity in the fluid side and the particle displacement and velocity in the solid. 
The pressure-velocity in the fluid is coupled with the displacement-velocity in the solid via two  transmission conditions:
the kinematic condition requires that the normal component of the velocity at the interface must match for the two bodies; unlike the case of a solid-solid interface, tangential slipping is allowed.
The dynamic transmission condition requires that  the vector valued traction across the interface must be continuous across it and normal to it. 
Those transmission conditions determine how parametrices constructed separately in the two sides of the interface must be combined to yield a parametrix for the full system.

In seismology the oceans and the outer core are often treated as inviscid fluids: it is mentioned in \cite[p.128]{AkiKeiiti2002Qs} that the assumption of zero viscosity is reasonable for wavelengths and periods typical of seismic waves. 
Models assuming a viscous outer core have also been studied by  some authors, see e.g \cite{10.1111/j.1365-246X.2004.02416.x} and the references there. 
In a solid, the main quantity one is interested in describing is particle displacement.
In a fluid, one is generally more interested in  the fluctuations of hydrostatic pressure and not as much in the displacement (see \cite[\sectionsymbol 2.4.3]{SheriffRobertE1995Es}), so our primary model (\ref{sys_1}-\ref{tr_3}) below involves a linear first order system of coupled velocity-pressure equations in the fluid. 
This system can be easily decoupled into second order  equations for the velocity and the pressure, though the transmission conditions at the solid-fluid interface for a  displacement-pressure or displacement-velocity system do not appear to be very natural from a physical point of view, at least in the time dependent formulation of the problem.
This is our reason for using the coupled first order velocity-pressure system in the fluid, which leads to naturally expressed transmission conditions.
The velocity-pressure system in a fluid was studied, e.g. in \cite{Bal98transporttheory}, and, coupled with a solid  via transmission conditions, in \cite{MR3803777}, \cite{MR4052755} (with constant Lam\'e parameters and densities).
Displacement-pressure systems for a solid-fluid in the stationary formulation have been studied e.g. in \cite{C.J.Luke1995FIAS}, \cite{MR4104321}.
Regarding our assumptions in the solid side, we use the classical model of linear elasticity describing the displacement in an isotropic linear elastic body (see e.g. \cite{AkiKeiiti2002Qs}, \cite{MR1262126}).

In order to simplify the presentation, our setup consists of a fluid occupying a domain $M^-\subset \R^3$, enclosed by a solid occupying a domain $M^+\subset \R^3$, such that $M^-$ and $M:=\o{M^-}\cup M^+$ are diffeomorphic to a ball (see Fig. \ref{fig:domain} below); we write $\G=\o{M^+}\cap \o{M^-}$ for the interface between the two.
If one wished to use a model more closely resembling the Earth structure, one might work on a manifold diffeomorphic to a ball which contains a number of layers, each occupied by a solid or fluid, with transmission conditions imposed at the various interfaces between layers; see \cite{DahlenF.A1998TGS}, \cite{dehoop2017elasticgravitational} and \cite{stefanov2020transmission}, with  only solid layers in the latter. 
Since the microlocal analysis of the transmission systems is local in nature and the solid-solid and fluid-fluid transmission problems are handled, for instance, in \cite{stefanov2020transmission}, our study of the transmission problem does not become less comprehensive by our choice of a simplified setup.
We mention that within the regime of linear elasticity it is also possible to use more involved models taking into account factors such as self-gravitation and rotation of the Earth (see e.g. \cite{DahlenF.A1998TGS}, \cite{dehoop2017elasticgravitational}).
One may also work with anisotropic solids; the transmission problem at the interface between anisotropic elastic solids was analyzed microlocally in the recent paper \cite{hansen2021propagation} as part of a study of the propagation of polarizations for geometric systems of real principal type.

The first question we address is the  well posedness of our system of evolution equations.
For this purpose, in Section~\ref{sec:well_posedness}, we turn the initial system for displacement and velocity in the solid, and pressure and velocity in the fluid, into a system of second order equations for the particle displacement fields in both the solid and fluid, subject to transmission conditions. 
This results in a PDE system of the form $\p_t^2\bu= P\bu$, where $\bu=(u^+,u^-)$ is the pair of the displacements in the solid and fluid region respectively, and $P=\diag(P^+, P^-)$
with $P^\pm$ second order matrix differential operators.
We show that $P$ with an appropriate domain $D(P)$ is a self-adjoint operator on $L^2(M^+)\times L^2(M^-)$ (with suitable measures) and produces a solution for given initial Cauchy data using functional calculus.
Well posedness for solid-fluid systems is also shown in e.g. \cite{MR4052755}, \cite{dehoop2017elasticgravitational}. 
We actually take the extra step of identifying the domain of the self adjoint operator $P$ explicitly. 
Although this is not strictly necessary to show well posedness, it is helpful for justifying the parametrix, i.e. showing that our parametrix differs from an actual solution by a smooth error.
For the case of an interface between two fluids, with acoustic equations satisfied on both sides, the justification of the parametrix follows from \cite{MR1183346}.
Identifying the domain of $P$ takes substantial effort; one needs to show regularity estimates closely resembling elliptic regularity estimates for solutions to a transmission problem for a pair of elliptic differential operators with smooth coefficients up to an interface (see e.g. \cite[Ch. 4]{McLean-book}).
However, the operator $P^-$ is not elliptic, thus such regularity results do not appear to be immediately quotable and we had to adapt the proofs to our situation; as they are somewhat lengthy and technical we included them in the Appendix.

Next, we need to construct a parametrix for our solid-fluid system.
The study of the elastic wave system with constant Lam\'e parameters is often simplified using potentials (see e.g. \cite{SheriffRobertE1995Es}).
In this way one obtains a decomposition of elastic waves into shear ($\rm s$) and pressure ($\rm p$) waves, which are transversal and longitudinal respectively.
In 
\cite{stefanov2020transmission}, it was shown that the elastic system with non-constant Lam\'e parameters can be decoupled microlocally, up to lower order matrix pseudodifferential operators.
In this way, from a microlocal point of view, its study reduces to the study of potentials satisfying principally scalar hyperbolic pseudodifferential systems.
For those, the construction of a parametrix   via geometric optics is standard (see e.g. \cite{MR618463}).
One also obtains a decomposition of an elastic wave into a microlocal $\rm s$ and $\rm  p$ wave up to lower order terms.
In the fluid side, we similarly use a potential to reduce the study of the evolution of the  ``momentum density'' $\r_{\rm f}v^-$, where  $v^-$ is the velocity in the fluid, to the study of a scalar hyperbolic equation with a source supported away from the interface at all times.
For such an equation we can again construct a parametrix away from the interface and boundary.

The parametrices constructed on the two sides of the interface between the solid and fluid must be matched using the transmission conditions.
Suppose that we have solutions of the elastic and acoustic wave equation on the solid and fluid side respectively, consisting of  incoming and outgoing waves (incoming/outgoing waves propagate singularities only in the past/future respectively, in their respective domains).
The Dirichlet and Neumann data of those solutions at the interface $\G\times \R$ are coupled by the transmission conditions.
To show microlocal well posedness for the transmission problem, it suffices to show that the Dirichlet data  of the outgoing waves at  $\G\times \R$ can be uniquely produced from Dirichlet data for the incoming ones.
If this is the case, then the geometric optics construction can be used to yield parametrices for the outgoing waves; combining them with parametrices for the incoming ones, we can obtain a parametrix for the full system near the interface.
With the aid of appropriate incoming and outgoing Dirichlet to Neumann maps relating Neumann and Dirichlet data, the system induced by the transmission conditions can be reduced 
to a pseudodifferential system on $\G\times \R$ for the Dirichlet data of the outgoing waves, in a conical neighborhood of the Dirichlet data of the  incoming waves.
In this way, microlocal well posedness of the transmission problem is reduced to the microlocal solvability (ellipticity) of this system.

It turns out that the form of those microlocal systems and their solutions (i.e. of the waves produced) depends on the  traces of the incoming waves, and we have to study six cases separately (we do not investigate the case of wave front sets in the glancing regions, see below).
In some of those cases,  evanescent waves are produced  on either or both sides of the interface, that is, waves which decay exponentially fast away from $\G$.
Those do not  propagate singularities into the interior of the solid or fluid region.
Of particular interest are surface  waves which are evanescent on both sides of $\G\times \R$ and propagate singularities along $\G\times \R$.
In the geophysical literature, surface waves at the interface between a solid and a fluid are known as Scholte waves.
For constant densities and Lam\'e parameters and a flat interface between two solids, the analogous surface waves (known as Stoneley waves), do not always exist;   however, in the constant parameter case, Scholte waves are known to  always be possible (see \cite{ScholteJ.G1947TROE}, \cite[\sectionsymbol 2.5.3]{SheriffRobertE1995Es},
 \cite[p. 156]{AkiKeiiti2002Qs}, \cite{ansell}).

We will always assume that the Dirichlet data of our solutions at $\G\times \R$ are away from the glancing regions in $T^*(\G\times \R)$ with respect to the wave speed of the fluid and the microlocal $\rm p$ and $\rm s$ waves in the solid (see Sections \ref{sec:acoustic} and \ref{sec:elastic}).
The projections to $M$ of bicharacteristic rays emanating from glancing covectors are tangential to the hypersurface $\G$.
The construction of a parametrix for the acoustic or elastic wave equation given Dirichlet data with wave front set in the glancing region corresponding to the acoustic or $\rm s$/$\rm p$ wave speed respectively is more delicate, see e.g. \cite{MR618463, MR1334206, Yamamoto_09}.  
We should also mention that in order  to construct a full parametrix for our system, one also needs to consider the behavior of singularities of elastic waves meeting the outer boundary $\p M$.
We do not pursue this here, since it has been studied in detail in \cite[Section 8]{MR3454376}.

We apply the analysis above  to study the inverse problem of recovering the densities and the Lam\'e parameters of the solid and  fluid from the Neumann to Dirichlet map at the boundary $\p M$. In Theorem~\ref{thm_inverse}, we prove that we can recover %
the shear and the pressure elastic speeds $c_{\rm s}=\sqrt{(\lambda_{\rm s} + 2\mu_{\rm s})/\rho_{\rm s}}$ and $c_{\rm p} = \sqrt{\mu_{\rm s}/\rho_{\rm s}}$ in $M_+$,  and the liquid speed  $c_{\rm f}=\sqrt{\l_{\rm f}/\r_{\rm f}}$ in $M_-$ under a foliation condition. The density  $\rho_{\rm s}$ is also recoverable under some condition. We do not recover $\l_{\rm f}$ and $\r_{\rm f}$ separately though; an attempt to this would make the exposition even longer.  %
The main idea is to reduce this problem to the lens/boundary rigidity one and use the result in \cite{MR3454376}, see also \cite{SUV_elastic, stefanov2020transmission, caday2019recovery}.

The paper is organized as follows: in Section \ref{sec:setup_and_model} we describe our geometric setup and main model and elaborate on the various physical quantities appearing in it.
In Section \ref{sec:well_posedness} we show well posedness for the coupled system of evolution equations in the solid and fluid and identify the domain of the self-adjoint operator $P$ mentioned before. 
In Sections \ref{sec:acoustic} and \ref{sec:elastic} we discuss some necessary background on the geometric optics construction for the acoustic and elastic equation respectively. In Section \ref{sec:well_posedness_microlocal} we study the transmission systems. %
The inverse problem is studied in Section~\ref{sec_IP}. 
Finally, in Section \ref{ssec:justification} we justify the parametrix, i.e. we show that it differs from an actual solution by a smooth error.
In Appendix \ref{appendix_a} we explain how well posedness and parametrix justification work for the solid-solid and fluid-fluid case, quoting some readily available results.
In Appendix \ref{appendix_b} we present two lengthy proofs omitted from Section \ref{sec:well_posedness}.

\section{The setup and main model}\label{sec:setup_and_model}

Suppose ${M} \subset \R^3$ and ${M} ^-\subset \subset {M} $ are precompact domains diffeomorphic to an open ball, and
$g$ is a smooth background Riemannian metric on $M$, whose purpose will be to help us conveniently change coordinates whenever necessary.
Let ${M} ^+={M} \setminus \o{{M}} ^-$ (see Figure \ref{fig:domain}).
We assume that ${M}^+$ is occupied by an isotropic elastic solid and ${M}^-$ is occupied by a compressible inviscid fluid.
Let $\nu$ be the outer pointing unit normal to $\p{M} _+$, and set $\Gamma=\p{{M} _-}$.
We will study the first order system
\begin{subequations}
\begin{align}
	\p_tu^+=w^+\qquad &\text { in } {M}^+\times \R\noeqref{sys_1}\label{sys_1},\\
	\p_{t}w^+=\r_{\rm s}^{-1}Eu^+\qquad &\text { in } {M}^+\times \R\noeqref{sys_2}\label{sys_2}, \\
	\p_tv^-=-\r_{\rm f}^{-1}\n p^-\qquad &\text{ in }{M}^-\times \R \noeqref{sys_3}\label{sys_3}, \\
	\p_tp^-=-\l_{\rm f}\Div v^-\qquad &\text{ in }{M}^-\times \R \noeqref{sys_4}\label{sys_4}, \\
	w^+\cdot\nu=v^-\cdot \nu \qquad &\text{ on }\Gamma\times \R \label{tr_1}, \\
	N(u^+)= -p^-\nu \qquad &\text{ on }\Gamma\times \R\label{tr_2}\noeqref{tr_2}, \\
	N(u^+)=0 \qquad &\text{ on } \p M\times \R,\label{tr_3}
	\intertext
	{with prescribed Cauchy data }
		(u^+,w^+,p^-,v^-)\big|_{t=0}=&(u^+_0,w^+_0,p_0^-,v^-_0).\label{eq_init}
	\end{align}
\end{subequations}
We  will later place assumptions on the data as needed.
The densities $\r_{\rm s}$, $\r_{\rm f}$ are assumed smooth functions of $x$ and positive and the same is assumed for the bulk modulus $\l_{\rm f}$ of the fluid.
In \eqref{tr_1} and throughout, $\,\cdot\,$ denotes pairing with respect to the metric $g$.
Physically, the vector field $u^+$ stands for the displacement field in the solid, and $v^+$ and $p^+$ stand for the velocity and pressure in the fluid, respectively.
In \eqref{sys_2}, $$Eu^+ =\div \s(u^+),$$ where $\s$ is the Cauchy Stress tensor, see \eqref{eq:Cauchy_stress} below.
We denote by
$$N(u^+)=\s(u^+)\cdot \nu$$
 the traction across $\G$ and $\p M$.
The transmission condition \eqref{tr_1} indicates that the normal component of the velocity is continuous across the interface, allowing tangential slipping.
\eqref{tr_2} indicates that the tangential components of the traction at the interface vanish, whereas its normal component is continuous.
At the interface between a solid and vacuum (or air, by approximation) one requires vanishing of the normal traction, i.e. \eqref{tr_3}.
Those transmission and boundary conditions for the interface between a solid and a fluid are physically reasonable and widely used in the geophysical literature, see e.g. \cite[Problem 2.10]{SheriffRobertE1995Es}, \cite[Section 5.2]{AkiKeiiti2002Qs}.

Given $u^+\in C^\infty(\o{{M}}^+;TM)$, the Cauchy Stress tensor is a symmetric (2,0)-tensor field, given  by 
	\begin{equation}\label{eq:Cauchy_stress}
		\s(u^+) =\l_{\rm s} (\div u^+)g^{-1}+2\mu_{\rm s}  d^{\rm s} u^+.
	\end{equation}
In \eqref{eq:Cauchy_stress}, $\l_{\rm s}$, $\mu_{\rm s}$ are the Lam\'e parameters, which are assumed to be smooth, positive and spatially varying on $\o{{M}}^+$ but constant in time. 
We denote by $d^{\rm s}u$ the symmetrized covariant differential of a vector field $u$, with a raised index, becoming a $(2,0)$-tensor field. In local coordinates,
\begin{equation}\label{eq:symmetric_diff}
	(d^{\rm s} u)^{ab}=\frac{1}{2}(\n^a u^b+\n^b u^a)=\frac{1}{2}(g^{ak}u^b{}_{;k}+g^{bk}u^a{}_{;k}+g^{ak}g^{b\ell}g_{k\ell;m}u^m),
\end{equation} where repeated indices indicate summation
 and for a vector field $u$ we write  $\n^au^b=g^{ak}\n_k u^b=g^{ak}(u^b{}_{;k}+ \G_{k\ell}^bu^\ell)$ with $\G_{k\ell}^b$ denoting the Christoffel symbols of $g$
  (also see Remark \ref{rmk:dif} below). Thus  $Eu$ can be written in local coordinates as
\begin{equation}\label{elastic_wave_operator}
	(Eu)^a=\n^a (\l_{\rm s} \n_k u^k)+\n _k \big(\mu_{\rm s} (\n^a u^k+\n^k u^a)\big).
\end{equation}
Note that in the first term above, the covariant derivative of a scalar function, with a raised index, agrees with the gradient $\n=\grad$, and this is the interpretation we will place on $\n f$ for $f$ scalar, i.e. $\n f$ will be an $(1,0)$ tensor field.

\begin{remark}\label{rmk:dif}
	Following \cite[Section 4.3]{MR1262126}, $u^+$ is considered a vector field, but it is also possible to treat it as a covector field, as done in  \cite{stefanov2020transmission};  one can switch between the two by  lowering or raising an index with respect to $g$. 
	A slight advantage of viewing $u^+$ as a vector field is the natural interpretation of the strain tensor as $\frac{1}{2}\calL_{u^+}g$, where $\calL$ denotes Lie derivative (see \cite{MR1262126}). 
	On the other hand, a disadvantage is that the notation $d^{\rm s}$ is more commonly used in the literature to denote  the symmetrized covariant derivative of a covector field $\omega$, with no indices raised.
	Denoting the latter by $d^{\rm s}_{\flat}$, we can see that it is related to $d^{\rm s}$ as defined in \eqref{eq:symmetric_diff} in a natural way.
	In local coordinates we have $(d^{\rm s}_{\flat}\w)_{\a\b}=\frac{1}{2}(\w_{a;b}+\w_{b;\a}-2\G_{ab}^k\w_k)$, thus a computation shows that 
		$d^{\rm s} u^+=(d^{\rm s}_{\flat}(u^+)^\flat)^{\#\#}=\frac{1}{2}(\calL_{u^+}g)^{\#\#},$
	where $\# $ and $\flat$ indicate raising and lowering of indices respectively.
\end{remark}

We make  the following assumption on our initial data, whose relevance will become clear in Section \ref{sec:well_posedness} below.
\begin{assumption}\label{as:int}
	We have $\int_{M^-}p_0^-/\l_{\rm f}\, dv_g=\int_\G u^+_0\cdot \nu\,  dA$, where $dv_g$, $dA$ are the natural measures induced by $g$ on $M^-$ and $\G$ respectively.
	Note that $\nu$ is inward pointing with respect to $M^-$.
\end{assumption}

\begin{figure}[h]
	\includegraphics[page=1]{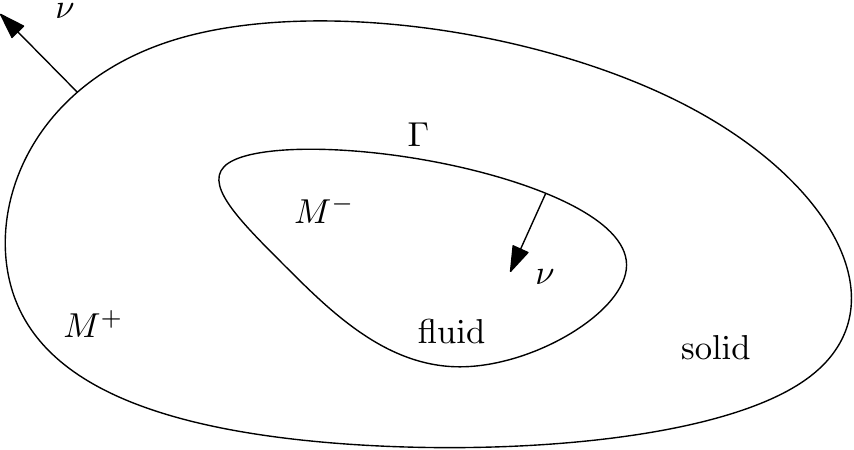}
	\caption{The geometric setting.}
	\label{fig:domain}
\end{figure}

\section{Well posedness of the acoustic-elastic wave equation}
\label{sec:well_posedness}

In order to prove well posedness for the system (\ref{sys_1}-\ref{tr_3}) we consider an auxiliary system, which physically corresponds to equations for the displacement in the solid and fluid.
As we show below, the system (\ref{sys_1}-\ref{tr_3}) with initial conditions \eqref{eq_init} satisfying Assumption \ref{as:int} is equivalent to the system
\begin{subequations}
\begin{align}
	\p_{t}^2u^+=\r_{\rm s}^{-1}Eu^+ &\qquad \text{ in } {M}^+\times \R,\noeqref{sys_1_aux}\label{sys_1_aux}\\*
	\p_t^2u^-=\r_{\rm f}^{-1}\n \l_{\rm f} \div u^-&\qquad \text{ in } {M}^-\times \R\noeqref{sys_2_aux}\label{sys_2_aux}, \\*
	 u^+\cdot \nu=u^-\cdot \nu&\qquad \text{ on } \Gamma\times \R\noeqref{tr_1_aux}\label{tr_1_aux}, \\*
	  N(u^+)=\l_{\rm f} (\div u^- )\;\nu&\qquad \text{ on } \Gamma\times \R\label{tr_2_aux}\noeqref{tr_2_aux}, \\*
	   N(u^+)=0&\qquad \text { on }\p M\times \R,\noeqref{tr_3_aux}\label{tr_3_aux}
\intertext{for $\bu=(u^+,u^-)$ with initial data}
	\bu=\bu_0,\quad \p_t \bu=\bw_0&\qquad \text { at } t=0\label{eq:initial_data_aux}
	\end{align}
\end{subequations}
chosen as follows: given sufficiently regular initial data as in \eqref{eq_init} satisfying Assumption \ref{as:int},
 we choose initial data $(\bu_0,\bw_0)=(u_0^+,u_0^-,w_0^+,w_0^-)$ by taking $w_0^-=v_0^-$ and choosing $u^-_0$ such that 
 \begin{equation}
 	p_0^-=-\l_{\rm f}\div u_0^-\text{ on }M^-,\quad  u_0^+\cdot \nu=u_0^-\cdot \nu\text { on }\G.\label{eq:displ_tr}
 \end{equation}
 We find such a $u_0$ as follows:
solve
\begin{equation}\label{eq:potential_neu}
	\Delta \w_0= -p_0/\l_{\rm f} \text { on }M^-, \quad \p_\nu\w_0= u_0^+\cdot \nu  \text { on }\Gamma,
\end{equation}
and take $u_0^-=\n \w_0$.
Assumption \ref{as:int} guarantees the existence of a solution to \eqref{eq:potential_neu}, unique up to a constant, for appropriate regularity of the initial data.
Moreover, if \eqref{eq:displ_tr} is satisfied (as is the case for a reasonable simple model for a fluid, see e.g. \cite[(8.2)]{AkiKeiiti2002Qs}) then Assumption \ref{as:int} is automatically satisfied. 

The choice of initial data $u_0^-$ satisfying \eqref{eq:displ_tr} is not unique.
Any other such choice $u_0^-{}'$ 
 will differ from $u_0^-$ by a divergence free vector field $z_0$ with $z_0\cdot \nu =0$ on $\G$. 
The solution $\bu'$ of (\ref{sys_1_aux}-\ref{tr_3_aux}) with initial data \eqref{eq:initial_data_aux}, where $u^-_0$ is replaced by $u_0^-{}'$, satisfies $\bu-\bu{}'\big|_{t=0}= (0,z_0)$ and $\p_t(\bu-\bu{}')\big|_{t=0}= (0,0)$.
By the energy conservation \eqref{aux_energy} below, $\bu-\bu'$ has vanishing energy for all time (since this is the case for $t=0$).
Therefore, $\p_t(\bu-\bu{}')\equiv 0$, implying that $\bu-\bu{}'= (0,z_0)$ for all time, thus constant (as before the pair stands for the $+$ and $-$ component of $\bu-\bu{}'$).

Now to produce a solution of the original system (\ref{sys_1}-\ref{tr_3})  given one of (\ref{sys_1_aux}-\ref{tr_3_aux}), set $w^+=\p_t u^+$, $v^- =\p_t u^-$ and $p^-=-\l_{\rm f} \div u^-$.
The solution of (\ref{sys_1}-\ref{tr_3}) obtained using those substitutions is independent of adding a pair $(0,z_0)$ to $\bu=(u^+,u^-)$, where $z_0$ is constant, divergence free and satisfies $z_0\cdot \nu \big|_{\G}=0$, since such a term does not alter $p^-$ or $v^-$.
Moreover, Assumption \ref{as:int} is satisfied automatically by the divergence theorem and \eqref{tr_1_aux}.
Conversely, given a solution to (\ref{sys_1}-\ref{tr_3}) with Assumption \ref{as:int} in effect for the initial data, one can produce a solution for (\ref{sys_1_aux}-\ref{tr_3_aux}) by taking $u^-=u_0^-+\int_0^t v^-(\t)d\t$ with $u_0^-$ chosen as described above (the solution is not unique due to the ambiguity in the choice of $u_0^-$).
To verify \eqref{sys_2_aux}, one needs to use that $\p_t^2u^-=\p_tv^-=-\r_{\rm f}^{-1}\n p^-$, and the fact  that $\p_tp^-=-\l_{\rm f} \div v^-$ implies  $\p_t(p^-+\l_{\rm f}\div u^-)=0$, to obtain $p^-=-\l_{\rm f}\div u^-$ for all time by \eqref{eq:displ_tr}.
Note that since the initial data are chosen so that $u_0^+\cdot\nu=u_0^-\cdot \nu$, \eqref{tr_1_aux} follows from \eqref{tr_1}.
In summary, solutions $(u^+,w^+,p^-,v^-)$ of (\ref{sys_1}-\ref{tr_3}) with initial data subject to Assumption \ref{as:int} are in 1-1 correspondence with solutions $\bu=(u^+,u^-)$ of (\ref{sys_1_aux}-\ref{tr_3_aux}) with initial data chosen as described before, modulo the addition of a pair $(0,z_0)$ with the aforementioned properties.

We would like to show that (\ref{sys_1_aux}-\ref{tr_3_aux}) has a unique solution given initial data \eqref{eq:initial_data_aux} lying in an appropriate space; consider the unbounded densely defined matrix operator $P_0$ on 
 \begin{equation}\label{eq:H^0}
 \calH^0:=L^2({M}^+,\r_{\rm s}dv_g;\C\otimes T{M})\times L^2({M}^-,\r_{\rm f}dv_g;\C\otimes TM),	
 \end{equation}
  given by 
\begin{equation}\label{aux_op}
	P_0=\begin{pmatrix}
		P^+ & 0\\ 
		0 & P^-
	\end{pmatrix}
	:=\begin{pmatrix}
		\r_{\rm s}^{-1}E & 0\\ 
		0 & \r_{\rm f}^{-1} \n \l_{\rm f} \div
	\end{pmatrix},
\end{equation}
with domain  
\begin{equation}\label{eq:D_P_0}
\begin{aligned}
	D(P_0)&=\big\{(u^+,u^-)\in C^\infty(\o{M}^+;\C\otimes T {M})\times C^\infty (\o{M}^-;\C\otimes T {M}) \\
	&\text{ with  } u^+\cdot \nu\big|_\Gamma=u^-\cdot \nu\big|_\Gamma,\quad  N(u^+)\big|_\Gamma=\l_{\rm f} (\div u^- )\;\nu\big|_{\Gamma}, \quad  N(u^+)=0\text { on } \p M\big\}.
\end{aligned}
\end{equation}
Note that $P^-=\r_{\rm f}^{-1}E$ with $E$ as in \eqref{elastic_wave_operator} but with the Lam\'e parameters given by  $\l=\l_{\rm f}$, $\mu=0$. The operator $P^-$ is {not} elliptic. 
In \eqref{eq:H^0} and \eqref{eq:D_P_0} we view $u^\pm$ as sections of the complexified tangent bundles, so that $\calH^0$ becomes a complex Hilbert space.
Once we have shown well posedness for the system (\ref{sys_1_aux}-\ref{tr_3_aux}) with initial data which are real vector fields in a subset of $\calH^0$, we will be able to produce real vector fields satisfying (\ref{sys_1_aux}-\ref{tr_3_aux}) by taking real parts of the a priori complex vector field valued solution associated with the given data.

\begin{remark}
One can equivalently let the measures in the $L^2$ spaces in \eqref{eq:H^0} be the Lebesgue measure on $\R^3$; the specific choice of measures is only relevant for the definition of the corresponding inner products.
Similarly, below we will use $L^2$ based Sobolev spaces, all of which are defined using smooth measures on precompact domains.
To avoid cluttering the notation, we will often not indicate the measure explicitly when writing them, but we will indicate the target space, when it is different from $\C$; for instance we will write $u^-\in L^2(M^-;\C\otimes TM)$ to mean that $\int_{M^-}|u^-|_g^2\r_{\rm f}\, dv_g<\infty$; we will also write $\| u^-\|^2_{L^2(M^-)} <\infty $ in this case.
On the other hand, when writing inner products we will specify the measure, e.g. we will write $(u^-,v^-)_{L^2(M^-,\r_{\rm f}dv_g)}:=\int_{M^-}u^-\cdot \o{v}^-\r_{\rm f} \, dv_g$.	
We henceforth assume everywhere that we have fixed Sobolev norms on $M^\pm$, $\G$ and $\p M$.
\end{remark}

Below we write
\begin{equation}
(\bu_1,\bu_2)_{L^2}=(\bu_1,\bu_2)_{\calH^0}
:=(u_1^+,u_2^+)_{L^2(M^+,\r_{\rm s}dv_g)}	+(u_1^-,u_2^-)_{L^2(M^-,\r_{\rm f}dv_g)},
\end{equation}
and 
$$\|\bu\|^2_{L^2}=\|\bu\|^2_{\calH^0}=\|u^+\|_{L^2(M^+,\r_{\rm s}dv_g)}^2+\|u^-\|_{L^2(M^-,\r_{\rm f}dv_g)}^2.$$
Using the identities
\begin{align}
	\int_{M^+} E u ^+\cdot \o{v}^+ dv_g-\int_{M^+}  u ^+ \cdot E\o{v}^+  dv_g=\int_{\p M^+} N( u ^+)\cdot \o{v}^+-  u ^+\cdot  N(\o{v}^+)
	 d A, %
\end{align}
and 
\begin{align}
	\int_{M^-}\n (\l_{\rm f} \div u ^-)\cdot \o{v}^-& dv_g-\int_{M^-}  u ^- \cdot \n (\l_{\rm f} \div\o{v}^- ) dv_g\\*
	&=\int_{\G}\left( -\l_{\rm f}(\div  u ^-)\o{v}^-\cdot \nu+\l_{\rm f}( u ^-\cdot \nu)\div(\o{v}^-)\right) d A, 
\end{align}
valid for $ u^\pm,v^\pm\in C^\infty (M^\pm;\C\otimes TM) $ (recall that $\nu$ is inward pointing for $M_-)$, one sees that $P_0$ is symmetric on $D(P_0)$. By a similar computation using the transmission conditions and the identity
\begin{equation}
  \int_{M^+} E u ^+\cdot \o{v}^+ dv_g=-\int_{M^+}\l_{\rm s} ( \div u^+ \div \o{v}^-)+2\mu_{\rm s} (d^{\rm s}u^+\cdot d^{\rm s} \o{v}^+)dv_g  +  \int_{\p M^+} N( u ^+)\cdot \o{v}^+dA,
\end{equation}
we find 
\begin{equation}
	( \bu,-P_0 \bu)_{L^2} \geq 0,\quad \bu\in D(P_0).
\end{equation}
By the Friedrichs extension construction (see e.g. \cite{MR1892228}),  $P_0$  can be extended to a self-adjoint operator $P$  with domain $D(P)$.
In Section \ref{ssec:domain} below we investigate $D(P)$ in more detail; this will be useful for showing well posedness for the system (\ref{sys_1_aux}-\ref{tr_3_aux}) and for justifying our parametrix.

\subsection{The domain of \texorpdfstring{$P$}{P}}\label{ssec:domain}

We briefly recall the Friedrichs construction, which produces a self-adjoint extension of $P_0$.
The first step in the construction of the domain $D(P)$  consists of  completing $D(P_0)$  with respect to the norm
$\| \bu\|_q^2=(-P_0\bu,\bu)_{\calH^0}+\|\bu\|^2_{\calH^0}$. 
This norm is induced by the positive definite quadratic form $q_0(\bu,\bw)=(-P_0\bu,\bw)_{\calH^0}+(\bu,\bw)_{\calH^0}$ with domain $D(P_0)$.
Then the completion of $D(P_0)$ in $\|\cdot \|_q$ is the domain of the closure of $q_0$; we denote this closure by $q$ and its domain by $D(q)$. 
We note that $D(q)$ can be identified with a subset of $\calH^0$ (the inclusion $\iota:(D(P_0),\|\cdot \|_{q})\to (\calH^0,\|\cdot \|_{\calH^0})$ extends to an injective bounded linear map $\hat{\iota}:(D(q),\|\cdot \|_{q})\to(\calH^0,\|\cdot \|_{\calH^0})$).
The operator 
$P_0$  subsequently extends to a bounded operator $P:D(q)\to D(q)^*$ 
by letting $(-P\bu,\bv)_{\calH^0}:=q(\bu,\bv)-(\bu,\bv)_{\calH^0}$, $\bu,\;\bv\in D(q)$. %
Then one takes
\begin{equation}\label{eq:domain_general_P}
	D(P)=\{\bw\in D(q): q(\bu,\bw)\leq C\|\bu\|_{\calH^0}\quad\text{for all } \bu\in D(q) \}.
	\end{equation}
Hence for $\bw\in D(P)$, $q(\cdot,\bw)$ extends to a bounded linear functional on $\calH^0$, and by the Riesz representation theorem there exists a unique $\td{\bw}\in \calH^0$ such that $q(\bu,\bw)=(\bu,\td{\bw})_{\calH^0}$; set $P\bw=-\td{\bw}+\bw$, which is a self-adjoint extension of $P_0$.

\medskip

We first identify the domain of the quadratic form $q$.
Below we set, for integer $k\geq 1$,
\begin{equation}
	H^k_{\div}(M^-;\C\otimes TM)=\{u^-\in L^2(M^-;\C\otimes TM):\div u^-\in H^{k-1}(M^-)\},
	\end{equation}
	where the divergence is with respect to $g$, and $\div u^-$ is a priori defined in a distributional sense.
If $u^-\in H^1_{\div}(M^-;\C\otimes TM)$,
then the trace  of the normal component of $u^-$ can be weakly defined as an element of $H^{-1/2}(\G)$ via
\begin{equation}
    -\lg\t(u^- \cdot\nu) ,\phi\rg _{L^2(\G,dA)}:=(\div u^-,
    \td{\phi})_{L^2(M^-,dv_g)}+( u^-,\n
    \td{\phi})_{L^2(M^-,dv_g)},\quad 
    {\phi}\in H^{1/2}(\p \G),
\end{equation}
where $\td{\phi} \in H^{1}(M^-)$ is an extension of $\phi$ off  $\G$ depending continuously on $\| \phi\|_{H^{1/2}(\G)}$ (it can be shown that the choice of extension does not affect the result).
Moreover, $\|\nu\cdot u^-\|_{H^{-1/2}(\G)}$ depends continuously on the norm $\|u^-\|_{H^1_{\div}(M^-)}:=(\|u^-\|_{L^2(M^-,dv_g)}^2+\|\div u^-\|_{L^2(M^-,dv_g)}^2)^{1/2}$.

\begin{lemma}\label{lm:h_1}
We have 
\begin{equation}\label{eq:H1div}
\bal
	D(q)&=\calH^1_{\div,\rm{tr}}\\
	:&=\{(u^+,u^-)\in H^1(M^+;\C\otimes TM)\times H^1_{\div}(M^-;\C\otimes TM): \t( u^+)\cdot \nu =
	\t(u^-\cdot \nu)\}.
\eal
\end{equation}
\end{lemma}

 The subscript ``$\rm{tr}$'' in \eqref{eq:H1div} stands for ``transmission''.
 The proof of Lemma \ref{lm:h_1} is contained in Appendix \ref{appendix_b}. We also include there
 the proofs of Proposition \ref{prop:estimate} and Corollary \ref{cor:estimate} below, which employ standard arguments used to show regularity estimates for the transmission problem for elliptic operators (see e.g. \cite{McLean-book}).

\begin{proposition}\label{prop:estimate}
Let $\bu\in D(P)$. 
We have the estimate
\begin{equation}\label{eq:prop_estimate}
\begin{aligned}
  &\|u^+\|^2_{H^2({M}^+)}+\|\div u^-\|^2_{H^1({M}^-)} \\
  &\qquad \leq C\Big(\|P^+u^+\|_{L^2(M^+)}^2+\|P^-u^-\|_{L^2(M^-)}^2 +\| u^+\|_{H^1(M^+)}^2+\| \div{u}^-\|_{L^2(M^-)}^2   \Big).
\end{aligned}
\end{equation}
\end{proposition}

By Proposition \ref{prop:estimate},
\begin{equation}\label{eq:transmission_domain}
\begin{aligned}
	&D(P)\subset \calH_{\div,\rm{tr}}^2:=\big\{(u^+,u^-)\in H^2(M^+;\C\otimes  TM)\times H^2_{\div}(M^-;\C\otimes  TM): \\
	&\qquad u^+\cdot \nu=u^-\cdot \nu  \text{ on } \Gamma, \quad
	  N(u^+)=\l_{\rm f} (\div u^- )\;\nu\text{ on }\Gamma,\quad 
	   N(u^+)=0\text{ on }\p M\big\}.
\end{aligned}
\end{equation}
The regularity follows directly from the estimate \eqref{eq:prop_estimate}, whereas the transmission conditions follow since $(P\bu,\bv)=(\bu,P\bv)$ for $\bu\in D(P)$, $\bv\in D(P_0)$.
Conversely, if $\bw\in \calH_{\div,\rm{tr}}^2$, an integration by parts and Cauchy Schwartz imply that $q(\bu,\bw)\leq C\|\bu\|_{L^2}$ for all $ \bu\in D(q)$. 
Thus we have:

\begin{proposition}\label{prop:domain}
The domain of the self-adjoint operator $P$ is given by
$D(P)=\calH_{\div,\rm{tr}}^2$.
\end{proposition}

The following corollary  will be useful in the justification of the parametrix.

\begin{corollary}\label{cor:estimate}
If $\bu\in D(P)$ with $P^\pm u^\pm \in H^k(M^\pm;\C\otimes TM)$, then for $k=0,1,2\dots $ we have
\begin{equation}\label{eq:cor_estimate}
\begin{aligned}
	&\|u^+\|^2_{H^{k+2}({M}^+)}+\|\div u^-\|^2_{H^{k+1}({M}^-)}, \\
	&\qquad \leq C\big(\|P^+u^+\|_{H^{k}(M^+)}^2+\|P^-u^-\|_{H^{k}(M^-)}^2 +\| u^+\|_{H^1(M^+)}^2+\| \div u^-\|_{L^2(M^-)}^2   \big).
\end{aligned}
\end{equation}
If $\bu=(u^+,u^-)\in D(P^m)$, $m\geq 1$ then $u^+\in H^{2m}(M^+;\C\otimes TM)$ and $u^-\in H_{\div}^{2m}(M^-;\C\otimes TM)$. 
\end{corollary}

\smallskip

\subsection{Well posedness}

Since $(-P\bu,\bu)\geq 0$ for $\bu\in D(P)$, there exists a unique non-negative self-adjoint square root of $-P$, written as $\sqrt{-P}$, and its domain is the completion of $D(P)$ in the graph norm, which implies that $D(\sqrt{-P})=D(q)=\calH^1_{\rm{div,tr}}$.
Moreover, by the functional calculus we can define the operators
$\cos(\sqrt{-P}\;t)$ and $\frac{\sin{(\sqrt{-P}\;t)}}{\sqrt{-P}}=t\text{sinc}(\sqrt{-P}\;t)$, 
which are strongly continuous in $t$ and satisfy
$\cos(\sqrt{-P}\;t)D(P)\subset D(P)$, $\frac{\sin{(\sqrt{-P}\;t)}}{\sqrt{-P}}D(\sqrt{-P})\subset D(P)$.
Now set 
\begin{equation}\label{aux_time}
	\bu(t)=\cos\left(\sqrt{-P}\;t\right)\bu_0+ \frac{\sin{(\sqrt{-P}\;t)}}{\sqrt{-P}}\bw_0\in D(P),
\end{equation}
where $\bu_0\in D(P)$ and $\bw_0\in D(\sqrt{-P})$, which solves \begin{equation}
	\p_t^2 \bu=P \bu,\quad \bu\big|_{t=0}=\bu_0,\quad \p_t\bu\big|_{t=0}=\bw_0
\end{equation}
subject to transmission  conditions satisfied by elements of $D(P)$, i.e. it solves (\ref{sys_1_aux}-\ref{tr_3_aux}).
We also have $\p_t\bu\in D(\sqrt{-P})\subset \calH^0$ for each $t$, thus
the energy
\begin{align}
	\calE(t)&=(\bu,-P \bu)_{L^2}+\|\p_t \bu\|^2_{L^2}, \\
			&=\|\Div(u^+)\|^2_{L^2(M^+,\l_{\rm s}dv_g)}+\|d^{\rm s}(u^+)\|^2_{L^2(M^+,2\mu_{\rm s} dv_g)}+\|\div u^-\|^2_{L^2(M^-,\l_{\rm f}dv_g)}+\|\p_t \bu\|^2_{L^2}\label{aux_energy}
\end{align}
is well defined for all time.
Moreover, since $\bu \in D(P)$ and $\p_t\bu\in D(\sqrt{-P})=\calH^1_{\rm{div,tr}}$, we can check that $\calE$ is constant upon differentiating $\calE$ in time and substituting $\p_t^2\bu=P\bu $ in $\calE'(t)$.

Returning to the original system (\ref{sys_1}-\ref{tr_3}) with the substitutions mentioned earlier, \eqref{aux_energy} implies that the energy
\begin{equation}\label{energy}
\begin{aligned}
		\|\Div(u^+)\|^2_{L^2(M^+,\l_{\rm s}dv_g)}+\|d^{\rm s}(u^+)\|^2_{L^2(M^+,2\mu_{\rm s} dv_g)}+\|w^+\|^2_{L^2(M^+,\r_{\rm s}dv_g)}&
\\
		+\|v^-\|^2_{L^2(M^-,\r_{\rm f} dv_g)}+\|p^-&\|^2_{L^2(M^-,\l_{\rm f}^{-1}dv_g)}
	\end{aligned}
\end{equation}
	is constant.
Now call $\calH$ the image of $D(P)\times D(\sqrt{-P})$ under the map
\begin{equation}\label{eq:map}
	D(P)\times D({\sqrt{-P}}\displaystyle)\ni ((u^+,u^-),(w^+,w^-))\mapsto(u^+,w^+,-\l_{\rm f} \div u^-,w^-)=  (u^+,w^+,p^-,v^-).
\end{equation}
Since $D(\sqrt{-P})=\calH^1_{\rm{div,tr}}$, by Proposition \ref{prop:domain} and Lemma \ref{lm:h_1} it follows that $\calH$ is contained in
\begin{align}
	  \Big\{ (u^+,w^+,p^-,v^-)\in H^2(M^+;\C\otimes TM)\times H^1(M^+;\C\otimes TM)\times H^1 (M^-)\times H^1_{\div}(M^-;\C\otimes TM):\quad\\*
	\t(w^+)\cdot \nu=\t(v^-\cdot\nu)\text{	and } N(u^+)=-p^- \nu \text{ on }\G,
	 \; N(u^+)=0\text{ on }\p M,\quad \label{eq:domain_init_system}\\*
	  \;\int_{M^-}p^- /\l_{\rm f}dv_g=\int_\G u^+\cdot \nu dA \Big\}.
\end{align}
It turns out that $\calH$ is actually equal to \eqref{eq:domain_init_system}:
given $(u^+,w^+,p^-,v^-)$ as  in \eqref{eq:domain_init_system}, one can produce $u^-\in H^2_{\div}(M^-;\C\otimes TM)$ such that $((u^+,u^-),(w^+,w^-))\in D(P)\times D({\sqrt{-P}}\displaystyle)$ is in its preimage under the map \eqref{eq:map} by solving $\Delta \w=-p/\l_{\rm f}$ on $M^-$, $\p_\nu \w=u^+\cdot \nu$ on $\G$ and taking $u^-=\n \w$ (just like in \eqref{eq:potential_neu}).
The last condition in \eqref{eq:domain_init_system} guarantees the solvability of this problem and the various transmission conditions are not hard to check.
Note that the map \eqref{eq:map} has non-trivial kernel; however,  any element in its kernel has 0 energy.
Therefore any solution of (\ref{sys_1_aux}-\ref{tr_3_aux}) produced by an element in the kernel as initial data via \eqref{aux_time} will be of constant 0 energy and will thus be constant, staying in the kernel for all times.
We have shown the following:
\begin{proposition}
	The system (\ref{sys_1}-\ref{tr_3}) subject to initial conditions  contained in the space $\calH$  given by \eqref{eq:domain_init_system} has a unique solution in $C(\R;\calH)$.
	The solution is given by the image of $(\bu,\p_t\bu)$ as in \eqref{aux_time} under the map \eqref{eq:map}, where the initial data $\bu_0$, $\bw_0$ are produced by given initial data in $\calH$ using the procedure described after \eqref{eq:domain_init_system}.
	The energy \eqref{energy} of such a solution is constant.
	\end{proposition}

\section{The Solid-Fluid Transmission System and Geometric Optics}\label{sec:solid_fluid}

\subsection{Transforming the system in the fluid region}
\label{sec:acoustic}

It will be convenient to work with a potential in the fluid region.
Let $(u^+,w^+,p^-,v^-)$ be the solution to (\ref{sys_1}-\ref{tr_3}) subject to initial data \eqref{eq_init}. 
To simplify the notations, henceforth we assume that the vector fields and functions we use are sections of appropriate regularity of $TM$ and $M\times \R$ respectively, instead of their complexified counterparts.
As mentioned earlier, the existence of real solutions of (\ref{sys_1_aux}-\ref{tr_3_aux}) is justified since we can find them by taking the real part of complex ones.
Then we can pass to real solutions of (\ref{sys_1}-\ref{tr_3}).
If in the initial data \eqref{eq_init} we have $\r_{\rm f} v_0^-\in H^1_{\Div}(M^-;TM)$ and $w_0\in H^1(M^+;TM)$, there exists a potential ${\psi}_0^-\in  H^2(M^-)$ and a divergence free vector field $ Z_0\in L^2(M^-;TM)$ with $\t(Z_0\cdot \nu)=0$ on $\G$ (in a weak sense) such that
\begin{equation}\label{helmholtz}
	\r_{\rm f} v^-_0= Z_0-\n {\psi}_0^-.
\end{equation}
They can be found by solving up to a constant
\begin{equation}\label{eq:neumenn_pot_density}
	\Delta {\psi}_0^- =-\div (\r_{\rm f} v_0^-)\text { on } M^-, \quad \p_\nu {\psi}_0^- =-\r_{\rm f} w^+_0\cdot \nu \text{ on }\G ,
\end{equation} 
and taking $  Z_0=\r_{\rm f}v_0^-+\n{\psi}_0^-$.
Here \eqref{eq:neumenn_pot_density} is solvable because of the transmission condition \eqref{tr_1} which guarantees that $\int_{M^-}-\div (\r_{\rm f} v_0^-)dv_g=\int_{\G}\r_{\rm f} w^+_0\cdot \nu \,dA$.
Setting 
\begin{equation}\label{eq:psi-pi}
	{\psi}^-(t)={\psi}_0^-+\int_0^tp^-(\t)d\t,
\end{equation}
we find that $\r_{\rm f}v^-+\n {\psi}^-$ is constant in time by \eqref{sys_3}, so by \eqref{helmholtz},
\begin{equation}
	\r_{\rm f} v^-(t)=  Z_0-\n {\psi}^-(t).\label{eq:decomposition_momentum_density}
\end{equation}
Now one checks that $(u^+,\psi^-)$ satisfies the following system of hyperbolic equations subject to transmission conditions and Cauchy data:
\begin{subequations}
\begin{align}
	\p_t^2u^+-\r_{\rm s}^{-1} Eu^+=0&\qquad \text{ in }M^+\times \R\label{pot_sys_1}, \\*
	\p_t^2{\psi}^--\l_{\rm f}\div(\r_{\rm f}^{-1}\n{\psi}^-)=F&\qquad \text{ in }M^-\times \R\noeqref{pot_sys_2}\label{pot_sys_2}, \\*
	\nu\cdot\p_tu^+=-\r_{\rm f}^{-1}\; \p_{\nu}{\psi}^-&\qquad \text{ on }\Gamma\times \R\label{pot_tr_1}, \\*
N(u^+)=-\p_t{\psi}^-\; \nu&\qquad \text{ on }\Gamma\times \R \label{pot_tr_2}\noeqref{pot_tr_2}, \\*
N(u^+)=0&\qquad \text{ on } \p M\times \R,\label{pot_tr_3}
\intertext{
with}
\label{pot_eq_init}
	(u^+,\p_tu^+,{\psi}^-,\p_t{{\psi}^-})\big|_{t=0}=&(u_0^+,w_0^+,{\psi}_0^-,p_0^-),
	\end{align}
\end{subequations}
where $F(x)=-\l_{\rm f}(\n\r_{\rm f}^{-1}) \cdot Z_0(x)$ is constant in time and Assumption \ref{as:int} applies to the initial data in \eqref{pot_eq_init}.
Note that by the construction of the initial potential ${\psi}_0^-$, $Z_0$ has no effect on the transmission conditions.
Moreover, the fact that $\psi_0^-$ is determined by $v_0^-$ up to constant is of no serious consequence.
Indeed, if $(u^+,\psi^-)$, $(u^+{}',\psi^-{}')$ are two solutions of (\ref{pot_sys_1}-\ref{pot_tr_3}) for which the initial data \eqref{pot_eq_init} differ by  $(0,0,c,0)$ for some constant $c$, then their difference $(\td{u}^+,\td{\psi}^-)$ satisfies a homogeneous version of (\ref{pot_sys_1}-\ref{pot_tr_3}), i.e. with $F=0$, for which the energy
\begin{align}
  \|\Div(\td{u}^+)\|^2_{L^2(M^+,\l_{\rm s}dv_g)}+\|d^{\rm s}(\td{u}^+)\|^2_{L^2(M^+,2\mu_{\rm s} dv_g)}+\|\p_t \td{u}^+\|^2_{L^2(M^+,\r_{\rm s} dv_g)}\qquad\\*
  +\|\p_t\td{\psi}^-\|^2_{L^2(M^-,\l_{\rm f}^{-1}dv_g)}+\|\n 
\td{\psi}^-\|^2_{L^2(M^-,\r_{\rm f}^{-1}dv_g)}
\end{align}
is constant in time. Therefore, $(\td{u}^+,\td{\psi}^-)=(0,c)$ for all time.

Conversely, we can produce a solution of (\ref{sys_1}-\ref{tr_3}) with initial conditions \eqref{eq_init}  by first decomposing $\r v_0^-$ using \eqref{helmholtz}, then solving (\ref{pot_sys_1}-\ref{pot_tr_3}) subject to \eqref{pot_eq_init}, and finally setting
\begin{equation}\label{going_back}
	(u^+,w^+,p^-,v^-)=(u^+,\p_tu^+,\p_t{\psi}^-,\r_{\rm f}^{-1}(  Z_0-\n {\psi}^-)).
\end{equation}
Note that the fact that $\psi_0^-$ is only determined up to constant does not affect \eqref{going_back}.

For the construction of our parametrix we will make the following assumption:
\begin{assumption}\label{as:cpct_sup}
  The initial data \eqref{eq_init} are supported away from $\G$ and $\p M$.
\end{assumption}

\noindent Note that this assumption implies that $\psi_0^-$ is smooth near the interface $\G$, by \eqref{eq:neumenn_pot_density} and elliptic regularity (notice that the Neumann condition becomes homogeneous under Assumption \ref{as:cpct_sup}).
Thus by \eqref{eq:decomposition_momentum_density}, $Z_0$ is also smooth near $\G$.
Now let $\chi \in C_c^\infty(M^-)$ be 1 in a neighborhood of the singular support of $Z_0$ and $\psi_0^-$.
With the techniques we use in Section \ref{ssec:justification} to justify the parametrix, it can be shown that the difference of a solution of (\ref{pot_sys_1}-\ref{pot_eq_init}) from one of the same system, but with $F$ replaced by $\chi F$ and $\psi_0^-$ replaced by $\chi \psi_0^-$, is smooth up to $\G$ and $\p M$.
Hence for the purposes of studying the transmission problem microlocally, \textit{it can be assumed without loss of generality that $\psi_0^-$ and $F$ are supported away from $\G$ (for all time in the case of the latter, since it does not depend on $t$), and we henceforth assume that this is the case}.
If the fluid is initially at rest (i.e. $v_0^-=0$) then $\psi_0^-=0$ and $Z_0=0$, hence $F$ vanishes, leading to \eqref{pot_sys_2} being homogeneous. 

The following justifies that, with Assumption \ref{as:cpct_sup} in effect, it suffices to use the system \eqref{pot_sys_1}-\eqref{pot_tr_3} to study the transmission problem microlocally:
 (also see the discussion below on traces at the interface $\G$):

\begin{lemma}\label{lm:wf}
In any conical neighborhood $U$ in $T^*(M^-\times \R)$ with  $U\cap\WF(Z_0)=\emptyset$ we have $\WF(\psi^-)=\WF(p^-)=\WF(v^-)$.
\end{lemma}
\begin{proof}
Note that $\WF(F)\subset \WF(Z_0)\subset \big\{(x,t,\xi,\t)\in T^*(M\times\R)\setminus 0:\t=0\big\}$.
By \eqref{pot_sys_2} and the analogous hyperbolic equation $\p_t^2{p}^--\l_{\rm f}\div(\r_{\rm f}^{-1}\n{p}^-)=0 $ for $p^-$ it follows that
\begin{equation}
  U\cap\big(\WF(p^-)\cup \WF(\psi^-)\big)\subset \Sigma_{\rm f}:= \big\{(x,t,\xi,\t)\in T^*(M^-\times\R)\setminus 0:c_{\rm f}^{-2}\t^2=|\xi|^2_g\big\},
\end{equation} where we set $c_{\rm f}=\sqrt{\l_{\rm f}/\r_{\rm f}}$ for the  speed of the fluid.
Since $p^-=\p_t \psi^-$ and $\p_t$ is elliptic in a conical neighborhood of $\Sigma_{\rm f}$, we conclude that in $U$ we have $\WF(p^-)= \WF(\psi^-)$.
By \eqref{eq:decomposition_momentum_density}, on $U$ we have $ \WF(v^-)=\WF(\n \psi^-)\subset \WF(\psi^-)$.
Since the metric is non-degenerate, in terms of local coordinates in a conical neighborhood of a covector $\zeta=(x,t,\xi,\t)\in \Sigma_{\rm f}$ at least one of the $\xi^k:=\sum_{j=1}^3g^{kj}\xi_j$, $k=1,2,3$,  is non-zero.
Assume that $\xi^3$ is non-zero (without loss of generality). Then 
write $\n \psi^-=A (0,0,\psi^-)^T$, where $A$ is a matrix differential operator with principal symbol $i\begin{pmatrix}
  1 & 0 &\xi^1\\
  0 & 1 &\xi^2\\
  0 & 0 &\xi^3
\end{pmatrix}$.
Then $A$ is elliptic in a conical neighborhood of $\zeta$, in the Douglis-Nirenberg sense, which shows that $\WF(\psi^-)=\WF(\n \psi^-)$, proving the claim.
\end{proof}

\smallskip

We will use the geometric optics construction to produce solutions up to a smooth error to the acoustic equation \eqref{pot_sys_2} in the fluid region near the interface $\G$.
We focus on what happens near the interface; away from it, the construction of a parametrix for \eqref{pot_sys_2} given initial data $\psi^-\big|_{t=0}$, $\p_t\psi ^-\big|_{t=0}$ can be carried out using geometric optics and Duhamel's formula (see e.g \cite{MR1269107}, \cite{MR618463}).
Choose local coordinates $(x',x_3,t)=(x_1,x_2,x_3,t)$ near a point  $(p_0,t_0)\in \G\times \R$
such that the interface $\G\times \R$ is given locally by $x_3=0$ and the unit normal is $\nu=-\p_{x_3}\big|_{\G}$ (so the solid region $M^+$ is locally given as  $x_3>0$).
 This can be done by using semigeodesic normal coordinates for $\Gamma$ centered at $p_0$.
We further assume that the metric is Euclidean at $p_0$ with our choice of spatial coordinates and below, we compute various symbols at that point; this simplifies the notation. 

Suppose we are given $f\in \calE'(\G\times \R)$ in a small neighborhood of $(p_0,t_0)$.
A covector $(x',t,\xi',\t)\in \WF(f)$ lies in one of the following three regions with respect to the acoustic speed of the fluid:
\begin{enumerate}
	\item Hyperbolic: $\{(x',t,\xi',\t)\in T^*(\G\times\R)\setminus 0:c_{\rm f}^{-2}\t^2-|\xi'|^2_g>0\}$,
	\item Elliptic: $\{(x',t,\xi',\t)\in T^*(\G\times\R)\setminus 0:c_{\rm f}^{-2}\t^2-|\xi'|^2_g<0\}$,
	\item Glancing: $\{(x',t,\xi',\t)\in T^*(\G\times\R)\setminus 0:c_{\rm f}^{-2}\t^2-|\xi'|^2_g=0\}$.
\end{enumerate}
In the three sets above $c_{\rm f}$ is always evaluated at $x=(x',0)$.
In everything below we assume that the wave front set of $f$ is disjoint from the glancing region.
As an intermediate step towards constructing a parametrix for the transmission problem, we seek approximate solutions up to smooth error for \eqref{pot_sys_2} in the fluid region  using $f$ as Dirichlet data on the interface; their form depends on whether $\WF(f)$ is contained in the hyperbolic or elliptic region.
We will also need to relate Dirichlet data with Neumann data at $\G\times \R$ using the Dirichlet to Neumann (DtN) map for an incoming or outgoing solution, see below.

\smallskip

First suppose that $\WF(f)$ is contained in the connected component of the \textbf{hyperbolic} region where $\t<0$ (in the $\t>0$ component the arguments are similar).
We extend $\l_{\rm f}$ and $\r_{\rm f}$ smoothly in a neighborhood of $p_0$ and use the geometric optics ansatz to produce an outgoing/incoming solution $\psi_{\Out/\In}^-$ which solves, in some neighborhood $\calU$ of $(p_0,t_0)$ in $M\times \R$,
\begin{align}
	\p_t^2\psi_{\Out/\In}^--\l_{\rm f}\div(\r_{\rm f}^{-1}\n\psi_{\Out/\In}^-)=&0 \quad  \text{ in } \calU\mod C^\infty\label{eq:psi_out_in}, \\
	\psi_{\Out/\In}^-
\big|_{\Gamma\times \R}=&f\quad \text{ in }\calU\cap (\G\times \R)\mod C^\infty,\label{eq:psi_out_in_b}\\
\psi_{\Out/\In}^-
\big|_{t\ll t_0/t\gg t_0}=&0\quad \text{ in } \calU\cap \o{M}^- \mod C^\infty.\label{eq:psi_out_in_c}\noeqref{eq:psi_out_in_c}
\end{align}
The defining property of the outgoing (resp. incoming) solution is that its singularities propagate to the future (resp. past) in the fluid region, along the null bicharacteristics of $c_{\rm f}^{-2}\t^2-|\xi|^2_g$.
Note that the equation \eqref{eq:psi_out_in} is taken to be homogeneous because we can assume without loss of generality that the source in \eqref{pot_sys_2} is supported outside of $\calU$, by Assumption \ref{as:cpct_sup}.
Taking a trace in \eqref{eq:psi_out_in_b} makes sense, since $\WF(\psi_{\Out/\In}^-)$ is disjoint from the conormal bundle of $\G\times \R$.
The outgoing/incoming parametrix has the form (see \cite{MR618463})
\begin{equation}\label{eq:go}
	\psi_{\Out/\In}^-=\frac{1}{(2\pi)^3}\int e^{i\varphi_{\Out/\In}^-(x',x_3,t,\xi',\t)}a_{\Out/\In}^-(x',x_3,t,\xi',\t)\widehat{f}(\xi',\t)d\xi'd\t,
\end{equation}
where the phase function solves an eikonal equation
   and satisfies
\begin{equation}\label{eq:phase}
	\varphi_{\Out/\In}^-(x',x_3,t,\xi',\t)=x'\cdot \xi'+t\t\mp x_3\sqrt{c_{\rm f}^{2}\t^2-|\xi'|^2_g}+O(x_3^2);
\end{equation}
the amplitude $a_{\Out/\In}^-$
solves a transport equation
with $a_{\Out/\In}^-\big|_{\Gamma\times \R}=1$ (the signs of the square root term in \eqref{eq:phase} are switched for the outgoing/incoming solution when $\WF(f)$ is contained in the component of the hyperbolic region where $\t > 0$).
In \eqref{eq:go}, $\widehat{\cdot}$ denotes the Fourier transform.

Differentiating $\psi_{\Out/\In}^-$ in the direction of $\nu=-\p_{x_3}$ we obtain the outgoing/incoming DtN map associated with our parametrix, which is defined by 
\begin{equation}
	\Lambda^-_{\Out/\In}f=\p_{\nu} \psi_{\Out/\In}^-\big|_{\Gamma\times \R}
\end{equation}
and is a pseudodifferential operator of order 1 on $\Gamma\times \R$ with principal symbol 
\begin{equation}
\s_{p_0}(\Lambda^-_{\Out/\In})=\pm i\sqrt{c_{\rm f}^{-2}\t^2-|\xi'|^2_g}.
\end{equation}

Now suppose that $f\in \calE'(\G\times \R)$ has wave front set in the \textbf{elliptic} region for the fluid, %
with $\t<0$. One can produce a parametrix for 
\begin{equation}\label{approx_pde}
\begin{aligned}
	\p_t^2\psi_\Ev^--\l_{\rm f}\div(\r_{\rm f}^{-1}\n\psi_\Ev^-)=&0 \quad  \text{ in } \calU\mod C^\infty\\
	\psi_\Ev^-
\big|_{\Gamma\times \R}=&f\quad \text{ in }\calU\cap (\G\times \R)\mod C^\infty
\end{aligned}
\end{equation}
 in the form \eqref{eq:go} but the phase function $\varphi^-_\Ev$ will now not be real.
To avoid exponentially growing waves we require that $\Im\varphi^-_\Ev\geq0$, which leads to evanescent waves. The phase function can be constructed asymptotically up to $O(x_3^\infty)$, having an expansion
\begin{equation}
	\varphi^-_\Ev(x',x_3,t,\xi',\t)\sim x'\cdot \xi'+t\t - x_3 i \sqrt{|\xi'|^2_g-c_{\rm f}^{-2}\t^2}+\sum_{j=0}^\infty x_3^{2+j}\td{\psi}_j(x',t,\xi',\t),
\end{equation}
where $\td{\psi}_j$ are symbols of order 1 (recall that we are interested in constructing an evanescent wave in the region $x_3\leq 0$, which dictates the negative sign in the square root term).
For more details on the construction see \cite{MR618463}, \cite{stefanov2020transmission}.
The corresponding microlocal DtN map 
\begin{equation}
	\Lambda_\Ev^-f=\p_{\nu}\psi^-_\Ev\big|_{\G\times \R}
\end{equation}
is a pseudodifferential operator on $\G\times \R$ with principal symbol
\begin{equation}
	\s_{p_0}(\Lambda_\Ev^-)=-\sqrt{|\xi'|^2_g-c_{\rm f}^{-2}\t^2}.\label{eq:ac_ell_dtn}
\end{equation}

\subsection{The elastic wave equation}
\label{sec:elastic}

On the solid side we follow \cite{stefanov2020transmission} to simplify the analysis.
A body wave in an isotropic elastic solid with constant Lam\'e parameters splits into a sum of a longitudinal wave ($\rm p$-wave) and a transversal one ($\rm s$-wave). 
The wave speed of the former is given by $c_{\rm p}=\sqrt{(\l_{\rm s}+2\mu_{\rm s})/\r_{\rm s}}$ whereas the one of the latter is given by $c_{\rm s}=\sqrt{\mu_{\rm s}/\r_{\rm s}}$. Note that $c_{\rm s}<c_{\rm p}$, since $\l_{\rm s}$, $\mu_{\rm s}>0$.
A p-wave (resp. s-wave) propagates singularities along the null bicharacteristics of $\t^2-c_{\rm p}^2|\xi|_g^2$ (resp. $\t^2-c_{\rm s}^2|\xi|_g^2$).
In our case the Lam\'e parameters and density are not constant,
however as shown in \cite{stefanov2020transmission}, in this setting one can decouple the system defined by the elastic wave equation up to smoothing operators.
By constructing a parametrix for the decoupled system one obtains  a microlocal splitting of elastic waves  into microlocal s- and p-waves at leading order, for which the statement on propagation of singularities still holds.
We let
\begin{align}
\Sigma_{\rm s}=\{(x,t,\xi,\t)\in T^*(M^+\times \R)\setminus 0:\t^2=c_{\rm s}^2|\xi|_g^2\},\\
	\Sigma_{\rm p}=\{(x,t,\xi,\t)\in T^*(M^+\times \R)\setminus 0:\t^2=c_{\rm p}^2|\xi|_g^2\};
\end{align}
if one uses local coordinates to identify a subset of $M^+$ with $\R^3$, $\Sigma_{\rm s/\rm p}$ can also be viewed as subsets of $T^*(\R^3\times \R)$, as we do in the statement of the following proposition.
Note that $\Sigma_{\rm p}\cap \Sigma_{\rm s}=\emptyset$.

\begin{proposition}[\cite{stefanov2020transmission}]\label{prop:decomp}
	Assume that local coordinates have been used to locally identify $M^+$ with $\R^3$.
	Let $u^+$ be a solution of the elastic wave equation 
	$\p_t^2u^+-\r_{\rm s}^{-1}Eu^+=0$ on an open set in  $\R^3\times  \R$  in the metric setting, and 
	 let $u^{\rm p}$ and $u^{\rm s}$  be microlocalizations of $u^+$ near $\Sigma_{\rm p}$, $\Sigma_{\rm s} $ respectively.
	 With respect to coordinates $(x,\xi,t,\t)\in T^*\R^3\times T^* \R$, in any conical set with $\xi_3\neq 0$ there exist
	 a scalar function $q^{\rm p}$ and a vector valued function $q^{\rm s}=(q^{\rm s}_1,q_2^{\rm s})$ such that microlocally $u^+=u^{\rm s}+u^{\rm p}$, where
	  \begin{equation}
		u^{\rm s}=(-i\curl +V_{\rm s})(q_1^{\rm s},q_2^{\rm s},0)^T ,\quad u^{\rm p}=-i \n  q^{\rm p}+V_{\rm p}(0,0, q^{\rm p})^T,\quad V_{\rm s}, V_{\rm p}\in \Psi^0 (\R^3),
	\end{equation}
and $q^{\rm s}, $ $q^{\rm p}$ satisfy
	 	\begin{align}
		\p_t^2 q^{\rm s}=(c_{\rm s}^2\Delta+A_{\rm s})q^{\rm s}+R_{\rm s}(q^{\rm s},q^{\rm p})^T\label{eq:principally_scalar_1}, \\*
		\p_t^2q^{\rm p}=(c_{\rm p}^2\Delta+A_{\rm p})q^{\rm p}+R_{\rm p}(q^{\rm s},q^{\rm p})^T\label{eq:principally_scalar_2}
	\end{align}
	with matrix valued $A_{\rm s}, A_{\rm p}\in \Psi^1(\R^3)$, $R_{\rm s}, R_{\rm p}\in \Psi^{-\infty}(\R^3)$.
	The $\curl$, gradient $\n$ and Laplace-Beltrami operator $\Delta$ are in the Riemannian sense and $\Delta$ is acting component-wise in \eqref{eq:principally_scalar_1}.
\end{proposition}

\noindent Note that the characteristic variety corresponding to \eqref{eq:principally_scalar_1} (resp. \eqref{eq:principally_scalar_2}) is $\Sigma_{\rm s}$ (resp. $\Sigma_{\rm p}$).

\smallskip

We use the same semigeodesic local coordinate setup introduced after Lemma \ref{lm:wf}. 
Moreover, extend smoothly the functions $\r_{\rm s}, $ $\mu_{\rm s}$ and $\l_{\rm s}$ near $p_0$ in order to make sense of a solution $u^+$ of  $\p_t^2u^+-\r_{\rm s}^{-1}Eu^+=0$ in an open set containing $(p_0,t_0)$.
Then by Proposition \ref{prop:decomp}, in a conic neighborhood of $T_{(p_0,t_0)}^*(M\times \R)$ where $\xi_3\neq 0$, we can write microlocally
 $u^+=u^{\rm s}+u^{\rm p}$ and
 \begin{equation}
		u^{\rm s}=U^+  (q_1^{\rm s},q_2^{\rm s},0)^T\quad u^{\rm p}=U^+(0,0,q^{\rm p})^T,\label{eq:decomp}
\end{equation}
where $U^+$ is an elliptic matrix valued pseudodifferential operator of order 1 with respect to the spatial variables.
Its principal symbol at a covector in $ T_{p_0}^*M$ is
\begin{equation}
	\s_{p_0}(U^+)=\begin{pmatrix}
		0 & -\xi_3 	&	\xi_1\\
		\xi_3 	& 	0 	& 	\xi_2\\
		-\xi_2 	& 	\xi_1 	& 	\xi_3 	
	\end{pmatrix}.
\end{equation}

\smallskip

Recall our assumption that the initial data in the solid region has support disjoint from $\p M^+$.
Then away from $\p M^+$ one can locally use geometric optics for principally scalar hyperbolic systems to construct a parametrix for the potentials $q^{\rm s}$, $q^{\rm p}$ with Cauchy data at $t=0$, see \cite{MR618463}, \cite{stefanov2020transmission}, yielding parametrices for the microlocal $s$ and $p$ waves by \eqref{eq:decomp}.
In the discussion below we focus on the transmission problem in a neighborhood of a point at the interface $\G\times \R$.
As before, we describe the construction of a parametrix to the  boundary value problem with Dirichlet data at the interface $\G\times \R$, as  an intermediate step towards constructing an parametrix for the transmission problem.
We also relate Neumann data at $\G\times \R$ to Dirichlet data using the DtN map for the incoming/outgoing parametrices.
The geometric optics construction near the outer boundary $\p M$ with homogeneous Neumann boundary condition can be done using the tools described in this section, and is discussed in detail in \cite[Section 8]{stefanov2020transmission}.

So suppose that we are given  Dirichlet data $f(x',t)\in \calE'(\Gamma \times \R;\R^3)$. %
Similarly to the acoustic case, the parametrix construction for the elastic equation depends on the location of the singularities of $f$.
A covector $(x',\xi',t,\t)\in \WF(f)$ can lie in one of the following regions:
\begin{enumerate}
	\item Hyperbolic: $\{(x',t,\xi',\t)\in T^*(\G\times\R)\setminus 0:\t^2>c_{\rm p}^{2}|\xi'|^2_g\}$,
	\item p-glancing: $\{(x',t,\xi',\t)\in T^*(\G\times\R)\setminus 0:\t^2=c_{\rm p}^{2}|\xi'|^2_g\}$,
	\item Mixed:  $\{(x',t,\xi',\t)\in T^*(\G\times\R)\setminus 0:c_{\rm p}^2|\xi'|^2_g>\t^2>c_{\rm s}^{2}|\xi'|^2_g\}$,
	\item s-glancing: $\{(x',t,\xi',\t)\in T^*(\G\times\R)\setminus 0:\t^2=c_{\rm s}^{2}|\xi'|^2_g\}$,
	\item Elliptic: $\{(x',t,\xi',\t)\in T^*(\G\times\R)\setminus 0:\t^2<c_{\rm s}^{2}|\xi'|^2_g\}$.
\end{enumerate}
We will always assume that our data have wave front set disjoint from the two glancing regions.

First assume that $f$ has wave front set in the component of the \textbf{hyperbolic} region where $\t<0$.
In a neighborhood $\calU$ of $(p_0,t_0)$ in $M$ we will use the geometric optics representation to construct approximate outgoing/incoming solutions to the elastic wave equation, i.e. one satisfying
\begin{equation}\label{eq:system_dir}
\begin{aligned}
	\p_t^2u_{\Out/\In}^+-\r_{\rm s}^{-1} Eu_{\Out/\In}^+=&0\mod C^\infty \text{ on }\calU,\\*
u_{\Out/\In}^+=&f\mod C^\infty \text { on }\calU\cap (\Gamma \times \R),\\
u_{\Out/\In}^+\big|_{t\ll t_0/t\gg t_0}=&0\mod C^\infty \text { on }\calU\cap\o{ M^+}.
\end{aligned}
\end{equation}
In \eqref{eq:system_dir},  outgoing (resp. incoming) means that the solution propagates singularities to the future (resp. past) in the solid region $x_3\geq 0$.
Since $\WF(f)$ is in the hyperbolic region and  $\WF(u^+_{\Out/\In})\subset \Sigma_{\rm p}\cup \Sigma _{\rm s} $, with respect to our coordinates we have $\xi_3\neq 0$ on $ \WF(u^+_{\Out/\In})$ sufficiently near $(p_0,t_0)$.
So as before we can write $u_{\Out/\In}^+=U^+(q_{\Out/\In}^{\rm s},q_{\Out/\In}^{\rm p})^T$.
Since $q^{\rm s}_{\Out/\In}$, $q^{\rm p}_{\Out/\In}$ satisfy \eqref{eq:principally_scalar_1}-\eqref{eq:principally_scalar_2}, it suffices to determine  data $q^{\rm s}_{j,\Out/\In ,b}$, $q^{\rm p}_{\Out/\In ,b}\in \calE'(\Gamma\times \R)$, $j=1,2$ in terms of $f$ at the interface and use them as  Dirichlet data for the geometric optics construction for a principally scalar acoustic system (see \cite[Section 3]{stefanov2020transmission}).
As shown in \cite{stefanov2020transmission}, there exist elliptic matrix pseudodifferential operators $U^+_{\Out/\In}$ on $\G\times \R$ which can be microlocally inverted to produce boundary values $q^{\rm s}_{\Out/\In,b}$, $q^{\rm p}_{\Out/\In,b}$ 
such that
\begin{equation}\label{eq:dir_data}
  f=U_{\Out/\In}^+(q^{\rm s}_{\Out/\In,b}, q^{\rm p}_{\Out/\In,b})^T.
\end{equation}
The subscript $b$ stands for ``boundary''.
The principal symbols of these operators take the form
\begin{equation}\label{eq:symbols_U}
	\s_{(p_0,t_0)}(U_{\In}^+)=\begin{pmatrix}
		0 & \xi_3^{\rm s} 	&	\xi_1\\
		-\xi_3^{\rm s} 	& 	0 	& 	\xi_2\\
		-\xi_2 	& 	\xi_1 	& 	-\xi_3^{\rm p} 	
	\end{pmatrix}, 
	\quad 
	\s_{(p_0,t_0)}(U_{\Out}^+)=\begin{pmatrix}
		0 & -\xi_3^{\rm s} 	&	\xi_1\\
		\xi_3^{\rm s} 	& 	0 	& 	\xi_2\\
		-\xi_2 	& 	\xi_1 	& 	\xi_3^{\rm p} 	
	\end{pmatrix}
\end{equation}
at the fiber of $T^*(\G\times \R)$ over $(p_0,t_0)$, where
\begin{equation}\label{eq:xis}
	\xi_3^{\rm s}=\sqrt{c_{\rm s}^{-2}\t^2-|\xi'|_g^2}, \qquad \xi_3^{\rm p}=\sqrt{c_{\rm p}^{-2}\t^2-|\xi'|_g^2}.
\end{equation} 
Then our solutions corresponding to potentials for $p$-waves will have the form
\begin{equation}\label{eq:go_solid}
	q^{\rm p}_{\Out/\In}=\frac{1}{(2\pi)^3}\int e^{i\varphi^{\rm p}_{\Out/\In}(x',x_3,t,\xi',\t)}a_{\Out/\In}^{\rm p}(x',x_3,t,\xi',\t)\widehat{q}_{\Out/\In,b}^{\rm p}(\xi',\t)d\xi'd\t,
\end{equation}
where the phase function solves an eikonal equation and satisfies
\begin{equation}
\varphi^{\rm p}_{\Out/\In}(x',x_3,t,\xi',\t)=x'\cdot \xi'+t\t\pm x_3\sqrt{c_{p}^{-2}\t^2-|\xi'|^2_g}+O(x_3^2),
\end{equation}
and the amplitude $a_{\Out/\In}^{\rm p}$ is a scalar valued classical symbol which solves a transport equation with $a_{\Out/\In}^{\rm p}(x',0,t,\xi',\t)=1$.
For $q^{\rm p}_{\Out/\In}$ the geometric optics solution is the same as \eqref{eq:go_solid} with all $\rm p$ superscripts replaced by $\rm s$ and with the difference that the amplitude $a_{\Out/\In}^{\rm s}$  has now values in $2\times 2$ matrix valued classical symbols with $a_{\Out/\In}^{\rm s}(x',0,t,\xi',\t)=Id$. 
One subsequently obtains the desired parametrix as $u_{\Out/\In}^+=u^{\rm p}_{\Out/\In}+u^{\rm s}_{\Out/\In}$, using \eqref{eq:decomp}.

From the solution $u_{\Out/\In}^+$ we obtain the traction across the interface $\G$, given by $N(u_{\Out/\In}^+)$.
Define matrix valued operators on $\G\times \R$ by
\begin{equation}\label{eq:dtn_elastic}
	\calM_{\In}^+(q^{\rm s}_{\In,b}, q^{\rm p}_{\In,b})^T=iN(u^+_{\In})\big|_{\G\times \R},\quad \calM_{\Out}^+(q^{\rm s}_{\Out,b}, q^{\rm p}_{\Out,b})^T=iN(u^+_{\Out})\big|_{\G\times \R},
\end{equation}
where $u_{\Out/\In}^+$ solves \eqref{eq:system_dir} with $f$ and $(q^{\rm s}_{\Out/\In,b   }$, $q^{\rm p}_{\Out/\In,b   })$ related by \eqref{eq:dir_data}.
The $i$ factors in \eqref{eq:dtn_elastic} are there to ensure that the principal symbol of $\calM_{\Out/\In}^+$ at $(p_0,t_0)\in\G\times \R$ is a matrix with real entries: it is shown in \cite{stefanov2020transmission} that with our choice of local coordinates
\begin{equation}\label{eq:symbols_M}
	\begin{aligned}
\s_{(p_0,t_0)}(\calM_{\In}^+)=\begin{pmatrix}
	-\mu_{\rm s}\xi_1\xi_2 & \mu_{\rm s} (2\xi_1^2+\xi_2^2)-\r_{\rm s}\t^2 	&	-2\mu_{\rm s}\xi_1\xi_3^{\rm p}\\
	-\mu_{\rm s} (\xi_1^2+2\xi_2^2)+\r_{\rm s}\t^2	&	\mu_{\rm s}\xi_1\xi_2	&	-2\mu\xi_2\xi_3^{\rm p}\\
	2\mu_{\rm s}\xi_2\xi_3^{\rm s} 	& 	-2\mu_{\rm s}\xi_1\xi_3^{\rm s} 	& 	-2\mu_{\rm s} (\xi_1^2+\xi_2^2)+\r_{\rm s}\t^2
\end{pmatrix},\\	
\s_{(p_0,t_0)}(\calM_{\Out}^+)=\begin{pmatrix}
	-\mu_{\rm s}\xi_1\xi_2 & \mu_{\rm s} (2\xi_1^2+\xi_2^2)-\r_{\rm s}\t^2 	&	2\mu_{\rm s}\xi_1\xi_3^{\rm p}\\
	-\mu_{\rm s} (\xi_1^2+2\xi_2^2)+\r_{\rm s}\t^2	&	\mu_{\rm s}\xi_1\xi_2	&	2\mu\xi_2\xi_3^{\rm p}\\
	-2\mu_{\rm s}\xi_2\xi_3^{\rm s} 	& 	2\mu_{\rm s}\xi_1\xi_3^{\rm s} 	& 	-2\mu_{\rm s} (\xi_1^2+\xi_2^2)+\r_{\rm s}\t^2
\end{pmatrix},
	\end{aligned}
\end{equation}
with $\xi_3^{\rm s}$, $\xi_3^{\rm p}$ given by \eqref{eq:xis}.

\smallskip

Assume now that  $f$ has wave front set in the \textbf{mixed} region, with $\t<0$.
Then the $p$-wave is evanescent; we construct  $q_{\Out/\In}^{\rm s}$ as before but now the potential for the $p$ wave will be evanescent, i.e. $q^{\rm p}_{\Ev}$ has complex valued phase function of the form
\begin{equation}
	\varphi_\Ev^{\rm p}(x',x_3,t,\xi',\t)\sim x'\cdot \xi'+t\t + x_3 i \sqrt{|\xi'|^2_g-c_{\rm p}^{-2}\t^2}+\sum_{j=0}^\infty x_3^{2+j}\td{\psi}_j(x',t,\xi',\t),
\end{equation}
where $\td{\psi}_j\in S^1$.
Note that the choice of the first order term in the expansion at $x_3=0$ is chosen so that $\Im \varphi_\Ev^{\rm p}\geq 0$ in ${M}^+$.
Depending on whether the solution with Dirichlet data $f$ has  outgoing or incoming shear waves, we  have $	f=U_{\Ev,\Out/\In}^+(q^{\rm s}_{\Out/\In,b}, q^{\rm p}_{\Ev,b})^T$, where 
\begin{equation}\label{eq:U_ev_out}
	\s_{(p_0,t_0)}(U_{\Ev,\Out}^+)=\begin{pmatrix}
		0 & -\xi_3^{\rm s} 	&	\xi_1\\
		\xi_3^{\rm s} 	& 	0 	& 	\xi_2\\
		-\xi_2 	& 	\xi_1 	& 	\td{\xi}_3^{\rm p} 	
	\end{pmatrix},\quad \td{\xi}_3^{\rm p}=i\sqrt{|\xi'|^2-c_{\rm p}^{-2}\t^2};
\end{equation}
for the principal symbol of $U_{\Ev,\In}^+$ replace  $\xi_3^{\rm s}$ by $-\xi_3^{\rm s}$ in \eqref{eq:U_ev_out} (with  $\td{\xi}_3^{\rm p}$ unchanged).
Moreover, let
\begin{equation}\label{eq:dtn_elastic_mixed}
	\calM_{\Ev,\Out/\In}^+(q^{\rm s}_{\Out/\In}, q^{\rm p}_{\Ev})^T=iN(u_{\Ev,\Out/\In}^+)\big|_{\G\times \R}.
\end{equation}
The principal symbol of $\calM_{\Ev,\Out}^+$ at a covector in $T^*_{(p_0,t_0)}(\G\times \R)$ is given by
\begin{equation}\label{eq:M_ev_out}
	\s_{(p_0,t_0)}(\calM_{\Ev,\Out}^+)=\begin{pmatrix}
	-\mu_{\rm s}\xi_1\xi_2 & \mu_{\rm s} (2\xi_1^2+\xi_2^2)-\r_{\rm s}\t^2 	&	2\mu_{\rm s}\xi_1\td{\xi}_3^{\rm p}\\
	-\mu_{\rm s} (\xi_1^2+2\xi_2^2)+\r_{\rm s}\t^2	&	\mu_{\rm s}\xi_1\xi_2	&	2\mu\xi_2\td{\xi}_3^{\rm p}\\
	-2\mu_{\rm s}\xi_2\xi_3^{\rm s} 	& 	2\mu_{\rm s}\xi_1\xi_3^{\rm s} 	& 	-2\mu_{\rm s} (\xi_1^2+\xi_2^2)+\r_{\rm s}\t^2
\end{pmatrix}.
\end{equation}
To obtain  $\s_{(p_0,t_0)}(\calM_{\Ev,\In}^+)$ again  one only needs to replace $\xi_3^{\rm s}$ by $-\xi_3^{\rm s}$ in \eqref{eq:M_ev_out}.

\smallskip

Finally suppose that $f\in \calE'(\G\times \R;\R^3)$ has wave front set in the \textbf{elliptic} region, 
 and with $\t<0$. 
Then we only have evanescent potentials $(q^{\rm s}_{\Ev}, q^{\rm p}_{\Ev})$ written as described in \eqref{eq:go_solid} and the subsequent discussion, with complex valued phase functions $\varphi_{\Ev}^{\rm p/s}$ having asymptotic expansions at $x_3=0$
\begin{equation}
	\varphi^{\rm p/s}_\Ev(x',x_3,t,\xi',\t)=x'\cdot \xi'+t\t+ix_3\sqrt{|\xi'|^2_g-c_{\rm p/s}^{-2}\t^2}+O(x_3^2).
\end{equation}
We can still produce a solution in the form $u_{\Ev}^+=u_{\Ev}^{\rm s}+u_{\Ev}^{\rm p}$, as before, with $u_{\Ev}^{\rm s}=U^+ (q^{\rm s}_{\Ev},0)$, $u_{\Ev}^{\rm p}=U^+(0,q^{\rm p}_{\Ev})$. 
We have an operator $U^+_\Ev$ on $\Gamma\times \R$ analogous to $U^+_{\Out/\In}$, with principal symbol 
\begin{equation}
	\s_{p_0}(U_\Ev^+)=\begin{pmatrix}
		0 & -\td{\xi}_3^{\rm s} 	&	\xi_1\\
		\td{\xi}_3^{\rm s} 	& 	0 	& 	\xi_2\\
		-\xi_2 	& 	\xi_1 	& 	\td{\xi}_3^{\rm p} 	
	\end{pmatrix},\quad \td{\xi}_3^{\rm s}=i\sqrt{|\xi'|^2-c_{\rm s}^{-2}\t^2}.
\end{equation}
This operator has the property $f=U_\Ev^+(q_{\Ev,b}^{\rm s},q_{\Ev,b}^{\rm p})$.
Moreover, writing 
\begin{equation}
	\calM_\Ev^+(q_{\Ev,b}^{\rm s},q_{\Ev,b}^{\rm p})^T=iN(u^+)\big|_{\G\times \R},
\end{equation}
one has
\begin{equation}\label{eq:ell_M}
	\s_{p_0}(\calM_\Ev^+)=\begin{pmatrix}
	-\mu_{\rm s}\xi_1\xi_2 & \mu_{\rm s} (2\xi_1^2+\xi_2^2)-\r_{\rm s}\t^2 	&	2\mu_{\rm s}\xi_1\td{\xi}_3^{\rm p}\\
	-\mu_{\rm s} (\xi_1^2+2\xi_2^2)+\r_{\rm s}\t^2	&	\mu_{\rm s}\xi_1\xi_2	&	2\mu\xi_2\td{\xi}_3^{\rm p}\\
	-2\mu_{\rm s}\xi_2\td{\xi}_3^{\rm s} 	& 	2\mu_{\rm s}\xi_1\td{\xi}_3^{\rm s} 	& 	-2\mu_{\rm s} (\xi_1^2+\xi_2^2)+\r_{\rm s}\t^2
\end{pmatrix}.
\end{equation}

\section{Microlocal well-posedness of the transmission problem}
\label{sec:well_posedness_microlocal}

In this section we study microlocally the transmission problem at the interface between a solid and fluid.
Given waves on the two sides of the interface in a neighborhood of a point $p_0\in \G$ and for time near a fixed $t_0$, their Dirichlet and Neumann data at $\G\times \R$ must match according to the transmission conditions; a covector in the wave front set of those data can lie in one of 15 possible regions, depending on whether it is in the hyperbolic/p-glancing/mixing/s-glancing/elliptic region for the solid and the hyperbolic/glancing/elliptic region for the fluid.
For instance, given an incoming microlocal p-wave $u_{\In}^{\rm p}$  in the solid,
the wave front set of its restriction to $\G\times \R$ will lie in the hyperbolic or glancing region for p-waves and in the hyperbolic one for  s-waves.
With respect to the acoustic speed in the fluid it can be in any of the elliptic, glancing or hyperbolic region, depending on the value of the acoustic speed in the fluid at the point of interest.
In this section we consider all possible cases for the location of wave front set of the boundary values of the various incoming, outgoing and evanescent waves, except the cases when the wave front set is contained in any of the glancing regions.

To construct the reflected and transmitted  waves generated by the arrival at $\G $ of  various combinations of incident p- or s-waves in the solid, or acoustic waves in the fluid, it suffices to determine Dirichlet data at $\G\times\R$ for their potentials, discussed in Sections \ref{sec:acoustic}-\ref{sec:elastic}.
For this purpose we use the transmission conditions at $\G$ and\ the microlocal DtN maps introduced in Sections \ref{sec:acoustic} and \ref{sec:elastic} and  set up systems for the principal amplitudes of the interface values of outgoing potentials; we then show that those systems can be solved in terms of the principal amplitudes of the incoming ones by proving ellipticity. Then they can be solved to any order as well. 
We also investigate the question of control, namely whether every configuration in the solid (resp. fluid) side can be produced by choosing appropriate waves in the fluid (respectively, solid) side. This is needed for the inverse problem.

Throughout this section we will work near a point $(p_0,t_0)\in \G\times \R$, with semigeodesic local coordinates chosen as described in the paragraph following Lemma \ref{lm:wf}. 
Our full local coordinate system $(x',x_3,t)$ induces local coordinates $(x',x_3,t,\xi',\xi_3,\t)$ on $T^*(M\times \R)\cong T^*M\times T^* \R$, and $(x',t,\xi',\t)$ are coordinates on $T^*(\G\times \R)$. 

\subsection{The hyperbolic-hyperbolic case (Figure~\ref{fig:hyp_hyp})}\label{ssec:hh}

Suppose that we have incoming body waves in the solid and in the fluid.
We may assume that the elastic wave in the solid side and the  potential in the fluid side solve \eqref{pot_sys_1} and \eqref{pot_sys_2} respectively in a neighborhood of a point $(p_0,t_0)\in \G\times \R$ in $M\times \R$, by extending $\mu_{\rm s}$, $\l_{\rm s}$, $\r_{\rm s}$, $\l_{\rm f}$ and $\r_{\rm f}$ smoothly near $p_0$ from their respective initial domains of definition.
We first consider the case where the wave front sets of the traces of all waves are contained in the hyperbolic region for all three speeds $c_{\rm s}$, $c_{\rm p}$, $c_{\rm f}$, with $\t<0$ (the case $\t>0$ is similar);
using a microlocal partition of unity if necessary, it suffices to assume that they
 are contained in a small conical neighborhood $\td{\Sigma}$ of a covector 
 $\big((p_0,t_0),(\xi_0',\t_0)\big)\in T^*(\G\times \R)\setminus 0$.
We recall from Section \ref{sec:elastic} that the boundary trace of the incoming wave $u^+_{\In}$ can be written as $u^+_{\In,b}=U_{\In}^+q_{\In,b}^+$, where the principal symbol of $U_{\In}^+$ at $(p_0,t_0)$ is  as in \eqref{eq:symbols_U} and $q_{\In,b}^+=(q^{\rm s}_{\In,b},q^{\rm p}_{\In,b}) ^T$. 
Moreover, the boundary value of the traction at the interface is given by $N(u^+_{\In})\big|_{\G\times \R}=-i \calM_{\In}^+q_{\In,b}$, where the principal symbol of $\calM_{\In}^+$ is given by \eqref{eq:symbols_M}.
The discussion regarding outgoing waves in the solid region is similar, except the subscripts are now replaced by ``$\Out$''.
On the fluid side, the normal derivative of $\psi^-_{\In}$ is $\p_\nu\psi_{\In}^-\big|_{\G\times \R}=\Lambda_{\In}^-\psi_{\In,b}^-$, and similarly for $\psi^-_{\Out}$.
Hence by the transmission conditions  \eqref{pot_tr_1}-\eqref{pot_tr_2}
\begin{subequations}
\begin{align}
{\nu} \cdot \p_t(U_{\In}^+ q_{\In,b}^++U_{\Out}^+ q_{\Out,b}^+)=&-\r_{\rm f}^{-1}\Lambda_{\In}^-\psi^-_{\In,b}-\r_{\rm f}^{-1}\Lambda_{\Out}^-\psi^-_{\Out,b}	\label{system_first}	\\*
-i(\calM_{\In}^+ q_{\In,b}^++\calM_{\Out}^+ q_{\Out,b}^+)=&-\p_t \psi^-_{\In,b}{\nu}-\p_t \psi^-_{\Out,b}\; {\nu}\label{system_second}.
\end{align}
\end{subequations}

Rewrite \eqref{system_first}-\eqref{system_second} as a system for the traces of the outgoing solutions:
\begin{equation}\label{rewritten_system}
	\qquad A_{\Out}^{\rm{hh}}\begin{pmatrix}
		q_{\Out,b}^+\\
		\psi_{\Out,b}^-
	\end{pmatrix}=
	  A_{\In}^{\rm{hh}}
	  \begin{pmatrix}
	  	 q_{\In,b}^+\\
	  	 \psi_{\In,b}^-
	  \end{pmatrix},\\
\end{equation}
where
\begin{equation}
A_{\Out}^{\rm{hh}} :=  \begin{pmatrix}
	 \p_t ({\nu} \cdot U_{\Out}^+) & 	\r_{\rm f}^{-1}\Lambda_{\Out}^-\\
	 -i \calM_{\Out}^+				& 	 {\nu}\; \p_t  
	 \end{pmatrix}, \quad 
	 A_{\In}^{\rm{hh}}:=\begin{pmatrix}
	  	-\p_t ( {\nu} \cdot U_{\In}^+) &  -\r_{\rm f}^{-1}\Lambda_{\In}^-\\
		i \calM_{\In}^+ 					& -\nu \; \p_t
	  \end{pmatrix},
\end{equation}
and the superscripts stand for hyperbolic-hyperbolic. 
We would like to show that the system \eqref{rewritten_system} is solvable microlocally, i.e. that the matrix operator $A_{\Out}^{\rm{hh}}$ on the left hand side is elliptic.
Since the matrix operators in the first column of $A_{\Out}^{\rm{hh}}$ are of order 2, whereas the ones on the right column are of order 1,
the homogeneous principal symbol of degree 2 of the operator is not invertible.
However we can seek ellipticity in the Douglis-Nirenberg sense (\cite{MR75417}), which in this case amounts to computing the matrix whose entries are the principal symbols of the individual operators appearing as entries in \eqref{rewritten_system}, and checking that its  determinant  is non-zero for $(\xi',\t)\neq 0$ in the hyperbolic-hyperbolic region.

By  \eqref{eq:symbols_U} and the fact that $\nu=-\p_{x_3}$ in terms of our local coordinates, we have
\begin{equation} 
	\s_{p_0}(\nu \cdot U_{\Out}^+)=\begin{pmatrix}
		\xi_2 & -\xi_1 &-\xi_3^{\rm p}
	\end{pmatrix},\quad 
	\s_{p_0}(\nu \cdot U_{\In}^+)=\begin{pmatrix}
		\xi_2 & -\xi_1 &\xi_3^{\rm p}
	\end{pmatrix}.
\end{equation}
By the  invariance of the principal symbols of $U^+_{\In/\Out}$, $\calM^+_{\In/\Out}$ and $\Lambda_{\In/\Out}^-$  under rotations in the $\xi_1$-$\xi_2$ plane observed in \cite[Section 7.2]{stefanov2020transmission} (this uses the specific choice of local coordinates made so that $g$ is Euclidean at $p_0$),
  the problem of showing the requisite ellipticity at $(\xi',\t)\in T^*_{(p_0,t_0)}(\G\times\R)$ reduces to showing it under the assumption $\xi_2=0$.
 Compute the principal symbols $\td{\s}_{p_0}(A_{\Out/\In}^{\rm{hh}})$ of $A_{\Out/\In}^{\rm{hh}}$, in the Douglis-Nirenberg sense described before, with $\xi_2=0$:
\begin{subequations}
\begin{align}\label{A_out}
	\td{\s}_{(p_0,t_0)}(A_{\Out}^{\rm{hh}})=\begin{pmatrix}
		0 	& 	-i\t \xi_1 	& 	-i\t \xi_3^{\rm p} & 	i\r_{\rm f}^{-1}\xi_3^{\rm f}\\
		0 	& 	-i(2\mu_{\rm s}\xi_1^2-\r_{\rm s}\t^2) 		& 	-2i\mu_{\rm s}\xi_1\xi_3^{\rm p} & 0 \\
		i(\mu_{\rm s}\xi_1^2-\r_{\rm s}\t^2) 	& 		0 & 0 & 0 \\
		0 & -2i\mu_{\rm s}\xi_1\xi_3^{\rm s} & 	i(2\mu_{\rm s}\xi_1^2-\r_{\rm s}\t^2) & -i\t
	\end{pmatrix},\\ %
	 \quad \td{\s}_{(p_0,t_0)}(A_{\In}^{\rm{hh}})
	 = \begin{pmatrix}
		0 	& 	i\t \xi_1 	& 	-i\t \xi_3^{\rm p} & 	i\r_{\rm f}^{-1}\xi_3^{\rm f}\\ 
		0 	& 	i(2\mu_{\rm s}\xi_1^2-\r_{\rm s}\t^2) 		& 	-2i\mu_{\rm s}\xi_1\xi_3^{\rm p} & 0 \\
		-i(\mu_{\rm s}\xi_1^2-\r_{\rm s}\t^2) 	& 		0 & 0 & 0 \\
		0 & -2i\mu_{\rm s}\xi_1\xi_3^{\rm s} & 	-i(2\mu_{\rm s}\xi_1^2-\r_{\rm s}\t^2) & i\t
	\end{pmatrix},\label{eq:A_in}%
\end{align}
\end{subequations}
where $\xi_3^\bullet =\sqrt{c_\bullet^{-2}\t^2-|\xi'|^2_g}$ for $\bullet =p,s, f$,  evaluated at $\xi'=(\xi_1,0)$.

Using \eqref{A_out}-\eqref{eq:A_in}, we can rewrite \eqref{system_first}-\eqref{system_second} at the principal symbol level as a system for the boundary values of the amplitudes of $q_{\Out}^+$, $\psi_{\In}^-$.
We write (with $\calF$ the Fourier transform)
\begin{subequations}
\begin{align}
\big(b^{\rm s}_{1,\Out/\In}(\xi',\t),b^{\rm s}_{2,\Out/\In}(\xi',\t),b^{\rm p}_{\Out/\In}(\xi',\t)\big)=\calF_{(x',t)}({q}^+_{\Out/\In,b})(\xi',\t), \label{amplitude1}\\*
 	b_{\Out/\In}^{\rm f}(\xi',\t)=\calF_{(x',t)}({\psi}^-_{\Out/\In,b})(\xi',\t ), \label{amplitude2}
\end{align}
\end{subequations}
 and we seek to determine $\big(b^{\rm s}_{1,\Out},b^{\rm s}_{2,\Out},b^{\rm p}_{\Out}\big)$ and $b_{\Out}^{\rm f}$
 given $\big(b^{\rm s}_{1,\In},b^{\rm s}_{2,\In},b^{\rm p}_{\In}\big)$ and $b_{\In}^{\rm f}$.
 Once this has been done, can construct parametrices for $q_{\Out}^+$, $\psi_{\Out}^-$ using the geometric optics ansatz.

We remark here that in the case where the direction of propagation of the wave $u^+_{\Out/\In}$ is given by $(\xi_1,0,\xi_3)$ (i.e. $\xi_2=0$) and the metric is taken to be Euclidean at $p_0$, the amplitudes $b_{1,\In/\Out}^{\rm s}$ correspond to  microlocal shear horizontal (SH) waves at $\G$, in the sense that the corresponding wave $u^{\rm sh}_{\Out/\In}=-i \curl (q_{1,\Out/\In}^{\rm s},0,0)$ is tangent to the interface $\G$ at $p_0$, up to lower order terms.
On the other hand, the amplitudes $b_{2,\In/\Out}^{\rm s}$ correspond to microlocal shear vertical (SV) waves at $\G$ in the sense that the corresponding wave $u^{\rm sv}_{\Out/\In}=-i \curl (0,q_{2,\Out/\In}^{\rm s},0)$ satisfies $(\curl u^{\rm sv}_{\Out/\In})\cdot \nu=0 $ at $\G$.
In our case where the Lam\'e parameters are non-constant, the decomposition into shear horizontal and shear vertical waves only makes sense at $\G$; for details see \cite[Section 7.2]{stefanov2020transmission}.

It now follows from \eqref{A_out}-\eqref{eq:A_in} that the system 
for the outgoing amplitudes at the principal symbol level decouples into the following two systems:
\begin{subequations}
\begin{align}
	 \begin{split}
			&\begin{pmatrix}
				\t\xi_1 					& 	\t\xi_3^{\rm p} 					& -\r_{\rm f}^{-1}\xi_3^{\rm f}\,\\
				2\mu_{\rm s} \xi_1^2-\r_{\rm s}\t^2 	&	2\mu_{\rm s}\xi_1\xi_3^{\rm p}			& 0\\
				2\mu_{\rm s}\xi_1\xi_3^{\rm s} 			& 	-2\mu_{\rm s} \xi_1^2+\r_{\rm s}\t^2	& \t 
			\end{pmatrix}
			\begin{pmatrix}
				b_{2,\Out}^{\rm s}\\ b_{\Out}^{\rm p}\\ {b}^{\rm f}_{\Out}
			\end{pmatrix} \\
			&\qquad =\begin{pmatrix}
				-\t\xi_1 	& 	\t\xi_3^{\rm p} 				&	-\r_{\rm f}^{-1}\xi_3^{\rm f}\, \\
				-2\mu_{\rm s} \xi_1^2+\r_{\rm s}\t^2 	&	2\mu_{\rm s}\xi_1\xi_3^{\rm p}			&	0 \\
				2\mu_{\rm s}\xi_1\xi_3^{\rm s} 	& 	2\mu_{\rm s} \xi_1^2-\r_{\rm s}\t^2			&	{}-\t 
			\end{pmatrix}
			\begin{pmatrix}
				b_{2,\In}^{\rm s}\\ b_{\In}^{\rm p}\\	{b}^{\rm f}_{\In}
			\end{pmatrix},
		\end{split}\label{matrix_system}\\* %
\intertext{ and }\\*
	&(-\mu_{\rm s} \xi_1^2+\r_{\rm s}\t^2)(b_{1,\In}^{\rm s}+b_{1,\Out}^{\rm s})=0.\label{full_reflection}
	\end{align}
\end{subequations}

The determinant of the $3\times 3$ matrix on the left hand side of \eqref{matrix_system} is given by 
\begin{equation}
	 \left(\t^4\r_{\rm s}\xi_3^{\rm p}+\r_{\rm f}^{-1}\xi_3^{\rm f}\left((2\mu_{\rm s}\xi_1^2-\r_{\rm s}\t^2)^2+4\mu_{\rm s}^2\xi_1^2\xi_3^{\rm s}\xi_3^{\rm p}\right)\right)\neq0 \label{determinant}
\end{equation}
for $(\xi',\t)=(\xi_1,0,\t)\neq 0$.
Thus \eqref{matrix_system}-\eqref{full_reflection} is solvable for the outgoing amplitudes.
Moreover, it follows  from \eqref{full_reflection} that  the microlocal shear horizontal waves are totally reflected. 
Notice that this total internal reflection of the microlocal SH waves takes place without creation of evanescent waves on the fluid side (unlike the case of total internal reflection of acoustic waves meeting an interface between two fluids in the hyperbolic-elliptic region, see e.g \cite[\sectionsymbol 3.3.2]{stefanov2020transmission}, where evanescent waves are created on the other side).
This can be explained by the transmission condition: the kinematic transmission condition \eqref{pot_tr_1} imposes no restriction on the SH waves at the interface (at the principal symbol level), since they are tangent to it.
Moreover, the dynamic transmission condition \eqref{pot_tr_2} forces the tangential components of the traction at $\G$ to vanish, which, as one can check, implies $\calF\big(N(u_{\In}^{\rm sh}+u_{\Out}^{\rm sh})\big)(\xi_1,0,\t)=0$ modulo lower order terms for $u^{\rm sh}_{\Out/\In}=-i \curl (q_{1,\Out/\In}^{\rm s},0,0)$, which is equivalent to \eqref{full_reflection}.
In other words, at the leading order the interface behaves like a ``hard boundary'' with respect to the SH waves, i.e. like an interface between the solid and vacuum; with that observation, the full reflection of the SH waves without transmission of singularities to the fluid side is to be expected, as shown e.g. in \cite[\sectionsymbol 8]{stefanov2020transmission}. 

\begin{figure}[h]
	\includegraphics[page = 1, scale=1]{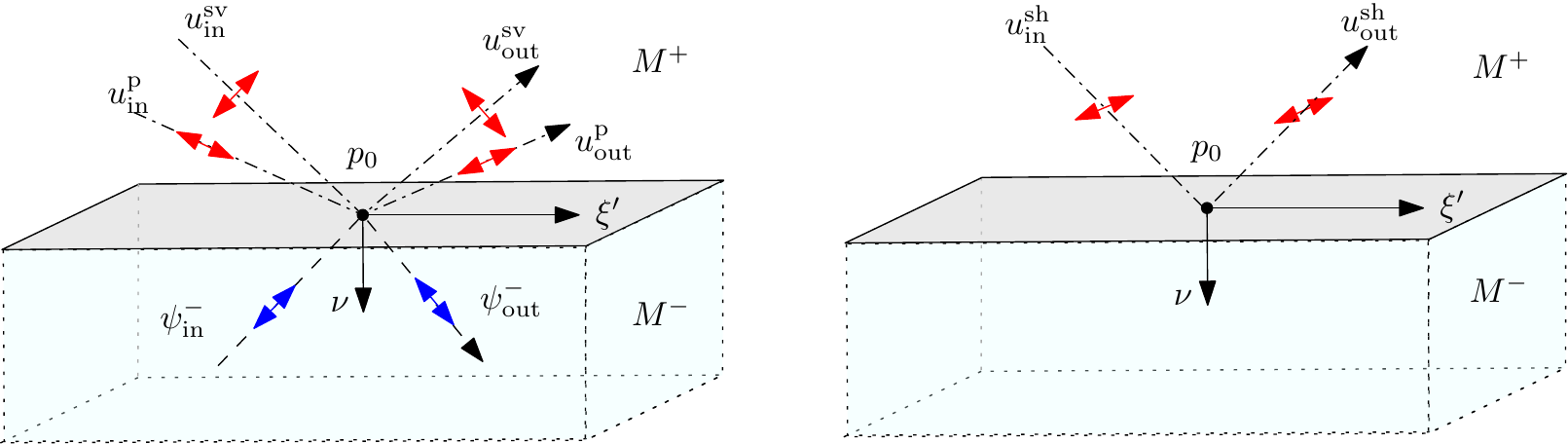}
	\caption{The hyperbolic-hyperbolic transmission system. The solid occupies the top region $M^+$, whereas the fluid occupies the bottom one, $M^-$. 
    On the left hand side we see the $p$ waves and the microlocal shear vertical (SV) waves, as well as the acoustic waves in the fluid. The microlocal shear horizontal (SH) waves are totally reflected and are pictured separately, on the right. }
	\label{fig:hyp_hyp}
\end{figure}

Given a solution of the acoustic equation on the fluid side whose Cauchy data at $\G\times \R$ has wave front set in the hyperbolic-hyperbolic region, one can choose suitable waves on the solid side to produce them:
finding appropriate amplitudes at the boundary for the incoming and outgoing waves in the solid reduces to solvability of the system
\begin{equation}
	\begin{aligned}\label{eq:control_hh}
	\begin{pmatrix}
		 	\t \xi_1 	& 	\t \xi_3^{\rm p} & 		\t \xi_1 	& 	-\t \xi_3^{\rm p}\\
		 	2\mu_{\rm s}\xi_1^2-\r_{\rm s}\t^2 		& 	2\mu_{\rm s}\xi_1\xi_3^{\rm p} & 2\mu_{\rm s}\xi_1^2-\r_{\rm s}\t^2 		& 	-2\mu_{\rm s}\xi_1\xi_3^{\rm p}  
		\\
		 2\mu_{\rm s}\xi_1\xi_3^{\rm s} & 	-2\mu_{\rm s}\xi_1^2+\r_{\rm s}\t^2 & -2\mu_{\rm s}\xi_1\xi_3^{\rm s} & 	-2\mu_{\rm s}\xi_1^2+\r_{\rm s}\t^2
	\end{pmatrix}
	\begin{pmatrix}
		b^{\rm s}_{2,\Out} \\
		b^{\rm p}_{\Out} \\
		b^{\rm s}_{2,\In} \\
		b^{\rm p}_{\In}
	\end{pmatrix}\qquad\\
	=\begin{pmatrix}
		-\r_{\rm f}^{-1}\xi_3^{\rm f} &  \r_{\rm f}^{-1}\xi_3^{\rm f}\\
		0 & 0\\ 
		 -\t & 			-\t
	\end{pmatrix}\begin{pmatrix}
		b_{\In}^{\rm f}\\
		b_{\Out}^{\rm f}
	\end{pmatrix},
\end{aligned}
\end{equation}
which is  underdetermined as a system for $(b^{\rm s}_{2,\Out},
		b^{\rm p}_{\Out},
		b^{\rm s}_{2,\In},
		b^{\rm p}_{\In})^T$.
This can be seen by row reduction (recall that $\t\neq 0$ in the hyperbolic-hyperbolic region).

On the other hand, we generally cannot control the solid side from the fluid one. 
Eq. \eqref{full_reflection} implies that microlocal shear horizontal waves in the solid side are structured and independent of the waves in the fluid one.
Arbitrary shear vertical and pressure waves in the solid side also cannot be created by an appropriate choice of waves in the fluid side: to do so we would have to solve \eqref{eq:control_hh} for $ (b_{\In}^{\rm f},		b_{\Out}^{\rm f})$, and this system is overdetermined; for instance it is solvable when
\begin{equation}
	\begin{pmatrix}
		2\mu_{\rm s}\xi_1^2-\r_{\rm s}\t^2 		& 	2\mu_{\rm s}\xi_1\xi_3^{\rm p} & 2\mu_{\rm s}\xi_1^2-\r_{\rm s}\t^2 		& 	-2\mu_{\rm s}\xi_1\xi_3^{\rm p}  
	\end{pmatrix}
	\begin{pmatrix}
		b^{\rm s}_{2,\Out}&
		b^{\rm p}_{\Out}&
		b^{\rm s}_{2,\In}&
		b^{\rm p}_{\In}
	\end{pmatrix}^T=0
\end{equation}
in the small conical neighborhood $\td{\Sigma}$ of interest containing $\WF(u^+)$.

\subsection{The mixed-hyperbolic region (Figure~\ref{fig:mix-hyp})} \label{ssec:mh}

This case can happen only  if $c_{\rm f}(p_0)<c_{\rm p}(p_0)$.
We have incoming $\rm s$-waves in the solid and acoustic waves in the fluid, but no $\rm p$-waves in the solid; we seek the latter as evanescent waves.
The transmission conditions \eqref{pot_tr_1}, \eqref{pot_tr_2} yield 
\begin{align}
{\nu} \cdot \p_t(U_{\Ev,\In}^+ q_{\Ev,\In,b}^++U_{\Ev,\Out}^+ q_{\Ev,\Out,b}^+)=&-\r_{\rm f}^{-1}\Lambda_{\In}^-\psi_{\In,b}^--\r_{\rm f}^{-1}\Lambda_{\Out}^-\psi^-_{\Out,b}	\label{mix_hyp_system_first},	\\*
-i(\calM_{\Ev,\In}^+ q_{\Ev,\In,b}^++\calM_{\Ev,\Out}^+ q_{\Ev,\Out,b}^+)=&-\p_t \psi^-_{\In,b}{\nu}-\p_t \psi^-_{\Out,b}\; {\nu}\label{mix_hyp_system_second}.
\end{align}
As in \eqref{amplitude1}-\eqref{amplitude2}, let $(b^{\rm s}_{1,\Out/\In},b^{\rm s}_{2,\Out/\In},b^{\rm p}_{\Ev})=\calF\big(q_{\Ev,\Out/\In,b}^+\big)$.
We also let  $b_{\Out/\In}^{\rm f}=\calF\big(\psi^-_{\Out/\In,b}\big)$.
Again we wish to solve a system of the form \eqref{rewritten_system}, where now $A_{\Out/\In}^{\rm{hh}}$ are replaced by 
\begin{equation}
A_{\Out} ^{\rm{mh}}:=  \begin{pmatrix}
	 \p_t ({\nu} \cdot U_{\Ev,\Out}^+) & 	\r_{\rm f}^{-1}\Lambda_{\Out}^-\\
	 -i \calM_{\Ev,\Out}^+				& 	 {\nu}\; \p_t  
	 \end{pmatrix}, \quad 
	 A_{\In}^{\rm{mh}}:=\begin{pmatrix}
	  	-\p_t ( {\nu} \cdot U_{\Ev,\In}^+) &  -\r_{\rm f}^{-1}\Lambda_{\In}^-\\
		i \calM_{\Ev,\In}^+ 					& -\nu \; \p_t
	  \end{pmatrix}.
\end{equation}
The principal symbol of the $A_{\In/\Out} ^{\rm{mh}}$ will  agree with the one of $A_{\In/\Out} ^{\rm{hh}}$ with the difference that occurrences of $\pm \xi_3^{\rm p}$ in the principal symbols of $U^+_{\Out/\In}$ and $\calM^+_{\Out/\In}$ will now be replaced by $\td{\xi}_3^{\rm p}=i\sqrt{|\xi'|_g^2-c_{\rm p}^{-2}\t^2}$.
Moreover, there is no pair of $b_{\In/\Out}^{\rm p}$ but only one $b_{\Ev}^{\rm p}$ in the system we set up.
Hence, with $\xi_2=0$ as before we reach the decoupled system
\begin{subequations}
\begin{align}
	 \begin{split}
			&\begin{pmatrix}
				\t\xi_1 					& 	2\t\td{\xi}_3^{\rm p} 					& -\r_{\rm f}^{-1}\xi_3^{\rm f}\,\\
				2\mu_{\rm s} \xi_1^2-\r_{\rm s}	\t^2 	&	2\mu_{\rm s}\xi_1\td{\xi}_3^{\rm p}			& 0\\
				2\mu_{\rm s}\xi_1\xi_3^{\rm s} 			& 	-4\mu_{\rm s} \xi_1^2+2\r_{\rm s}\t^2	& \t 
			\end{pmatrix}
			\begin{pmatrix}
				b_{2,\Out}^{\rm s}\\ b_{\Ev}^{\rm p}\\ {b}^{\rm f}_{\Out}
			\end{pmatrix} \\
			&\qquad=\begin{pmatrix}
				-\t\xi_1 		&	-\r_{\rm f}^{-1}\xi_3^{\rm f}\, \\
				-2\mu_{\rm s} \xi_1^2+\r_{\rm s}\t^2 			&	0 \\
				2\mu_{\rm s}\xi_1\xi_3^{\rm s} 			&	{}-\t 
			\end{pmatrix}
			\begin{pmatrix}
				b_{2,\In}^{\rm s}\\ {b}^{\rm f}_{\In}
			\end{pmatrix},
		\end{split}\label{matrix_system_mh}\\ %
\intertext{ and }
	&(-\mu_{\rm s} \xi_1^2+\r_{\rm s}\t^2)(b_{1,\In}^{\rm s}+b_{1,\Out}^{\rm s})=0.\label{full_reflection_mh}\noeqref{full_reflection_mh}
	\end{align}
\end{subequations}

The determinant of the $3\times 3$ matrix on the left hand side of \eqref{matrix_system_mh} is given by 
\begin{equation}
	2\big(\t^4\r_{\rm s}\td{\xi}_3^{\rm p}+\r_{\rm f}^{-1}\xi_3^{\rm f}\big((2\mu_{\rm s}\xi_1^2-\r_{\rm s}\t^2)^2+4\mu_{\rm s}^2\xi_1^2\xi_3^{\rm s}\td{\xi}_3^{\rm p}\big)\big), \label{mix_hyp_determinant}%
\end{equation}
with real part $2\r_{\rm f}^{-1}\xi_3^{\rm f}(2\mu_{\rm s}\xi_1^2-\r_{\rm s}\t^2)^2$.
When the real part vanishes, that is, when $\r_{\rm s}\t^2=2\mu_{\rm s}\xi_1^2$, the imaginary part of \eqref{mix_hyp_determinant} becomes
$
	-4i\mu_{\rm s}\td{\xi}_3^{\rm p}\xi_1^2(\t^2+2\mu_{\rm s}\r_{\rm f}^{-1}\xi_3^{\rm f}\xi_3^{\rm s})>0,%
$
thus the system \eqref{matrix_system_mh} can be solved for $(
b_{2,\Out}^{\rm s}, b^{\rm p}_{\Ev}, b^{\rm f}_{\Out})^T$. In addition, by \eqref{full_reflection_mh}, the microlocal shear horizontal waves experience full internal reflection.

\begin{figure}[h]
	\includegraphics[page=2]{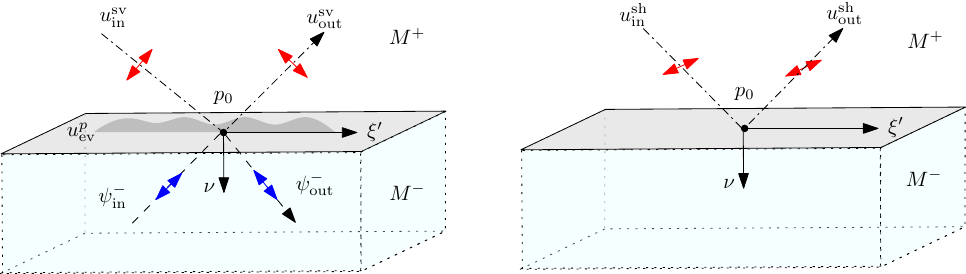}
	\caption{The mixed-hyperbolic transmission system. On the left we see microlocal SV waves in the solid and acoustic waves in the fluid, with creation of an evanescent $p$ wave. As pictured on the right, the microlocal SH waves are totally reflected.}
	\label{fig:mix-hyp}
\end{figure}

In order to produce an arbitrary acoustic wave  in the fluid side whose Cauchy data at $\G\times \R$ has wave front set in the mixed-hyperbolic region, using appropriate $s$ waves in the solid and with a possible creation of evanescent $p$ waves,  we have to solve the system
\begin{equation}\label{eq:mh_control}
	\begin{aligned}
	\begin{pmatrix}
		 	\t \xi_1 	& 	2\t \td{\xi}_3^{\rm p} & 		\t \xi_1 	\\
		 	2\mu_{\rm s}\xi_1^2-\r_{\rm s}\t^2 		& 	4\mu_{\rm s}\xi_1\td{\xi}_3^{\rm p} & 2\mu_{\rm s}\xi_1^2-\r_{\rm s}\t^2 		
		\\
		 2\mu_{\rm s}\xi_1\xi_3^{\rm s} & 	-4\mu_{\rm s}\xi_1^2+2\r_{\rm s}\t^2 & -2\mu_{\rm s}\xi_1\xi_3^{\rm s} 
	\end{pmatrix}
	\begin{pmatrix}
		b^{\rm s}_{2,\Out} \\
		b^{\rm p}_{\Ev} \\
		b^{\rm s}_{2,\In} \\
	\end{pmatrix}=\begin{pmatrix}
		-\r_{\rm f}^{-1}\xi_3^{\rm f} &  \r_{\rm f}^{-1}\xi_3^{\rm f}\\
		0 & 0\\ 
		- \t & 			-\t
	\end{pmatrix}\begin{pmatrix}
		b^{\rm f}_{\In} \\
		b^{\rm f}_{\Out}	\end{pmatrix}.
\end{aligned}%
\end{equation}
A computation shows that the determinant in the left hand side of 
\eqref{eq:mh_control} equals $-8\mu_{\rm s} \r_{\rm s}\t^3\xi_1\xi_3^{\rm s}\td{\xi}_3^{\rm p}$. %
Since  $\xi_1\t\neq0$ in the mixed region, the determinant is nonzero there and the system is elliptic.

On the other hand, as in the hyperbolic-hyperbolic case, \eqref{full_reflection_mh} 
implies that we cannot produce every configuration in the solid side by appropriately choosing the waves on the fluid side.
Given incoming and outgoing microlocal shear vertical waves in the solid side, we can construct them (up to lower order) using waves in the fluid side, and with creation of evanescent $p$-waves, 
if we can solve for $(b^{\rm f}_{\In} ,	b^{\rm f}_{\Out},		b^{\rm p}_{\Ev})$ the system
\begin{equation}\label{eq:mh_control_2}
	\begin{aligned}
	\begin{pmatrix}
		-\r_{\rm f}^{-1}\xi_3^{\rm f} &  \r_{\rm f}^{-1}\xi_3^{\rm f} &- 2\t \td{\xi}_3^{\rm p}\\
		0 & 0& - 4\mu_{\rm s}\xi_1\td{\xi}_3^{\rm p} &\\ 
		- \t & 	-\t& 4\mu_{\rm s}\xi_1^2-2\r_{\rm s}\t^2
	\end{pmatrix}\begin{pmatrix}
		b^{\rm f}_{\In} \\
		b^{\rm f}_{\Out} \\
		b^{\rm p}_{\Ev}
	\end{pmatrix}=
	\begin{pmatrix}
		 	\t \xi_1 	& 	  		\t \xi_1 	\\
		 	2\mu_{\rm s}\xi_1^2-\r_{\rm s}\t^2 		& 	 2\mu_{\rm s}\xi_1^2-\r_{\rm s}\t^2 		
		\\
		 2\mu_{\rm s}\xi_1\xi_3^{\rm s} & 	 -2\mu_{\rm s}\xi_1\xi_3^{\rm s} 
	\end{pmatrix}
	\begin{pmatrix}
		b^{\rm s}_{2,\Out} \\
		b^{\rm s}_{2,\In} \\
	\end{pmatrix}.
\end{aligned}
\end{equation}
The determinant of the matrix on the left is $8\r_{\rm f}^{-1}\mu_{\rm s} \t \xi_1\xi_3^{\rm f}\td{\xi}_3^{\rm p}\neq 0$, so this system is microlocally solvable and we can control the microlocal shear vertical waves from the fluid side.

\subsection{The elliptic-hyperbolic case (Figure~\ref{fig:ell-hyp})}

This case can happen only if $c_{\rm f}<c_{\rm s}$ in a neighborhood of the point at the interface we are interested in.
We have waves on both sides whose traces have wave front sets in the elliptic region for $c_{\rm s}$, $c_{\rm p}$ and the hyperbolic region for $c_{\rm f}$.
We seek to determine Dirichlet data for an outgoing acoustic  wave in the fluid region and an evanescent wave in the solid in terms of Dirichlet data for an incoming acoustic wave in the fluid.
We have the system
\begin{align}
{\nu} \cdot \p_t(U^+_{\Ev} q_{\Ev}^b)=&-\r_{\rm f}^{-1}\Lambda_{\In}^-\psi_{\In,b}^--\r_{\rm f}^{-1}\Lambda_{\Out}^-\psi_{\Out,b}^-	\\*%
-i\calM^+_{\Ev}(q_{\Ev}^{b})=&-\p_t \psi_{\In,b}^- {\nu}-\p_t \psi_{\Out,b}^-\; {\nu}.%
\end{align}
Its solvability reduces to the ellipticity in terms of the outgoing and evanescent amplitudes of %
\begin{subequations}
\begin{align}\label{matrix_system_ell_hyp}
	\begin{pmatrix}
		\t\xi_1 					& 	\t\td{\xi}_3^{\rm p} 					& -\r_{\rm f}^{-1}\xi_3^{\rm f}\,\\
		2\mu_{\rm s} \xi_1^2-\r_{\rm s}\t^2 	&	2\mu_{\rm s}\xi_1\td{\xi}_3^{\rm p}			& 0\\
		2\mu_{\rm s}\xi_1\td{\xi}_3^{\rm s} 			& 	-2\mu_{\rm s} \xi_1^2+\r_{\rm s}\t^2	& \t 
	\end{pmatrix}
	\begin{pmatrix}
		b_{2,\Ev}^{\rm s}\\ b^{\rm p}_{\Ev} \\ {b}^{\rm f}_{\Out}
	\end{pmatrix}
	 &=\begin{pmatrix}
			-\r_{\rm f}^{-1}\xi_3^{\rm f}  \\
			0 \\
			{}-\t 
	\end{pmatrix}
	b^{\rm f}_{\In},
	\intertext{ and }
	(-\mu_{\rm s} \xi_1^2+\r_{\rm s}\t^2) \;b_{1,\Ev}^{\rm s}&=0,\label{eq:ell-hyp-sh}\noeqref{eq:ell-hyp-sh}
	\end{align} %
\end{subequations}
with
\begin{equation}
	\td{\xi}_3^{\rm s}=i\sqrt{|\xi'|_g^2-c_{\rm s}^{-2}\t^2}, \quad \td{\xi}_3^{\rm p}=i\sqrt{|\xi'|_g^2-c_{\rm p}^{-2}\t^2},\quad {\xi}_3^{\rm f}= \sqrt{c_{\rm f}^{-2}\t^2-|\xi'|^2_g}, %
\end{equation} 
all evaluated at $\xi'=(\xi_1,0)$.
From \eqref{eq:ell-hyp-sh}, $b_{1,\Ev}^{\rm s}=0$, so there exist no microlocal ``SH'' evanescent waves (note here that there is no propagating wave in the solid region, so a distinction between microlocal ``SH'' and ``SV'' evanescent waves is only made by analogy to the case were the wave front set of the elastic waves is in the hyperbolic region for the solid).
The determinant of the matrix in \eqref{matrix_system_ell_hyp} is
\begin{align}
	 &\t^4\r_{\rm s}i\sqrt{\xi_1^2-c_{\rm p}^{-2}\t^2}+\r_{\rm f}^{-1}\sqrt{c_{\rm f}^{-2}\t^2-\xi_1^2} \Big((2\mu_{\rm s}\xi_1^2-\r_{\rm s}\t^2)^2-4\mu_{\rm s}^2\xi_1^2\sqrt{\xi_1^2-c_{\rm s}^{-2}\t^2}\sqrt{\xi_1^2-c_{\rm p}^{-2}\t^2}\Big),%
\end{align}
and it has positive imaginary part (note that $\t\neq 0$ since we are in the hyperbolic region for the fluid and $\t$ can vanish only in the elliptic region for any of the three speeds) 
so the system is elliptic.
Therefore the principal amplitude of the outgoing acoustic wave in the fluid and the evanescent wave in the solid are uniquely determined by the one of the incoming acoustic wave in the fluid.

\begin{figure}[h]
	\includegraphics[page=3]{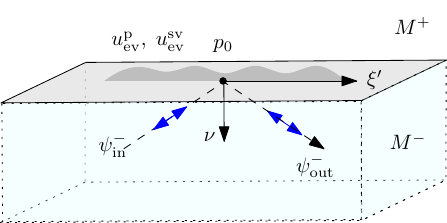}
	\caption{The elliptic-hyperbolic system. The acoustic waves in the fluid experience total internal reflection, with creation of evanescent $p$ and evanescent ``SV'' waves. No microlocal ``SH'' waves are created.}
	\label{fig:ell-hyp}
\end{figure}

\subsection{The hyperbolic-elliptic region (Figure~\ref{fig:hyp-ell})}

In this case we the wave front set of the traces of the various waves is in the hyperbolic region for $c_{\rm p}$ and $c_{\rm s}$ and in the elliptic region for $c_{\rm f}$ (by assumption in the connected component of the elliptic region in which $\t<0$); this case can only happen if $c_{\rm f}>c_{\rm p}$ at $p_0$. %
We thus seek solutions in the fluid region as evanescent waves; we obtain the following system:
\begin{align}
	{\nu} \cdot \p_t(U_{\In}^+ q_{\In,b}^++U_{\Out}^+ q_{\Out,b}^+)=&-\r_{\rm f}^{-1}\Lambda_{\Ev}^-\psi_{\Ev,b}^-	\label{hyp_ell_system_first},	\\*
	-i(\calM_{\In}^+ q_{\In,b}^++\calM_{\Out}^+q_{\Out,b}^+)=&-\p_t \psi_{\Ev,b}^-\; {\nu}\label{hyp_ell_system_second}.
\end{align}
We write $b_{\Ev}^{\rm f}=\calF(\psi^-_{\Ev,b})$ and, as before, at the principal symbol level our system becomes
\begin{subequations}\label{hypell_matrix_system}
\begin{align}
	\begin{pmatrix}
		\t\xi_1 					& 	\t\xi_3^{\rm p} 					& -\r_{\rm f}^{-1}\td{\xi}_3^{\rm f}\\
		2\mu_{\rm s} \xi_1^2-\r_{\rm s}\t^2 	&	2\mu_{\rm s}\xi_1\xi_3^{\rm p}			& 0\\
		2\mu_{\rm s}\xi_1\xi_3^{\rm s} 			& 	-2\mu_{\rm s} \xi_1^2+\r_{\rm s}\t^2	& \t 
	\end{pmatrix}
	\begin{pmatrix}
		b_{2,\Out}^{\rm s}\\ b_{\Out}^{\rm p}\\ {b}_{\Ev}^{\rm f}
	\end{pmatrix} 
	 =&\begin{pmatrix}
		-\t\xi_1 	& 	\t\xi_3^{\rm p} 		\\		
		-2\mu_{\rm s} \xi_1^2+\r_{\rm s}\t^2 &	2\mu_{\rm s}\xi_1\xi_3^{\rm p}\\
		2\mu_{\rm s}\xi_1\xi_3^{\rm s} 	& 	2\mu_{\rm s} \xi_1^2-\r_{\rm s}\t^2		
	\end{pmatrix}
	\begin{pmatrix}
		b_{2,\In}^{\rm s}\\ b_{\In}^{\rm p}
	\end{pmatrix}
	\intertext{and}
	(-\mu_{\rm s} \xi_1^2+\r_{\rm s}\t^2)(b_{1,\In}^{\rm s}+b_{1,\Out}^{\rm s})=&0, \label{hyp_ell_full_reflection}
	\end{align} %
\end{subequations}
where $$\td{\xi}_3^{\rm f}=i\sqrt{|\xi'|^2-c_{\rm f}^{-2}\t^2}.$$
Using \eqref{determinant} with $\xi_3^{\rm f}$ replaced by $\td{\xi}_3^{\rm f}$, it is easy to see that the real part of the determinant is given by $\t^4\r_{\rm s}\xi_3^{\rm p}\neq 0$ (recall that $\t<0$), demonstrating the ellipticity of the system \eqref{hypell_matrix_system} and the microlocal well-posedness of the transmission problem in this case.
The microlocal shear horizontal waves experience total internal reflection.

\begin{figure}[h]
	\includegraphics[page = 4]{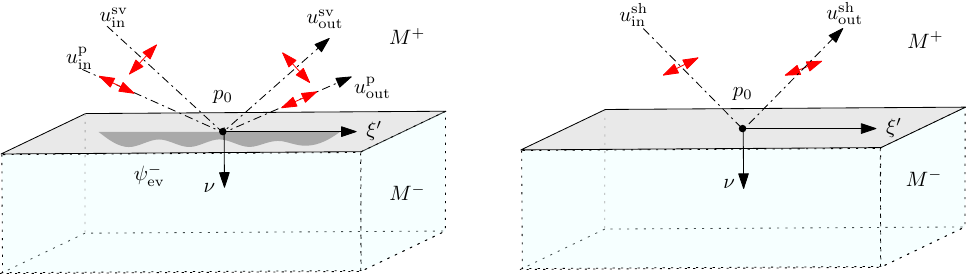}
	\caption{The hyperbolic-elliptic transmission system. We have $p$ and SV waves in the solid and evanescent waves in the fluid. Microlocal SH waves are totally reflected.}
	\label{fig:hyp-ell}
\end{figure}

\subsection{The mixed-elliptic case (Figure~\ref{fig:mix-ell})}

This case can happen only if $c_{\rm s}<c_{\rm f}$ at $p_0$.
We have incoming $\rm s$ waves in the solid meeting the interface at an angle greater than the critical angle for $\rm p$ waves.
The wave front set of the trace at $\G\times \R$ is by assumption contained in the elliptic region for the fluid, with $\t<0$.
We seek $\rm p$ waves and body waves in the fluid as evanescent modes.
Our transmission system takes the form
\begin{align}
{\nu} \cdot \p_t(U_{\Ev,\In}^+ \, q^+_{\Ev,\In,b}+U_{\Ev,\Out}^+ \,q^+_{\Ev,\Out,b})=&-\r_{\rm f}^{-1}\Lambda_{\Ev}^-\psi_{\Ev,b}^-	\label{mix_ell_system_first},	\\*
-i(\calM_{\Ev,\In}^+ \,q^+_{\Ev,\In,b}+\calM_{\Ev,\Out}^+ \,q^+_{\Ev,\Out,b})=&-\p_t \psi_{\Ev,b}^-\; {\nu}\label{mix_ell_system_second}.
\end{align}
Then at the principal symbol level and  for $\xi_2=0$ we find the decoupled system
\begin{subequations}
\begin{align}
	\begin{pmatrix}
		\t\xi_1 					& 	2\t\td{\xi}_3^{\rm p} 					& -\r_{\rm f}^{-1}\td{\xi}_3^{\rm f}\,\\
		2\mu_{\rm s} \xi_1^2-\r_{\rm s}\t^2 	&	4\mu_{\rm s}\xi_1\td{\xi}_3^{\rm p}			& 0\\
		2\mu_{\rm s}\xi_1\xi_3^{\rm s} 			& 	-4\mu_{\rm s} \xi_1^2+2\r_{\rm s}\t^2	& \t 
	\end{pmatrix}
	\begin{pmatrix}
		b_{2,\Out}^{\rm s}\\ b^{\rm p}_{\Ev} \\ {b}^{\rm f}_{\Ev}
	\end{pmatrix}& =\begin{pmatrix}
		-\t\xi_1 	 				\\
		-2\mu_{\rm s} \xi_1^2+\r_{\rm s}\t^2 			\\
		2\mu_{\rm s}\xi_1\xi_3^{\rm s} 	 	
	\end{pmatrix}
	 	b_{2,\In}^{\rm s}\label{mix_ell_reduced_system}, \\
	 	\intertext {and}
	\label{mix_hyp_full_reflection}\noeqref{mix_hyp_full_reflection}
	(-\mu_{\rm s} \xi_1^2+\r_{\rm s}\t^2)(b_{1,\In}^{\rm s}+b_{1,\Out}^{\rm s})&=0,
	\end{align} %
\end{subequations}
where $\td{\xi}_3^{\rm f}=i\sqrt{|\xi'|^2_g-c_{\rm f}^{-2}\t^2}$.
The determinant of the square matrix in \eqref{mix_ell_reduced_system} takes the form
\begin{equation}
	2\big(\t^4\r_{\rm s}\td{\xi}_3^{\rm p}+\r_{\rm f}^{-1}\td{\xi}^{\rm f}_3\big((2\mu_{\rm s}\xi_1^2-\r_{\rm s}\t^2)^2+4\mu_{\rm s}^2\xi_1^2\xi_3^{\rm s}\td{\xi}_3^{\rm p}\big)\big),%
\end{equation}
(cf. \eqref{mix_hyp_determinant}).
Its real part is given by 
$	8\mu_{\rm s}^2\r_{\rm f}^{-1}\xi_1^2\xi_3^{\rm s}\td{\xi}_3^{\rm p}\td{\xi}^{\rm f}_3$, 
which does not vanish (recall that one cannot have $\xi_1=0$ in the elliptic region for $c_{\rm p}$ or $c_{\rm f}$).
We reach the conclusion that the matrix on the left hand side of \eqref{mix_ell_reduced_system}  is elliptic, showing microlocal solvability of the system. Again the microlocal shear horizontal waves experience total internal reflection.
Notice also that by \eqref{mix_ell_reduced_system}, in the absence of incoming SV waves, i.e. if $b_{2,\In}^{\rm s}=0$, or if there are no SV waves at all,  no evanescent waves are created on either side of the interface.
In other words, in such a case we do not obtain surface waves as we do in the elliptic-elliptic case, see below.

\begin{figure}[h]
	\includegraphics[page = 5]{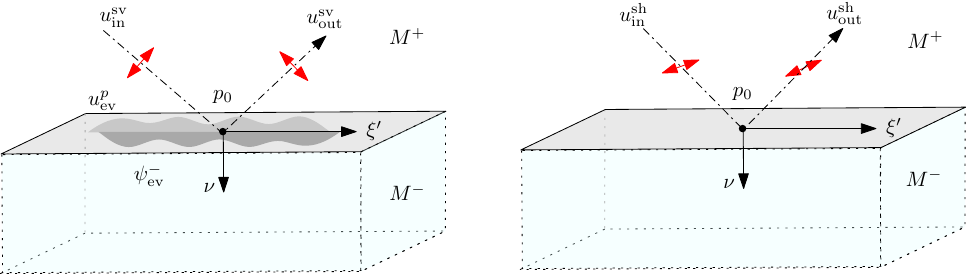}
	\caption{The mixed-elliptic transmission system. We have only $s$ body waves in the solid (SV on the left, SV on the right); the $p$ waves and the acoustic waves in the fluid are evanescent.}
	\label{fig:mix-ell}
\end{figure}

\subsection{The elliptic-elliptic case (Figure~\ref{fig:ell-ell})}

Any nontrivial solutions to the system \eqref{pot_sys_1}-\eqref{pot_tr_2} whose traces at the interface $\G\times\R$ have wave front set in the elliptic region for both $c_{\rm f}$ and $c_{\rm s}$ (thus also automatically for $c_{\rm p}$) cannot be produced by body waves on either side.
The only possibility is that they are produced by sources at the interface.
We seek such solutions as 
 evanescent waves, which decay exponentially away from $\G$.
 As mentioned in the introduction, surface waves at the interface between two media are generally referred to as Stoneley waves and in the particular case of a solid-fluid interface they are often called Scholte waves.
 We look for solutions of the system
\begin{subequations}
	\begin{align}
	{\nu} \cdot \p_t(U_{\Ev}^+ \, q^+_{\Ev,b})=&-\r_{\rm f}^{-1}\Lambda_{\Ev}^-\psi_{\Ev,b}^-	\label{ell_system_first}	\\*
	-i(\calM_{\Ev}^+ \, q^+_{\Ev,b})=&-\p_t \psi_{\Ev,b}^-\; {\nu}.\label{ell_system_second}
	\end{align}
\end{subequations}
As before, we let $(b_{1,\Ev}^{\rm s},b_{2,\Ev}^{\rm s},b_{\Ev}^{\rm p})=\calF\big(q^+_{\Ev}\big)$ and $b^{\rm f}_{\Ev}=\big(\psi^-_{\Ev}\big)$.
With notations as before, at the principal symbol level we obtain the system (with $\xi_2=0$ as usual)
\begin{subequations}
\begin{align}
\label{system_ell_ell}
	A_{\Ev}^{\rm{ee}}\begin{pmatrix}
		b_{2,\Ev}^{\rm s}\\ b_{\Ev}^{\rm p}\\ b_{\Ev}^{\rm f}
	\end{pmatrix}=\begin{pmatrix}
		0\\0\\0
	\end{pmatrix},& \quad A_{\Ev}^{\rm{ee}}=\begin{pmatrix}
		\t\xi_1 					& 	\t \td{\xi}_3^{\rm p} 					& -\r_{\rm f}^{-1}\td{\xi}_3^{\rm f}\,\\
		2\mu_{\rm s} \xi_1^2-\r_{\rm s}\t^2 	&	2\mu_{\rm s}\xi_1\td{\xi}_3^{\rm p}		& 0\\
		2\mu_{\rm s}\xi_1\td{\xi}_3^{\rm s}			& 	-2\mu_{\rm s} \xi_1^2+\r_{\rm s}\t^2	& \t 
	\end{pmatrix}\\ %
\intertext{ and }
	&(-\mu_{\rm s} \xi_1^2+\r_{\rm s}\t^2) \, b_{1,\Ev}^{\rm s}=0. 
	\label{system_ell_ell_b}\noeqref{system_ell_ell_b}
	\end{align}
\end{subequations}
We immediately obtain $b_{1,\Ev}^{\rm s}=0$ and the determinant of $A_{\Ev}^{\rm{ee}}$ becomes
\begin{equation}
	i \Big[\t^4\r_{\rm s}\sqrt{\xi_1^2-c_{\rm p}^{-2}\t^2}+\r_{\rm f}^{-1}\sqrt{\xi_1^2-c_{\rm f}^{-2}\t^2}\Big((2\mu_{\rm s}\xi_1^2-\r_{\rm s}\t^2)^2-4\mu_{\rm s}^2\xi_1^2\sqrt{\xi_1^2-c_{\rm s}^{-2}\t^2}\sqrt{\xi_1^2-c_{\rm p}^{-2}\t^2}\Big)\Big]. %
\end{equation}
Setting $z=\t^2/\xi_1^2$ (recall that $\xi_1\neq 0$), the vanishing of the determinant is equivalent to the secular equation for  Scholte waves  $S_{p_0}(z)=0$, where
\begin{equation}\label{eq:secular_equation}
	S_{p_0}(z)=\xi_1^5 \Big[z^2\r_{\rm s}\sqrt{1-c_{\rm p}^{-2}z}+\r_{\rm f}^{-1}\sqrt{1-c_{\rm f}^{-2}z}\Big((2\mu_{\rm s}-\r_{\rm s}z)^2-4\mu_{\rm s}^2\sqrt{1-c_{\rm s}^{-2}z}\sqrt{1-c_{\rm p}^{-2}z}\Big)\Big]. %
\end{equation}

Equation~\eqref{eq:secular_equation}   has been studied in the geophysical literature, see e.g. \cite{StrickE1956Svfa}, \cite{ansell}.  
It follows from the analysis in \cite{ansell} that for any positive  values of the Lam\'e parameters and the densities at $p_0$ there exists a positive simple root $z:=c_{{\rm Sc}}^2(p_0)$ with  $0<c_{{\rm Sc}}^2(p_0)<\min\{c_{\rm s}^2(p_0),c_{\rm f}^2(p_0)\}$ (along with possibly other real and complex roots, upon appropriately interpreting the square roots).
This root can be viewed as the only positive zero of a complex valued function\\  $S_1(\r_{\rm f},\r_{\rm s}, \l_{\rm f},\l_{\rm s},\mu_{\rm s},z)$ which is holomorphic in $z$ in a neighborhood of $c_{{\rm Sc}}^2(p_0)$ in $\C$ and depends smoothly on the rest of its entries, as long as they are positive.
In the invariant formulation we can now replace $\xi_1 $ by $|\xi'|_g$, and also multiply the third column of $A_{\Ev}^{\rm{ee}}$ by a homogeneous real valued elliptic symbol  $\a(\xi',\t)$ of order 1 in order to make $A_{\Ev}^{\rm{ee}}$ homogeneous of order 2 (this has the effect of turning \eqref{system_ell_ell} into a system for  $(b_{2,\Ev}^{\rm s}, b_{\Ev}^{\rm p},\a^{-1}(\xi',\t)  b_{\Ev}^{\rm f})$).
 Denote this modified matrix valued symbol by $\td{A}_{\Ev}^{\rm{ee}}$.
Then $\td{A}^{\rm{ee}}_{\Ev}(\xi',\t)$ fails to be elliptic at $(p_0,t_0,\xi',\t)\in T^*(\G\times \R)\setminus 0$ when $\t^2=c_{{\rm Sc}}^2(p_0)|\xi'|_g^2$.
In fact, for $(x',t)$ near $(p_0,t_0)$ this symbol fails to be elliptic when $\t^2=c_{{\rm Sc}}^2(x')|\xi'|_g^2$, where $c_{\rm Sc}^2$ is a smooth and positive function near $p_0$.
This can be seen by changing to semigeodesic coordinates with the metric being Euclidean at $x'$, setting up a system as \eqref{system_ell_ell}, and defining $c_{{\rm Sc}}^2(x')$ as the unique positive zero of the function $S_1$ mentioned earlier  corresponding to the Lam\'e parameters and densities evaluated at $x'$ (this zero  is also a simple zero of $S_{x'}$).
Then smoothness of $c_{{\rm Sc}}^2(x')$ can be shown using the implicit function theorem for $z\mapsto S_1(\r_{\rm f},\r_{\rm s}, \l_{\rm f},\l_{\rm s},\mu_{\rm s},z)$, viewed as a function from a subset of $\R^2$ to one of $\R^2$ by writing $z=x+yi$.

Now in a conical neighborhood of the characteristic variety $\Sigma_{\rm Sc}:=\{(x',t,\xi',\t)\in T^*(U\times \R):\t^2=c_{\rm Sc}^2|\xi'|_g^2\}$ write $S_{x'}(\t^2/|\xi'|_g^2)=(\t^2-c_{\rm Sc}^2|\xi'|_g^2)\td{S}(x',t,\xi',\t)$, where $\td{S}$ is an elliptic real valued symbol of order 3.
The adjugate matrix $\adj(\td{A}^{\rm{ee}}_{\Ev})$ is a matrix valued symbol which is homogeneous of order 4, and $-i\adj(\td{A}^{\rm{ee}}_{\Ev})\td{A}^{\rm{ee}}_{\Ev}=\a(\xi',\t)(\t^2-c_{\rm Sc}^2|\xi'|_g^2)\td{S}(x',t,\xi',\t) Id$.
This shows that $\rm{Op}(\td{A}^{\rm{ee}}_{\Ev})$ is an operator of real principal type as defined in \cite{MR661876}, in a suitable open conical set in $T^*(\G\times \R)$,
  which propagates singularities along the null bicharacteristics of $\t^2-c_{{\rm Sc}}^2(x')|\xi'|^2_g$.
   Using it to propagate Cauchy data given at a spacelike hypersurface in $\G\times \R$ such as  $\G\times \{t=t_0\}$, we obtain microlocally non-trivial solutions of \eqref{system_ell_ell}.
Those can then be used as Dirichlet data for evanescent waves on both sides of the interface.

\begin{figure}[h]
	\includegraphics[page = 6]{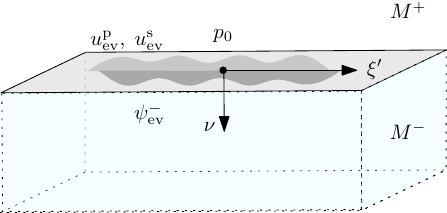}
	\caption{The elliptic-elliptic transmission system, with only evanescent waves on both sides.}
	\label{fig:ell-ell}
\end{figure}

We summarize the findings of this section in the following theorem:

\begin{theorem}\label{thm:first}
    Suppose we are given solutions of the coupled system of evolution equations (\ref{sys_1}-\ref{tr_3}), with Cauchy data supported away from the interface $\G$, and satisfying Assumption \ref{as:int}.
    Then provided that the wave front set of the boundary data of incoming solutions to the interface $\G\times \R$ is disjoint from the glancing regions then the transmission problem for the solid-fluid interface is microlocally well posed.
    Scholte surface waves propagating singularities along  the interface $\G\times \R$  are always possible and can be created by sources at the interface, for instance by Cauchy data on $\G\times \{t_0\}$ for some $t_0\in \R$.
\end{theorem}

\section{Justification of the Parametrix}
\label{ssec:justification}{}

In this section we justify the parametrix which can be constructed for \eqref{pot_sys_1}-\eqref{pot_tr_3} with initial data supported away from the interface and boundary, using the techniques described in Sections \ref{sec:acoustic} and \eqref{sec:elastic} once solvability for the transmission systems at the principal level has been established (in Section \ref{sec:well_posedness_microlocal}).
That is, we show that the parametrix  produced in that way differs from an actual solution by a smooth function/vector field.
The method we use is an adaptation of one used by Taylor in \cite{MR535598}. For a discussion on the justification of the parametrix in the more standard solid-solid or fluid-fluid case, see Appendix \ref{appendix_a}.

The difference between the parametrix and an actual solution satisfies the system 
\begin{equation}\label{eq:nonhomog_system}
\begin{cases}
	(\p_t^2-P^+) u^+ =f^+ &\quad\text{ in }M^+ \times \R,\\*
	(\p_t^2-\td{P}^-) {\psi}^-=f^-&\quad\text{ in }M^-\times \R,\\*
	\nu\cdot\p_tu^++\r_{\rm f}^{-1}\; \p_{\nu}{\psi}^-=h_1&\quad\text{ on }\Gamma\times \R,\\*
N(u^+)+\p_t\psi^-\; \nu=h_2 &\quad\text{ on }\Gamma\times \R,\\*
N(u^+)=h_3&\quad\text{ on } \p M\times \R,\\*
(u^+,\psi^-)=0&\quad \text{ for }t\ll0,
\end{cases}
\end{equation}
where 
$P ^+=\r_{\rm s}^{-1} E$, $\td{P}^-=\l_{\rm f}\div(\r_{\rm f}^{-1}\n(\cdot ))$, 
 $f^+\in C^\infty (\oM^+\times \R;\pi_1^*TM)$, $f^-\in C^\infty (\oM^-\times \R)$, $h_1\in C^\infty(\G\times \R)$, $h_2\in C^\infty(\G\times \R;\pi_1^*TM)$, $h_3\in C^\infty (\p M\times  \R;\pi_1^*TM)$ and $\pi_1:M\times \R\to M$ is the projection.
To justify that the parametrix has the same smoothness properties as the actual solution we need to show that $u^+$, $\psi^-$ are smooth up to the interface and boundary.

The timelike hypersurfaces $\G\times \R$ and $\p M\times \R$ are non-characteristic for $\p_t^2-P^+$ and  $\p_t^2-\td{P}^-$ and knowledge of $N(w^+)\big|_{\G}$ allows the recovery of $ \p_\nu w^+ \big|_{\G}$ from $w^+ \big|_{\G}$.
With the Cauchy-Kovalevskaya method and Borel's lemma we can produce $w^+$, $\chi^-$ which are smooth up to $\p M\times \R$ and $\G\times \R$, vanish for $t\ll 0$, and satisfy
\begin{equation}
	\begin{cases}
		\p_\nu^k (\p_t^2{w}^+-P^+ {w}^+-f^+)=0&\text{ for } k\geq 0\text { on }\p M^+\times \R,\\
		{w}^+=0& \text { on }\G\times \R,\\
		N({w}^+)=h_2& \text { on }\G\times \R,\\
		N({w}^+)=h_3& \text { on }\p M\times \R,
	\end{cases}
	\end{equation}
and 
\begin{equation}
	\begin{cases}
		\p_\nu^k (\p_t^2{\chi}^--\td{P}^- {\chi}^--f^-)=0&\text{ for } k\geq 0\text { on }\G\times \R,\\
		{\chi}^-=0& \text { on }\G\times \R,\\
		\p_\nu{\chi}^-=\r_{\rm f} h_1&\text { on }\G\times \R. \\ %
	\end{cases}
\end{equation}
Then the difference $(z^+,{\phi}^-):=(u^+-{w}^+,\psi^--{\chi}^-)$  satisfies
\begin{subequations}
		\begin{align}
			(\p_t^2-P ^+) z^+ =\td{f}^+ &\quad\text{ in }M^+ \times \R,\label{hom_sys_1}\noeqref{hom_sys_1}
			\\*
			(\p_t^2-\td{P}^-) {{\phi}}^-=\td{f}^-&\quad\text{ in }M^-\times \R,\label{hom_sys_2}\noeqref{hom_sys_2}
			\\*
			\nu\cdot\p_tz^+=-\r_{\rm f}^{-1}\; \p_{\nu}{{\phi}}^-&\quad\text{ on }\Gamma\times \R,\label{hom_tr_1}\noeqref{hom_tr_1}
			\\*
		N(z^+)=-\p_t{\phi}^-\; \nu &\quad\text{ on }\Gamma\times \R \label{hom_tr_2}\noeqref{hom_tr_2}
		\\*
		N(z^+)=0&\quad\text{ on } \p M\times \R,\label{hom_tr_3}\noeqref{hom_tr_3}
		\\*
		(z^+,{\phi}^-)=0&\quad \text{ for }t\ll 0\label{hom_sys_4}\noeqref{hom_sys_4},
		\end{align}
\end{subequations}
where $\td{f}^\pm$ vanish to infinite order at $\G\times \R$ and $\p M\times \R$.

Now pass to the displacement-displacement system (\ref{sys_1_aux}-\ref{tr_3_aux}) we used to show well posedness in Section \ref{sec:well_posedness}: set 
 $u^-=z_0-\int_{-\infty}^t\r_{\rm f}^{-1}\n\phi^-(x,\t)d\t$, where $z_0$ is divergence free, constant in time and $\nu\cdot z_0=0$.
 Observe that by \eqref{eq:potential_neu}, the potential part of the displacement $u^-$ vanishes for  $t\ll 0$ since the pressure then is 0 by \eqref{eq:psi-pi} and \eqref{hom_sys_4}.
Using that $\p_t^2 u^-=-\r_{\rm f}^{-1}\n \p_t \phi^{-}$ and 
\begin{equation}\label{phi_u}
 \p_t\phi^-=- \l_{\rm f} \div u^-+\int_{-\infty}^{t} \td{f}^-(x,\t) d\t,
\end{equation}
which follows by \eqref{hom_sys_2} upon integrating in time and using the expression above for $u^-$, we find
\begin{equation}\label{eq:sys_2_nonhom}
\begin{cases}
	\p_{t}^2z^+-\r_{\rm s}^{-1}Ez^+=\td{f}^+ &\quad \text{ in } {M}^+\times\R,\\
	\p_t^2u^--\r_{\rm f}^{-1}\n \l_{\rm f} \div u^-=F^-&\quad \text{ in } {M}^-\times\R,\\
	 z^+\cdot \nu=u^-\cdot \nu&\quad \text{ on } \Gamma\times \R,\\
	  N(z^+)=\l_{\rm f} (\div u^- )\;\nu&\quad \text{ on } \Gamma\times \R,\\
	   N(z^+)=0&\quad \text { on }\p M\times \R,\\
	  (z^+, \div u^-)=0&\quad \text { for }t\ll 0,\\
\end{cases}
\end{equation}
where $F^-(x,t)=-\r_{\rm f}^{-1}\int_{-\infty}^t\n\td{f}^-(x,\t)d\t$. Note that $\td{f}^+$ and $F^-$ are smooth and both vanish to infinite order at $\G\times \R $ and $\p M\times \R$, thus for each $s\geq 0$ Proposition \ref{prop:domain} implies $\mathbf{F}(s):=(\td{f}^+(\cdot, s), F^-(\cdot, s))\in D(P^k)$ for all $k=1,2,\dots$.
Now \eqref{eq:sys_2_nonhom}
 can be solved using Duhamel's formula:
\begin{equation}
(z^+,u^-)(\cdot, t)=(0,z_0)+\int _{-\infty}^t\frac{\sin(\sqrt{-P}(t-s))}{\sqrt{-P}} \mathbf{F}(s)ds, \qquad \div z_0=0,\quad \p_t z_0\equiv 0,\quad z_0 		\cdot \nu \big|_{\G}=0.
\end{equation}
By the functional calculus, $(z^+,u^-) \in C^\infty(\R; D(P^k))$ for all $k$.
Therefore,  Corollary \ref{cor:estimate} implies that $z^+\in C^\infty(\R; H^{2k}(M^+))$ and $\div u^-(x,t)\in C^\infty(\R; H^{2k-1}(M^-))$ for all $k\geq 0$.
Thus by Sobolev embedding  $z^+$ (and hence also $u^+$ in \eqref{eq:nonhomog_system}) is smooth  up to $\G\times \R$ and $\p M\times \R$, and $\div u^-$ is smooth up to $\G\times \R$.
We conclude by \eqref{phi_u} that ${\phi}^-$, hence also $\psi^-$ in \eqref{eq:nonhomog_system}, is smooth up to the interface.

Once a parametrix $(\td{u}^+,\td{\psi}^-)$ has been constructed for \eqref{pot_sys_1}-\eqref{pot_tr_3}, differing from an actual solution by a smooth vector field/function, we can obtain a parametrix to the original system (\ref{sys_1}-\ref{tr_3}) by setting $  (\td{u}^+,\td{w}^+,\td{p}^-,\td{v}^-)=(\td{u}^+,\p_t\td{u}^+,\p_t\td{\psi}^-,\r_{\rm f}^{-1}(  Z_0-\n \td{\psi}^-)),$ where $Z_0$ is the solenoidal part of the decomposition \eqref{helmholtz} of the initial data for the actual solution.

\section{The inverse problem} \label{sec_IP}
In this section, we consider the inverse problem of recovery of the  solid coefficients $\rho_{\rm s}$, $\lambda_{\rm s}$, $\mu_{\rm s}$ and the fluid ones $\rho_{\rm f}$, $\lambda_{\rm f}$ from boundary measurements.  %
As explained in the Introduction, we will use the boundary rigidity result in \cite{MR3454376}. To recover $c_{\rm f}$ in $M_-$, we would need rays in $M_-$ which can be created by incoming ones from $\bo$, eventually creating  a ray back to $\bo$; moreover, we want all such rays in $M_-$ to have such property. Hence, we have to exclude speeds $c_{\rm f}$ allowing for totally reflected rays in $M_-$. This happens when $c_{\rm f}|_{\Gamma^-} < c_{\rm s}|_{\Gamma^+} $, where $c|_{\Gamma_\pm}$ are limits from $M_\pm$. Therefore, we assume 
\begin{equation}  \label{condG}
c_{\rm s}|_{\Gamma^+} < c_{\rm f}|_{\Gamma^-}.
\end{equation}
Then the rays hitting $\Gamma_-$ would leave a trace on $T^*(\Gamma\times \R)$  either in the hyperbolic-hyperbolic region (excluding tangential rays), see Section~\ref{ssec:hh},  or in the mixed-hyperbolic one, see Section~\ref{ssec:mh}.

We assume the following foliation condition. 
Assume that there exist two  smooth non-positive functions $\foliation_{\rm s}$ and $\foliation_{\rm p}$ in $\bar M_+$ with $d\foliation\not=0$, $\foliation^{-1}(0)=\bo$, and $\foliation^{-1}(-1)=\Gamma$, where $\foliation$ is either $\foliation_{\rm s}$ or $\foliation_{\rm p}$.  Assume that the level sets $\foliation_{\rm s}^{-1}(c)$, $\foliation_{\rm p}^{-1}(c)$, $c\in [-1,0]$, are strictly convex w.r.t.\ the  speed $c_{\rm s}$, $c_{\rm p}$ in $M_+$, respectively, when viewed from $\Gamma_{0}=\bo$. Of course, we may have just one such function, i.e., $\foliation_{\rm s}=\foliation_{\rm p}$ is possible. Assume also that there is a smooth non-positive $\foliation_{\rm f}$ defined on $\bar M_-$, so that $\foliation^{-1}(0)=\Gamma$, and $d \foliation_{\rm f} \not=0$ except at one interior point, where  $\foliation_{\rm f} $ attains its minimum. We require that the level set $\foliation_{\rm  f}^{-1}(c)$,  $c<-1$, is  strictly convex w.r.t.\ the  speed $c_{\rm f}$  in $M_-$, when viewed from $\Gamma$.  

Recall that  the foliation condition implies non-trapping  as noted in \cite{MR3454376}, for example. In our case, in $M_+$, this means that rays in $M_+$ not hitting $\Gamma$, would hit $\bo$ both in the future and in the past. In $M_-$, we have the usual non-trapping property. 

We define the outgoing Neumann-to-Dirichlet map $\mathcal{N}_{\rm out}$ as follows. Given $f\in C_0^\infty(\bo \times \R_+;\, \C^3)$, let $u$ be the solution to (\ref{sys_1}--\ref{tr_2}) with the homogeneous condition \eqref{tr_3} replaced by $N(u^+) =f$ on $\bo\times\R$, and zero Cauchy data  at $t=0$ zero instead of \eqref{eq_init}. Set $\mathcal{N}_{\rm out}f = u $ on $\bo\times\R$. Then $\mathcal{N}_{\rm out}$ measures the response to boundary sources related to waves propagating to the future. 

Note first that $\mathcal{N}_{\rm out}$ is well defined since we can construct a solution to $N(u^+)=f$ near $\bo\times\R$ locally (not solving the PDE), subtract it from the actual solution, and reduce the problem to one with homogeneous Neumann boundary condition but a non-trivial source. Then we can use Duhamel's principle to reduce it to a superposition of linear problems of the kind (\ref{sys_1}--\ref{tr_3}) with non-trivial Cauchy data of the kind \eqref{eq_init}. 

\begin{theorem}\label{thm_inverse} 
Assume we have two systems in $M$ with coefficients  $ \rho_{\rm s}$, $ \mu_{\rm s}$, $ \lambda_{\rm s}$   and  $\tilde \rho_{\rm s}$, $\tilde \mu_{\rm s}$, $\tilde \lambda_{\rm s}$ in $M_+$ and $\tilde M_+$, respectively; and $(\rho_{\rm f}, \lambda_{\rm f})$,  and $(\tilde \rho_{\rm f}, \tilde \lambda_{\rm f})$ in $M_-$ and $\tilde M_-$, respectively. Assume  $\mathcal{N}_{\rm out}=\tilde{\mathcal{N}}_{\rm out}$ with $T\gg1$. Assume the foliation condition and  \eqref{condG}  for each one of them.  
Then  $\Gamma=\tilde \Gamma$, and $c_{\rm s}=\tilde c_{\rm s}$,  $c_{\rm p}=\tilde c_{\rm p}$ in $M_+$, and $c_{\rm f}=\tilde c_{\rm f}$ in $M_-$. Also,  if $c_{\rm p}\not=2c_{\rm s}$ in $M_+$, then $\rho_{\rm s}=\tilde\rho_{\rm s}$ in $M_+$. 
\end{theorem}

\begin{proof}
The first part of the theorem, concerning the recovery of $\Gamma$ and the elastic parameters in $M_+$ follows directly from \cite[Lemma~10.1]{stefanov2020transmission}. The only difference is that we have the ND instead of the DN map but Dirichlet data can be easily converted to Neumann and vice-versa, microlocally, by ellipticity arguments.

We prove below  $c_{\rm f}=\tilde c_{\rm f}$ in $M_-$. We follow the proof of \cite[Lemma~10.2]{stefanov2020transmission} here.

Choose two points $x$, $y$, on $\Gamma$ connected by a unit speed geodesic $\gamma_0$  of $c_{\rm f}^2 g$  hitting $x$ and $y$ at times $t_1$ and $t_2$, respectively, see Figure~\ref{fig:IP}. We chose a microlocal solution in $M_-$ concentrated near $\gamma_0$.
\begin{figure}[h] 
	\includegraphics[page = 7]{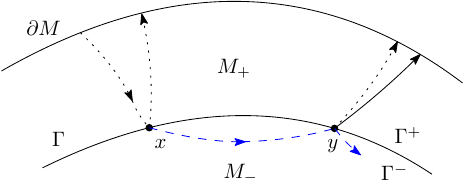}
	\caption{Illustration to the proof of Theorem~\ref{thm_inverse}. Here, only an s waves hits $x$ from $M_+$, and s wave only reflects but we could have two p waves in addition as well. 
	}
	\label{fig:IP}
\end{figure}
 The projected singularities near $x$ are either in the hyperbolic-hyperbolic region, see Section~\ref{ssec:hh} or in the mixed-hyperbolic one, see Section~\ref{ssec:mh} with the exception of a set of geodesics of measure zero (giving rise to tangential rays in $M_+$). In either case, the fluid side is controllable from the solid one: one can choose incoming and outgoing solutions at $x$ to have the refracted fluid wave to be the prescribed one at $x$, and no incoming one at $x$ from $M_-$, on principal level. We extend the outgoing waves back to $M_+$ a bit outside $M$, where we extend the coefficients (of both systems, equally) in a smooth way, and cancel possible reflections at $\partial M$ by sending outgoing waves with opposite Neumann data back to $M_+$,  on principal level. The analysis in  \cite{stefanov2020transmission} shows that this is possible. 

When the wave reaches $y$, it will create a reflected fluid wave back to $M_-$ and two, or one refracted waves into $M_+$. At least one will be non-zero. 
We kill  possible reflections as above, in other words, we may assume that they leave $M$.

Consider the second, ``tilded'' system now. We apply the same Neumann condition and assume the same Dirichlet data. 
By \cite{Rachele_2000}, $ (\rho_{\rm s}, \mu_{\rm s}, \lambda_{\rm s})$ and $ (\tilde \rho_{\rm s}, \tilde \mu_{\rm s}, \tilde \lambda_{\rm s})$   coincide at $\bo$ at infinite order. By the first step, $c_{\rm p} = \tilde c_{\rm p}$, $c_{\rm s} = \tilde c_{\rm s}$ in $M_+$. The rays leading to $x$ for both systems would be the same (and therefore, the ``tilded'' ones would really hit $x$ as well) but the amplitudes are not necessarily equal. The energy of the two waves combined would be positive however. 
The $\td{c}_{\rm f}g$ geodesic $\tilde\gamma_0$ in $M_-$ may not hit $\Gamma$ again at $y$ a priori but we can do time reversal from $\bo$ back to $\Gamma$ to see that in fact, it does; and that happens at the same time $t_2$. There might be other rays hitting $\bo$ since in $M_-$, there is a reflected ray which will eventually refract; and some of them may even hit $\bo$ earlier. This is not a problem since we can identify $y$ on $\Gamma_+$ as the first point at which a singularity comes back.  

This argument proves that the travel time between $x$ and $y$ is the same for both $c_{\rm f}$ and $\tilde c_{\rm f}$.  This is true for pairs $(x,y)\in \Gamma\times\Gamma$ away from a zero measure set. We can extend it for all $(x,y)\in \Gamma\times\Gamma$ by continuity.  Therefore, $c_{\rm f}= \tilde c_{\rm f}$ in $M_-$ by \cite{MR3454376}.
\end{proof}

\appendix\section{Well Posedness and Justification of Parametrix in the Solid-Solid and Fluid-Fluid Case}\label{appendix_a}

Although the main focus of the present paper is the transmission problem at the interface between a solid and a fluid, in this appendix we discuss how with similar methods to the ones used in Sections \ref{sec:well_posedness} and \ref{ssec:justification} one can prove the justification of a parametrix in the case of two solids or two fluids being in contact.
Such a parametrix for the solid-solid case was constructed in \cite{stefanov2020transmission}, but it was not shown that the difference from an actual solution is smooth all the way to the interface.
As an intermediate step, we also discuss well posedness.
The results in this appendix in the solid-solid case were mentioned in \cite{hansen2021propagation}, though without detailed proofs.
The justification of the parametrix again follows  \cite{MR535598}. In the fluid-fluid case (with the pressure satisfying an acoustic equation on both sides of the interface), it also follows from \cite{MR1183346}, which used different methods.

\subsection{Well Posedness}\label{sssec:ss_ff}

Suppose that instead of our original system (\ref{sys_1}-\ref{tr_3}) we either had both $M^\pm$ occupied by solids or had both of them occupied by inviscid fluids.
Below we assume that the setup regarding the geometry of the domains and the metric is as described in Section \ref{sec:setup_and_model}.
So suppose we had one of the following two systems with transmission conditions:
\begin{subequations}
  \begin{align}
    \p_t^2u_1^\pm=(\r_1^\pm)^{-1}E^\pm_1 u_1^\pm&\quad \text{ on }M^\pm \times \R,\label{eq:s-s1} \\*
    N^+(u_1^+)=N^-(u_1^-)&\quad \text{ on }\G\times \R,\label{eq:s-s2}\noeqref{eq:s-s2} \\*
    u_1^+=u_1^-&\quad \text{ on }\G\times \R,\label{eq:s-s3}\noeqref{eq:s-s3} \\*
    N^+(u^+_1)=0&\quad \text{ on }\p M\times \R\label{eq:s-s4},
  \end{align}
\end{subequations}
or 
\begin{subequations}
  \begin{align}
    \p_t^2p_2^\pm=\l_2^\pm\div ((\r_2^\pm) ^{-1}\n p_2^\pm)&\quad \text{ on }M^\pm \times \R,\label{eq:f-f1}\noeqref{eq:f-f1} \\*
    (\r_2^-) ^{-1}\p_\nu p_2^-=(\r_2^+) ^{-1}\p_\nu p^+_2 &\quad \text{ on }\G\times \R,\label{eq:f-f2}\noeqref{eq:f-f2} \\*
    p_2^+=p_2^-&\quad \text{ on }\G\times \R,\label{eq:f-f3}\noeqref{eq:f-f3} \\*
    p_2^+=0&\quad \text{ on }\p M\times \R,\label{eq:f-f4}\noeqref{eq:f-f4}
  \end{align}
\end{subequations}
both subject to Cauchy data at $t=0$, corresponding to the solid-solid and the fluid-fluid case respectively.

Equations (\ref{eq:s-s1}-\ref{eq:s-s4}) are a system for the (complexifications of the) vector valued displacements in the two solids; by choosing global coordinates we assume for simplicity that $u^\pm$ is $\C^3$-valued.
In \eqref{eq:s-s1}, the elastic wave operator is as described in Section \ref{sec:setup_and_model} on each side of the interface $\G$, with Lam\'e parameters $\l_1^\pm, $ $\mu_1^\pm$ which are positive and smooth all the way to $\G$ and $\p M$ but not necessarily matching at $\G$.
The transmission conditions  \eqref{eq:s-s2}-\eqref{eq:s-s3} guarantee continuity of  traction and displacement across the interface respectively, whereas \eqref{eq:s-s4} stands for vanishing of the traction across the surface of contact of the solid and vacuum (or air by approximation).
The densities $\r_1^\pm$ are positive and smooth up to the interface/boundary but might jump at $\G$.

The system (\ref{eq:f-f1}-\ref{eq:f-f4}) describes the scalar valued acoustic pressure for inviscid fluids on $M^\pm$.
The transmission condition \eqref{eq:f-f3} originates from continuity of traction at the interface $\G$, whereas \eqref{eq:f-f2} from continuity of the normal component of displacement at $\G$ (recall that according to \eqref{sys_3} we have $\p_t v_2^\pm=-(\r_2^\pm)^{-1}\n p_2^\pm$, where $v_2^\pm$ stands for the velocity field of the fluid in $M^\pm$).
\eqref{eq:f-f4} stands for vanishing of traction across the surface of contact of the fluid and vacuum.
Again, $\l_2^\pm$ and $\r_2^\pm$ are smooth and positive all the way to $\p M $ and $\G$, generally not matching at $\G$.

We unify the presentation by writing, for $j=1,2$,
\begin{subequations}
  \begin{align}
    \p_t^2z_j^\pm=P_j^\pm  z_j^\pm&\quad \text{ on }M^\pm \times \R,\label{eq:un_1}\noeqref{eq:un_1} \\*
    \mathfrak{B}_{j,\nu}^+ z_j^-=\mathfrak{B}_{j,\nu}^- z^+_j &\quad \text{ on }\G\times \R,\label{eq:un_2}\noeqref{eq:un_2} \\*
    z_j^+=z_j^-&\quad \text{ on }\G\times \R,\label{eq:un_3}\noeqref{eq:un_3} \\*
    \begin{split}
    \mathfrak{B}_{1,\nu}^+ z_1^+=0\text{ or }z_2^+=0&\quad \text{ on }\p M\times \R,\\*
     \text{ corresponding to }&j=1,2 \text{ in }\eqref{eq:un_1}\text{-}\eqref{eq:un_3},
    \end{split}\label{eq:un_4}
  \end{align}
\end{subequations}
where 
\begin{align*}
   z_1^\pm=&u_1^\pm, &z_2^\pm=&p^\pm_2, & P_1^\pm=&(\r_1^\pm)^{-1}E_1^\pm,\\
    P_2^\pm=&\l_2^\pm\div((\r_2^\pm)^{-1}\n\cdot ), & \mathfrak{B}_{1,\nu}^\pm=&N^\pm, & \mathfrak{B}_{2,\nu}^\pm=&(\r_2^\pm)^{-1}\p_\nu.
\end{align*}
Note that $P_j^\pm$ is an elliptic operator for $j=1,2$ (matrix valued for $j=1$).

We view $P_{j,0}=\begin{pmatrix}
  P_j^+& 0\\
  0&P_j^-
\end{pmatrix}$ as an unbounded operator on $$L^2(M^+,d\mu_j^+;\C^{m(j)})\times L^2(M^-,d\mu_j^-;\C^{m(j)}),$$ where  
$$d\mu_1^\pm=\r_1^\pm dv_g,\quad  d\mu_2^\pm=(\l_2^\pm)^{-1} dv_g, \quad  m(1)=3,\text{ and }m(2)=1,$$
 with domain
\begin{equation}
\begin{aligned}
  D({P_{j,0}})=\big\{(z_j^+,z_j^+)\in C^\infty(\oM^+;\C^{m(j)})\times C^\infty(\oM^-;\C^{m(j)}): \; z_j^+=z_j^- \text{ and } \mathfrak{B}_{j,\nu}^+ z_j^-=\mathfrak{B}_{j,\nu}^- z^+_j\text{ on }\G,&\\
   \mathfrak{B}_{1,\nu}^+ z_1^+=0 \text{ or } z_2^+=0\text{ on } \p M \text { corresponding to }j=1 \text{ or }j=2\big\}.&
  \end{aligned}
\end{equation}
 By the transmission and boundary conditions, $-P_{j,0}$ is symmetric and semibounded below on its domain, hence $P_{j,0}$ admits a self-adjoint extension $P_j$ with domain $D(P_j)$.
 As before, to construct the domain first complete $D({P_{j,0}})$ in the squared norm 
 \begin{equation}
 \|(z_j^+,z_j^-)\|_{q_j}^2=\|z_j^+\|_{q_j^+}^2+\|z^+_j\|_{L^2(M^+ ,d\mu^+_j)}^2+ \|z_j^- \|_{q_j^-}^2+\|z^-_j\|_{L^2(M^- ,d\mu_j^-)}^2, 
 \end{equation}
where $\|\cdot \|_{q_j^\pm}$ are the seminorms induced on $C^\infty(\oM^\pm;\C^{m(j)})$ by the  quadratic forms
 \begin{subequations}
\begin{align}
  &q_1^\pm({z}_1^\pm,{w}_1^\pm)=
  (\Div z_1^\pm,\Div w_1^\pm)_{L^2(M^\pm,\l_1^\pm dv_g)}
  +(d^{\rm s} z_1^\pm,d^{\rm s}w_1^\pm)_{L^2(M^\pm,2\mu_1^\pm dv_g)},\\
  &q_2^\pm(z_2^\pm,w_2^\pm)=(\n z^\pm_2,\n w^\pm_2)_{L^2(M^\pm,\r_2^\pm dv_g)}.
\end{align}
\end{subequations}

\smallskip
\begin{lemma}\label{lm:dom_q_j}
Denote the completion of $D({P_{j,0}})$ in $\|\cdot \|_{q_j}$ by $D(q_j)$, $j=1,2$.
We have 
\begin{subequations}
\begin{align}
  &\qquad D(q_1)=\calH^1_{1,\rm{tr}}:=\{(z^+,z^-)\in H^1(M^+;\C^3)\times H^1(M^-;\C^3):\t (z_1^+)=\t(z^-_1)\}\label{eq:dom_q_1} \text{ and }\\
    &\qquad D(q_2)=\calH^1_{2,\rm{tr}}:=\{(z^+,z^-)\in H^1(M^+)\times H^1(M^-):\t (z_2^+)=\t(z^-_2)\text{ and }\t'(z^+_2)=0\},
  \label{eq:dom_q_2}
\end{align}
\end{subequations}
where $\t$, $\t'$ are the traces at $\G$ and $\p M$ respectively.\footnote{Strictly speaking, for each $j=1,2$ we have two trace operators corresponding to $\G$, with different domains, mapping $C^\infty(M^\pm;\C^{m(j)})\to C^\infty(\G;\C^{m(j)}) $ and extending continuously $H^1(M^\pm;\C^{m(j)})\to H^{1/2}(\G;\C^{m(j)}) $.
However, we will not differentiate between them in the notation and it will be clear from the argument which one is used.}
The subscript ``$\rm{tr}$'' stands for transmission.
\end{lemma}
\begin{proof} The fact that $D(q_j)\subset H^1(M^+;\C^{m(j)})\times H^1(M^-;\C^{m(j)})$ follows from the equivalence of $\|\mathbf{z}_j \|_{q_j}^2$ with the squared norm $\|z^+_j\|_{H^1(M^+)}^2+\|z^-_j\|_{H^1(M^-)}^2$, where we wrote $\bz_j=(z^+_j,z^-_j)$ (in the case $j=1$ the equivalence of norms follows from Korn's inequality).
The transmission/boundary conditions in \eqref{eq:dom_q_1}-\eqref{eq:dom_q_2} hold by the trace theorem, since they do so for elements of $D(P_{j,0})$.

For the other inclusion, suppose that $\bz_j=(z_j^+,z_j^-)\in\calH_{j,\rm{tr}}^1$ is given and we seek an element in $D(P_{j,0})$ close to it.
The transmission condition at $\G$ guarantees that upon defining ${z}_j=\begin{cases}
  z_j^+\text { on }M^+\\
  z_j^-\text { on }M^-
\end{cases}$, we have $z_1\in H^1(M;\C^3)$ and $z_2\in H_0^1(M)$. Thus given $\e>0$ we can find $X_1\in C^\infty(\oM;\C^3)$, $X_2\in C_c^\infty(M)$ such that $\|z_j-X_j\|_{H^1(M)}\leq \e$. Setting
$\mathbf{X}_j=(X_j \big|_{M^+},X_j \big|_{M^-})$,
\begin{equation}
  \| \bz-\mathbf{X}_j\|_{q_j}\leq C\left(\|z^+_j-X_j \big|_{M^+}\|_{H^1(M^+)}+\|z^-_j-X_j \big|_{M^-}\|_{H^1(M^-)}\right)\leq  C\|z_j-X_j\|_{H^1(M)}\leq C\e.
\end{equation}

Finally, since $\mathbf{X}_j$ does not generally satisfy the requisite Neumann type transmission conditions, 
adjust it by finding
$\td{X}_j^+\in C^\infty(\oM^+)$ with $\td{X}_j^+\big|_{\p M^+}=0$, $\mathfrak{B}_{j,\nu}(X_j \big|_{M^+}+\td{X}_j^+)=\mathfrak{B}_{j,\nu}(X_j \big|_{ M^-})$ on $\G$ and $\mathfrak{B}_{1,\nu}(X_1 \big|_{M^+}+\td{X}_1^+)=0$ on $\p M$ if $j=1$.
By shrinking its support it can be arranged that $\|\td{X}_j^+\|_{H^1(M^+)}<\e$, implying that $\|\bz_j-(\mathbf{X}_j+(\td{X}_j^+,0))\|_{q_j}\leq C \e$ with $\mathbf{X}_j+(\td{X}_j^+,0)\in D(P_{j,0})$ and thus showing the claim.
\end{proof}

We now have:

\begin{proposition}\label{prop:estimate_ff}
For $j=1,2$, if $\mathbf{z}_j=(z^+_j,z^-_j)\in D(P_j)$ with $P_j^\pm u^\pm \in H^k(M^\pm)$ for $k=0,1,2\dots $ then we have
\begin{equation}\label{eq:estimate_ss_ff}
\begin{aligned}
  &\|z_j^+\|^2_{H^{k+2}({M}^+)}+\| z_j^-\|^2_{H^{k+2}({M}^-)} \\
  &\qquad \leq C\big(\|P_j^+z_j^+\|_{H^{k}(M^+)}^2+\|P_j^-z_j^-\|_{H^{k}(M^-)}^2 +\| z_j^+\|_{H^1(M^+)}^2+\|  z_j^-\|_{H^1(M^-)}^2   \big).
\end{aligned}
\end{equation}
If $\mathbf{z}_j=(z_j^+,z_j^-)\in D(P_j^n)$, $n\geq 1$ then $z_j^\pm\in H^{2n}(M^\pm,\C^{m(j)})$.
\end{proposition}
\begin{proof}
The operators $P^\pm_j$ are all elliptic and coercive on $H^1$, with coefficients smooth down to the interface $\G$ and the boundary $\p M$.
Now suppose that $\bz_j\in D(P_j)$, $j=1,2$, implying  that $\t(z^+_j)-\t(z^-_j)=0$ and $\t'(z^+_2)=0$ if $j=2$.
Moreover, 
the
integration by parts property
\begin{equation}\label{eq:ss_ff_ibp}
  \sum _{\bullet =\pm}\big( ( P^\bullet_j z^\bullet_j,w_j^\bullet)_{L^2(M^\bullet,d\mu_j^\bullet)}+q^\bullet_j(z^\bullet_j,w^\bullet_j)\big)=0, \quad (z_j^+,z_j^-)\in D(P_j),\quad  (w_j^+,w_j^-)\in D(q_j)
\end{equation}
implies $\mathfrak{B}_{j,\nu}^+(z_j^+)-\mathfrak{B}_{j,\nu}^-(z_j^-)=0$ and $\mathfrak{B}_{1,\nu}^+(z^+)=0$  (those quantities are a priori defined weakly as elements of  $ H^{-1/2}(\G;\C^{m(j)})$ and $H^{-1/2}(\p M;\C^3)$ respectively, see e.g. \cite[Lemma 4.3]{McLean-book}).
Hence \eqref{eq:estimate_ss_ff} follows from Theorems 4.18 and 4.20 of \cite{McLean-book}.

For the second statement, if $\bz=(z_j^+,z_j^-)\in D(P_j^n)$, $n\geq 2$ then \eqref{eq:estimate_ss_ff} for $k=0$ implies that $(P^\pm_j)^{n-1} z_j^\pm\in H^2(M^\pm;\C^{m(j)}) $.
Then, using \eqref{eq:estimate_ss_ff} for $k=2$  and $z^\pm_j$ replaced by $(P^\pm_j)^{n-2} z_j^\pm$ we find that $(P^\pm_j)^{n-2} z_j^\pm\in H^4(M^\pm;\C^{m(j)}) $. 
Proceeding inductively for $n-1$ steps, we find that $(z_j^+,z_j^-)\in D(P_j^n)$  implies $P_j^\pm z_j^\pm\in H^{2n-2}(M^\pm;\C^{m(j)})$.
Then the claim follows from \eqref{eq:estimate_ss_ff} again applied for $k=2n$.
\end{proof}

As a corollary we obtain the following:
\begin{corollary}\label{cor:dom_P_j}
  The domain of the self-adjoint operator $P_j$ for $j=1,2 $ is given by 
  \begin{subequations}
  \begin{align}
  \begin{split}
    D(P_1)=\{(z_1^+,z_1^-)\in H^2(M^+;\C^3)\times H^2(M^-;\C^3): \t(z_1^+)=\t(z_1^-)\text{ on }\G,\quad\\
     \mathfrak{B}_{1,\nu}^+(z_1^+)=\mathfrak{B}_{1,\nu}^- (z_1^-) \text{ on }\G\text { and }  \mathfrak{B}_{1,\nu}^+(z_1^+)=0 \text{ on }\p M \} \quad \text{ and }
     \end{split}\label{eq:dom_P_1}  
     \\
     \begin{split}
     D(P_2)=\{(z_2^+,z_2^-)\in H^2(M^+)\times H^2(M^-): \t(z_2^+)=\t(z_2^-)\text{ on }\G,\quad\\
     \mathfrak{B}_{2,\nu}^+(z_2^+)=\mathfrak{B}_{2,\nu}^- (z_2^-) \text{ on }\G   \text{ and } \t'(z_2^+)=0 \text { on }\p M\}.
     \end{split}\label{eq:dom_P_2}
  \end{align}
  \end{subequations}
\end{corollary}
\begin{proof}
The regularity of elements in $D(P_j)$ follows from Proposition \ref{prop:estimate_ff}.
The transmission/boundary conditions follow from the inclusion $D(P_j)\subset \calH_{j,\rm{tr}}^1$ and  the integration by parts property \eqref{eq:ss_ff_ibp}.

Conversely, any element $(z_j^+,z_j^-)$ in the right hand side of \eqref{eq:dom_P_1}-\eqref{eq:dom_P_2} lies in $\calH_{j,\rm{tr}}^1$ and satisfies  \eqref{eq:ss_ff_ibp} for $(w_j^+,w_j^-)\in \calH_{j,\rm{tr}}^1$, implying that 
\begin{equation}
    |q^+_j(z^+_j,w^+_j)+q^-_j(z^-_j,w^-_j)|\leq C\Big(\|w^+_j\|_{L^2(M^+,d\mu_j^+)}^2+\|w^-_j\|_{L^2(M^-,d\mu_j^-)}^2\Big)^{1/2},
\end{equation}
thus $(z_j^+,z_j^-)\in D(P_j)$.
\end{proof}

Finally, the Sobolev Embedding Theorem and Proposition \ref{prop:estimate_ff}  yield the following:
\begin{corollary}\label{cor:P^n}
  For $j=1,2 $, if $\bz_j=(z_j^+,z_j^-)\in D(P_j^n) $ for all $n\geq 1$, then $z_j^\pm\in C^\infty(\oM^\pm;\C^{m(j)})$.
\end{corollary}

\subsection{Parametrix Justification}
\label{sssec:just_ff_ss}

Using the techniques described in Sections \ref{sec:acoustic} and \ref{sec:elastic} one can construct parametrices for (\ref{eq:f-f1}-\ref{eq:f-f4}) and (\ref{eq:s-s1}-\ref{eq:s-s4}) respectively (see \cite{stefanov2020transmission}).
With the combined presentation used in \eqref{eq:un_1}-\eqref{eq:un_4}, the difference between an actual solution and a parametrix satisfies 
\begin{equation}\label{eq:nonhomog_system_ss_ff}
  \begin{cases}
    \p_t^2z_j^\pm-P_j^\pm  z_j^\pm=F_j^\pm&\quad \text{ on }M^\pm \times \R,\\
    \mathfrak{B}_{j,\nu}^+ z_j^+-\mathfrak{B}_{j,\nu}^- z^-_j=f_j &\quad \text{ on }\G\times \R,\\
    z_j^+-z_j^-=g_j&\quad \text{ on }\G\times \R, \\
    \mathfrak{B}_{1,\nu}^+ z_1^+=h_1&\quad \text{ on }\p M\times \R \text{ if } j=1,\\
    z_2^+=h_2&\quad \text{ on }\p M\times \R \text{ if } j=2,
     \\
     z_j^\pm=0 &\quad\text{ for }t\ll 0,
  \end{cases}
\end{equation}
where for $j=1,2$, $F_j^\pm\in C^\infty(\oM^\pm\times \R;\R^{m(j)})$, $f_j$, $g_j\in C^\infty(\G\times \R;\R^{m(j)})$, $h_j\in C^\infty(\p M\times \R;\R^{m(j)})$ and we recall the notation $m(1)=3$, $m(2)=1$.

As in Section \ref{ssec:justification}, by the Cauchy-Kovalevskaya method and Borel's lemma, our task reduces to showing smoothness up to the boundary/interface for the solutions $v^\pm_j$ of the system
\begin{equation}\label{eq:diff_ss_ff}
  \begin{cases}
    \p_t^2v_j^\pm-P_j^\pm  v_j^\pm=\td{F}_j^\pm&\quad \text{ on }M^\pm \times \R,\\*
    \mathfrak{B}_{j,\nu}^+ v_j^+-\mathfrak{B}_{j,\nu}^- v^-_j=0 &\quad \text{ on }\G\times \R,\\*
    v_j^+-v_j^-=0&\quad \text{ on }\G\times \R,\\*
    \mathfrak{B}_{1,\nu}^+ v_1^+=0&\quad \text{ on }\p M\times \R \text{ if }j=1,\\
    v_2^+=0 &\quad \text{ on }\p M\times \R \text{ if }j=2,\\
     v_j^\pm=0 &\quad\text{ for }t\ll 0,
  \end{cases}
\end{equation}
where $\td{F}_j^\pm$ are smooth and vanish to infinite order at  $\G$ and at $\p M$.
Therefore, by Corollary \eqref{cor:dom_P_j} we have that $\td{\mathbf{F}}_j(s)=(\td{F}_j^+(\cdot,s),\td{F}_j^-(\cdot,s))\in D(P^k)$ for all $k\geq 1$ and $s\in \R$.
Thus \eqref{eq:diff_ss_ff} can be solved using Duhamel's formula, namely
\begin{equation}
  \bv(t)=\big(v_j^+(\cdot,t),v_j^-(\cdot,t)\big)= \int _{-\infty}^t\frac{\sin\big(\sqrt{-P_j}(t-s)\big)}{\sqrt{-P_j}}\mathbf{F}_j(s)ds.
\end{equation}
The functional calculus  implies that $\bv\in C^\infty(\R; D(P_j^k))$ for all integers $k\geq 0$, so by Corollary \ref{cor:P^n} we obtain that $ v_j^\pm\in C^\infty(\oM\times\R;\C^{m(j)})$.

\section{Proofs for Section \ref{ssec:domain}}\label{appendix_b}

In this appendix we prove Lemma \ref{lm:h_1}, Proposition \ref{prop:estimate} and Corollary \ref{cor:estimate}. 
In what follows, given $\bu=(u^+,u^-)\in \calH_{\div,\rm{tr}}^1$ (see \eqref{eq:H1div}),  it will be convenient to have a procedure for producing a decomposition of ${u}^-$ into a divergence free and a potential part.
The potential part generally possesses higher regularity than $u^-$, as we see below.
Start by solving (up to a constant) 
\begin{equation}\label{eq:neumann_problem}
    \Delta \w = \div u^-\text { on }M^-,\quad \p_\nu \w =\t(u^+)\cdot \nu  \text{ on }\Gamma
 \end{equation}
 to find $\w\in H^2(M^-)$ such that $u^-=\n \w+z$, where $\div z=0$ and $z\cdot \nu \big|_\G=0$ (a priori defined in a weak sense).
An $\w$ solving \eqref{eq:neumann_problem} exists due to the transmission conditions, which guarantee that $\int_{M^-}\div u^- dv_g=-\int_\G \t(u^+)\cdot\nu dA$.
The regularity of $\w$ follows e.g. by \cite[Theorem 4.18 (ii)]{McLean-book}, since $\p_\nu \w=\t(u^+)\cdot \nu\in H^{1/2}(\G)$ and $\div u^-\in L^2(M^-)$.
Henceforth we set
\begin{equation}\label{eq:potential_part}
    \td{u}^-:=u^--z=\n \w\in H^1(M^-;\C\otimes TM).
 \end{equation}  

For the solenoidal (divergence free) part we can work somewhat more generally, and start with an element $u^-\in H_{\div}^1(M^-;\C\otimes TM)$ (so it need not be coupled with a vector field in $M^+$). 
We can now find $\phi\in H^1(M^-)$ satisfying $\Delta \phi=\div u^-\in L^2(M^-)$ on $M^-$, $\p_\nu \phi=\t( u^- \cdot   \nu)\in H^{-1/2}(\G)$ on $\G$. 
The compatibility condition for the  Neumann problem is automatically satisfied here by the divergence theorem.
Then we set 
\begin{equation}\label{eq:solenoidal}
    \Pi u^-= u^--\n \phi\in L^2(M^-,\C\otimes TM).
\end{equation}
Then $\div \Pi u ^-=0$ on $M^-$ and $\Pi u^-\cdot \nu =0$ on $\G$ (in a weak sense).
The advantage of defining $\Pi$ independently of the existence of a coupled vector field in $M^+$ is that it becomes  an orthogonal projector on the subspace of divergence free vector fields in $L^2(M^-,dv_g;\C\otimes TM)$; the reason we used the coupling in defining $\td{u}^-$ is the gain in regularity we were able to obtain as a result of it.
Given $\bu=(u^+,u^-)\in \calH_{\rm{div,tr}}^1$, the decomposition of $u^-$ can be written as $u=\td{u}^-+\Pi u^-$.

\begin{proof}[Proof of Lemma \ref{lm:h_1}]

For $\bu\in D(P_0)$ we have
\begin{equation}\label{q_norm}
    \|\bu \|_q^2=\|\Div(u^+)\|^2_{L^2(M^+,\l_{\rm s}dv_g)}+\|d^{\rm s}(u^+)\|^2_{L^2(M^+,2\mu_{\rm s} dv_g)}+\|\div u^-\|^2_{L^2(M^-,\l_{\rm f}dv_g)}+\| \bu\|^2_{L^2}.
\end{equation}
By Korn's inequality, the squared norm $\|\Div(u^+)\|^2_{L^2(M^+)}+\|d^{\rm s}(u^+)\|^2_{L^2(M^+)}+\|u^+\|_{L^2(M^+)}^2$ is equivalent to $\|u^+\|_{H^1(M^+)}^2$, therefore
 $D(q)\subset H^1(M^+;\C\otimes TM)\times H^1_{\div}(M^-;\C\otimes TM)$.
The continuous dependence of $\| \t(u^+)\cdot \nu\|_{H^{1/2}(\G)}$ and $\|\t(u^-\cdot \nu) \|_{H^{-1/2}(\G)}$ on $\|u^+\|_{H^1(M^+)} $ and $\|u^-\|_{H^1_{\div}(M^-)}$ respectively implies that the transmission condition in \eqref{eq:H1div} is satisfied and hence $D(q)\subset \calH_{\div,\rm{tr}}^1$ (so one also has $\t(u^-\cdot \nu) \in H^{1/2}(\G) $).

For the converse, assume $\bu=(u^+,u^-)\in \calH^1_{\div,\text{tr}}$ is given; we will show that we can find an element of $ D(P_0)$ arbitrarily close to it in $\|\cdot \|_q$.
Let $\td{u}^-\in H^1(M^-;\C\otimes TM)$ and $\Pi u^-\in L^2(M^-;\C\otimes TM)$ be as in \eqref{eq:potential_part} and \eqref{eq:solenoidal}.
Further, consider semigeodesic local coordinates $(x_1,x_2,x_3)$ in a neighborhood $U$ of a point in $\G$ such that $x_3=0 $ on $\G$, $\p_{x_3}\big|_{\G}=-\nu$ %
 and $\p_{x_3}\cdot \p_{x_j} \big|_{\G}=0$ for $j=1,2$, and write $u^+_j=dx_j (u^+)$, $\td{u}^-_j=dx_j(\td{u}^-)$. \footnote{Here we are not using the convention of writing upper indices for the components of a vector field.}
We deal with the potential and the divergence free part of $u^-$  separately.
To handle the former, for $j=1,2,3$ we will approximate $(u^+_j,\td{u}_j^-)$ in $H^1(M^+)\times H^1(M^-)$ by pairs of smooth functions; specifically for $j=3$ we will use the transmission condition $u^+_3 \big|_{\G}=\td{u}_3^-\big|_{\G}$ satisfied by $(u^+_3,\td{u}_3^-)$ to ensure that the  approximating pair satisfies it too.
For the latter, we will first approximate $\Pi u^-$ in $L^2$ and then only use the divergence free part of the approximating vector field to build the vector field approximating $u^-$.

Consider $\varphi\in C^\infty_c(U)$ and write $U^\pm:=U\cap M^\pm$.
Then $\varphi u^+_j\in H^1(U^+)$, $\varphi\td{u}^-_j\in H^1(U^-)$ for all $j$ and  they vanish on $\p U$; moreover, the function defined on $U$ as 
  $\begin{cases}
    \phi u^+_{3}\text{ on }U^+\\
    \phi\td{u}^-_{3} \text{ on }U^-
 \end{cases}$ lies in $ H^1_0 (U)$, by \cite[Exercise 4.5]{McLean-book}.
 Thus given $\e>0$ we can find functions $X_3\in C_c^\infty (U)$, $X^\pm_j\in C_c^\infty ({U}\cap \o{M}^\pm)$, $j=1,2$ such that 
 \begin{align}
    \mathbf{X}_U=&(X_U^+,X_U^-)=\big(\sum_{j=1}^2 X_j^+\p_{x_j}+X_3\p_{x_3},\sum_{j=1}^2 X_j^-\p_{x_j}+X_3\p_{x_3}\big)\\*
    &\in  C_c^\infty(U\cap \o{M}^+;\C\otimes T\o{M})\times C_c^\infty (U\cap \o{M}^-;\C\otimes T\o{M})\text{ with }X_U^+\cdot \nu \big|_{\Gamma}=X_U^-\cdot \nu \big|_{\Gamma}
 \end{align}
 and  $\|\varphi u^+-X^+_U\|^2_{H^1(U^+)}+\|\varphi\td{u}^--X_U^-\|^2_{H^1( U^-)}\leq \e
$.
In coordinate neighborhoods which do not intersect $\G$ we can construct smooth approximations to $(u^+,\td{u}^-)$ in a similar, though simpler, fashion.
Using a partition of unity, we find 
\begin{equation}
 \mathbf{X}=(X^+,X^-)\in    C^\infty(\o{M}^+;\C\otimes T\o{M})\times C^\infty (\o{M}^-;\C\otimes T\o{M})\text{ with }X^+\cdot \nu \big|_{\Gamma}=X^-\cdot \nu \big|_{\Gamma}    
 \end{equation}
  such that  
 \begin{equation}
    \left( \|{u}^+-X^+\|_{H^1(M^+)}^2+\|\td{u}^--X^-\|_{H^1(M^-)}^2\right)\leq \e.  
\end{equation}
To deal with the divergence free part of $u^-$, we find $Y^-\in C^\infty(\oM^-;\C\otimes TM)$ which satisfies $\|\Pi u^- -Y^-\|_{L^2(M^-)}^2\leq \e$.
Since $\Pi$ is an orthogonal projector, we have  
\begin{equation}
    \|\Pi u^- -\Pi Y^-\|_{L^2(M^-)}^2=\|\Pi(\Pi u^- -Y^-)\|_{L^2(M^-)}^2\leq \|\Pi u^- -Y^-\|_{L^2(M^-)}^2 \leq \e.
\end{equation}
Now set $\mathbf{X}_1=\mathbf{X}+(0,\Pi Y^-)$;  by  construction of $\Pi$ we have that $X^+\cdot \nu \big|_{\G}=(X^-+\Pi Y^-)\cdot \nu \big|_{\G}$. 
Now
\begin{equation}
    \begin{aligned}
    \|\bu-\mathbf{X}_1\|_q^2&\leq C\left( \|u^+-X^+\|_{H^1(M^+)}^2+\|\td{u}^-+\Pi u^--(X^-+\Pi Y^-)\|_{H^1_{\div}(M^-)}^2\right)\\
    &\leq C\left(\|{u}^+-X^+\|_{H^1(M^+)}^2+\|\td{u}^--X^-\|_{H^1_{\div}(M^-)}^2+\|\Pi u^--\Pi Y^-\|_{L^2(M^-)}^2\right)\\
    &\leq C\left(\|{u}^+-X^+\|_{H^1(M^+)}^2+\|\td{u}^--X^-\|_{H^1(M^-)}^2+\|\Pi u^--\Pi Y^-\|_{L^2(M^-)}^2\right)\leq C\e.\\
\end{aligned}
\end{equation}
The vector field $\mathbf{X}_1$ we constructed does not necessarily satisfy all of the requisite transmission and boundary conditions to lie in $D(P_0)$.
Hence we adjust $X^+$ by adding a vector field $\td{X}^+\in C^\infty(\oM^+;\C\otimes T\oM)$ satisfying $$\td{X}^+=0\; \text{ and }N(X^++\td{X}^+)=\l _{\rm f} (\div X^-)\nu  \text{ on }\G,\quad  N(X^++\td{X}^+)=0\text { on } \p M,$$ and supported in a sufficiently small neighborhood of $\p M^+$ to ensure that $\|\td{X}^+\|_{H^1(M^+)}^2\leq \e$.
We find that $\mathbf{X}_2=(X^++\td{X}^+,X^-+\Pi Y^- )\in D(P_0)$ and $\|\bu-\mathbf{X}_2\|_q^2\leq C \e$, as claimed. 
\end{proof}

The proofs for Proposition \ref{prop:estimate} and Corollary \ref{cor:estimate} below closely follow those of elliptic regularity estimates in \cite[Ch. 4]{McLean-book}, though the difficulty here is the lack of ellipticity of $P^-$.
We will use difference quotients:
for a function $w\in L^2(\R^n)$ let 
$$\delta_{\ell,h}w(x)=\frac{1}{h}\big(w(x+he_\ell)-w(x)\big),\quad  \ell=1,\dots, n,$$
 where $e_\ell$ is the $\ell$-th standard unit vector.
If $\p_{x_\ell} w\in L^2(\R^n)$, then by \cite[Lemma 4.13]{McLean-book}, $\| \delta_{\ell,h}w\|_{L^2(\R^n)}\leq C\| \p_{x_\ell} w\|_{L^2(\R^n)}$ for  $h\in \R$, and $\delta_{\ell,h}w\overset{h\to 0}{\to} \p_{x_\ell} w$ in $L^2$.
Moreover, the fact that $[\delta_{\ell,h},\p_{x_k}]=0$ and interpolation imply that for any $s\in R$, $\delta_{\ell,h}:H^{s+1}(\R^n)\to H^{\rm s}(\R^n)$ is bounded for all $h\in \R$, uniformly in $h$.

\begin{proof}[Proof of Proposition \ref{prop:estimate}]
We will first assume that $\Pi u^-=0$, i.e. that $u^-=\td{u}^-$ (see \eqref{eq:potential_part}-\eqref{eq:solenoidal}), and show that if $(u^+,u^-)=(u^+,\td{u}^-)\in D(P)$ then we have the estimate
\begin{equation}\label{eq:estimate_u_td}
  \begin{aligned}
  \|u^+&\|^2_{H^2({M}^+)}+\|\div u^-\|^2_{H^1({M}^-)}, \\
  &\leq C\left(\|P^+u^+\|_{L^2(M^+)}^2+\|P^-u^-\|_{L^2(M^-)}^2 +\| u^+\|_{H^1(M^-)}^2 +\| u^-\|_{H^1(M^+)}^2  \right), \quad \Pi u^-=0.
\end{aligned}
\end{equation}
Once \eqref{eq:estimate_u_td} has been established under the assumption $u^-=\td{u}^-$,
the statement of the proposition follows for general $\bu\in D(P)$; we now demonstrate how to see this. 
Let $u^-=\td{u}^-+\Pi u^-$ and write $\td{u}^-=\n \w$, where $\w$ is determined up to constant by \eqref{eq:neumann_problem} and it has been chosen so that  $\|\w\|_{H^1(M^-)}= \|\w+\C\|_{H^1(M^-)/\C}=\inf_{z\in \C}\| \w+z\|_{H^1(M^-)}$ (the precompactness of $M^-$ implies that the infimum is realized for some complex number).
Notice that if $(u^+,u^-)\in D(q)  $ then $(u^+,u^-)\in D(P)$ if and only if $(u^+,\td{u}^-)\in D(P)$ because $(0,\Pi u^-)\in D(P)$ by \eqref{eq:domain_general_P}.

For each $r\geq 0$, we have the elliptic regularity estimate
\begin{align}
  \|\td{u}^-\|_{H^{r+1}(M^-)}&=\|\n\w\|_{H^{r+1}(M^-)}\leq C\|\w\|_{H^{r+2}(M^-)}\\
  &\leq C\big( \|\Delta \w \|_{H^r(M^-)}+\|\w\|_{H^1(M^-)}+\|\p_\nu\w\|_{H^{r+1/2}(\Gamma)}\big)\\
  &\leq C\big( \|\div(\td{u}^-) \|_{H^r(M^-)}+\|\w\|_{H^1(M^-)}+\|u^+\cdot \nu\|_{H^{r+1/2}(\Gamma)}\big) \\
  &\leq C\left( \|\div({u}^-) \|_{H^r(M^-)}+\|u^+\|_{H^{r+1}(M^+)}\right) ,\label{eq:elliptic_regularity}
\end{align}
using the trace theorem and the fact that \cite[Theorem 4.10(ii)]{McLean-book} implies
\begin{equation}\label{eq:omega_estimate}
\|\w\|_{H^1(M^-)}=\|\w+\C\|_{H^1(M^-)/\C}\leq C\big(\|\div u^-\|_{L^2(M^-)}+\|\nu\cdot u^+\|_{H^{1/2}(\G)}\big). 
\end{equation}
So suppose that $(u^+,u^-)\in D(P)$ is given. Then $(u^+,\td{u}^-)\in D(P)$, so if \eqref{eq:estimate_u_td} is known to hold  with $u^-$ replaced by $\td{u}^-$, we obtain the original claim \eqref{eq:prop_estimate} using that
 $P^-\td{u}^-=P^-{u}^-$, $\div u^-=\div \td{u}^-$,
and
 \eqref{eq:elliptic_regularity} for $r=0$.

\medskip

So now assume that $\bu\in D(P)$ and $\Pi u^-=0$.
To prove \eqref{eq:estimate_u_td} we localize in neighborhoods where we can choose coordinates conveniently.
Assume that $U$  is a neighborhood of a point in $\G$ and semigeodesic coordinates are chosen on $U$ such that  $\G$ is given  locally by $x_3=0$ and such that $\nu=-\p_{x_3}\big|_{\G}$, and consider $\chi\in C_c^\infty(U)$.
With some abuse of notation we write $\chi u^\pm:=\chi \big|_{M^\pm}u^\pm$ and $\chi \bu:=(\chi u^+,\chi u^-)$. Note that if $\bu\in D(P)$ we have $\chi \bu\in \calH^1_{\div,\rm{tr}}$ but generally not $\chi \bu\in D(P)$.
(This is one of the reasons why we have to do the localization explicitly by multiplying by $\chi$ instead of assuming that $u\in D(P)$ and is supported in $U$; we would have some loss of generality with such an assumption.)
For $\ell=1,2$ we can form the difference quotient $\d_{\ell,h}(\chi\bu):=(\d_{\ell,h}(\chi u^+),\d_{\ell,h}(\chi u^-))$ (throughout this proof we assume that $|h|$ is small enough that $\supp \delta_{\ell,h}(\chi \bu)\subset\subset U$).
Again, in general $\bu\in D(P)\not\implies \delta_{h,\ell}(\chi \bu)\in D(P)$ due to the Neumann transmission condition \eqref{tr_2_aux},
but $\bu\in D(P)\subset\calH^1_{\div,\rm{tr}}\implies  \delta_{\ell,h}(\chi \bu)\in \calH^1_{\div,\rm{tr}}$.

For $v^\pm$, $w^\pm\in H^1({M}^\pm; \C\otimes TM)$, we set below
\begin{align}
q^+(v^+,w^+)=&(\Div  v^+ ,\Div w^+ )_{L^2(M^+,\l_{\rm s}dv_g)}+(d^{\rm s}  v^+ ,d^{\rm s} w^+ )_{L^2(M^+,2\mu_{\rm s} dv_g)},\\
 q^-(v^-,w^-)=&(\div  v^-,\div w^-)_{L^2(M^-,\l_{\rm f}dv_g)}.
\end{align}
 By \cite[Lemma 4.15]{McLean-book} we then have (assuming $v^\pm$, $u^\pm$ are supported in $U\cap \oM^\pm$)
\begin{align}\label{eq:difference_quot_dirichlet}
  \qquad |q^\pm(\delta_{\ell,h}v^\pm,w^\pm)-q^\pm(v^\pm,\delta_{\ell,-h}w^\pm)|
  \leq C\|v^\pm\|_{H^1(M^\pm)}\|w^\pm \|_{H^1(M^\pm)},\quad |h|\text{ small},\quad  \ell= 1,2.
\end{align}
Inequality \eqref{eq:estimate_u_td} will be proved by means of the following coerciveness type estimates, which follow from Korn's inequality and \eqref{eq:difference_quot_dirichlet}:
\begin{align}
  \|\delta_{\ell,h}&(\chi  u^+)\|^2_{H^1({M}^+)}+\|\div \delta_{\ell,h}(\chi u^-)\|_{L^2({M}^-)}^2\\
  &\leq C\big(|q^+(\delta_{\ell,h}(\chi  u^+),\delta_{\ell,h}(\chi  u^+))+q^-(\delta_{\ell,h}(\chi u^-),\delta_{\ell,h}(\chi u^-))| +\|\delta_{\ell,h}(\chi  u^+)\|^2_{L^2({M}^+)}\big)\\
  &\leq C\Big(|q^+(\chi  u^+,\delta_{\ell,-h}\delta_{\ell,h}(\chi  u^+))+q^-(\chi u^-,\delta_{\ell,-h}\delta_{\ell,h}(\chi u^-))|\\
  &\quad+\| \chi  u^+ \|_{H^1({M}^+)}\|\delta_{\ell,h}(\chi{u}^+)\|_{H^1({M}^+)}+\| \chi u^- \|_{H^1({M}^-)}\|\delta_{\ell,h}(\chi u^-)\|_{H^1({M}^-)} \\
  &\qquad+ \|\delta_{\ell,h}(\chi  u^+)\|^2_{L^2({M}^+)}\Big).\label{eq:coerciveness_1}
\end{align}
Eventually our goal is to let $h\to0$, thus turning the difference quotients into derivatives, once we manage to move all of the expressions involving  highest order derivatives and difference quotients of $u^\pm$ to the left hand side.
We will establish two claims that will allow us to further manipulate \eqref{eq:coerciveness_1}: The purpose of Claim \ref{cl_1} is to estimate $\|\d_{\ell,h}(\chi u^-)\|_{H^1(M^-)}$, which appears in the right hand side of \eqref{eq:coerciveness_1}, by $\|\div(\d_{\ell,h}(\chi u^-)) \|_{L^2(M^-)} + \|\d_{\ell,h}(\chi{u}^+)\|_{H^{1}(M^+)}$ (which appears in its left hand side) plus controlled quantities.
The purpose of Claim \ref{cl_2} is to show how integration by parts can be used to replace the quadratic form terms in \eqref{eq:coerciveness_1} by expressions involving $P^\pm u^\pm$.

\begin{claim}\label{cl_1}
  If $\Pi u^-=0$ and $\ell=1,2$, 
\begin{equation}\label{eq:H1-div}
  \begin{aligned}
  \|\d_{\ell,h}&(\chi u^-)\|_{H^1(M^-)}, \\&\leq C\big(\|\div(\d_{\ell,h}(\chi u^-)) \|_{L^2(M^-)}+\|{u}^- \|_{H^1(M^-)}+\|{u}^+ \|_{H^1(M^+)}
    +\|\d_{\ell,h}(\chi{u}^+)\|_{H^{1}(M^+)}\big).
  \end{aligned}
\end{equation}
\end{claim}
\noindent To prove Claim \ref{cl_1}, write $u^-=\td{u}^-=\n\w$, where $\w $ solves \eqref{eq:neumann_problem} and satisfies $\|\w\|_{H^1(M^-)}= \|\w+\C\|_{H^1(M^-)/\C}$.
Below we write $m_\chi$ for the operator of multiplication by $\chi$.
Using elliptic regularity estimates (e.g. \cite[Theorem 4.18]{McLean-book})
\begin{align}
  \|\d_{\ell,h}(\chi u^-)\|_{H^1(M^-)}  =&\|\d_{\ell,h}m_\chi\n \w\|_{H^1(M^-)}, \\
  \leq& C\left( \|\d_{\ell,h}m_\chi\w\|_{H^2(M^-)}+\|[\d_{\ell,h}m_\chi,\n]{\w}\|_{H^1(M^-)}\right)\\
  \leq& C\big( \|\Delta (\d_{\ell,h}m_\chi\w)\|_{L^2(M^-)}
  +\|\d_{\ell,h}m_\chi\w\|_{H^1(M^-)} \\*
  &\qquad +\|\p_\nu(\d_{\ell,h}m_\chi\w)\|_{H^{1/2}(\G)}+
  \|[\d_{\ell,h}m_\chi,\n]{\w}\|_{H^1(M^-)}\big)\\
  \leq& C\big( \|\div (\d_{\ell,h}m_\chi\n \w)\|_{L^2(M^-)}+ \|\div ([\d_{\ell,h}m_\chi,\n] \w)\|_{L^2(M^-)}+\|\w\|_{H^2(M^-)} \\*
  &\qquad +\|\p_\nu(\d_{\ell,h}m_\chi\w)\|_{H^{1/2}(\G)}+
  \|[\d_{\ell,h}m_\chi,\n]{\w}\|_{H^1(M^-)}\big)\\
  \leq& C\big( \|\div (\d_{\ell,h}(\chi u^-))\|_{L^2(M^-)}+\|\n \w\|_{H^1(M^-)}+\|\w\|_{H^1(M^-)} \\*
  &\qquad +\|\p_\nu(\d_{\ell,h}m_\chi\w)\|_{H^{1/2}(\G)}+
  \|[\d_{\ell,h}m_\chi,\n]{\w}\|_{H^1(M^-)}\big).\label{last_line}
\end{align}
Now one checks that if $a\in C^\infty(\o{M}^-)$ and $j=1,2,3$
\begin{equation}
  [\d_{\ell,h}m_\chi, a\p_{x_j}]=[\d_{\ell,h},m_{\chi a}]\p_{x_j}+a[m_\chi,\delta_{\ell,h}]\p_{x_j}+a\d_{\ell,h}[m_\chi,\p_{x_j}], %
\end{equation}
so by \cite[Lemma 4.14(iii)]{McLean-book}, which describes the behavior of the first two commutators,
\begin{align}
  \|[\d_{\ell,h}m_\chi,\n]{\w}\|_{H^1(M^-)}\leq C\|\w\|_{H^2(M^-)}\leq C\left(\|\n \w\|_{H^1(M^-)}+\|\w\|_{H^1(M^-)}\right)\\
  \leq  C \left(\|{u}^-\|_{H^1(M^-)}+\|\w\|_{H^1(M^-)}\right)\label{commutator}.
\end{align}
Further, using the trace theorem, the fact that $\p_\nu\w=u^+\cdot \nu $, and that $[\d_{\ell,h},\p_\nu]=0$ in our coordinates,  we check that
\begin{align}
  \|\p_\nu(\d_{\ell,h}m_\chi\w)\|_{H^{1/2}(\G)}
  &\leq C\left(\|\w \|_{H^2(M^-)}+\|\d_{\ell,h}(\chi u^+\cdot \nu) \|_{H^{1/2}(\G)}\right)\\
&
\leq C\left(\|{u}^-\|_{H^1(M^-)}+\|\w\|_{H^1(M^-)}+\|\d_{\ell,h}(\chi u^+)) \|_{H^{1}(M^+)}\right).\label{eq:normal_der_estimate}
\end{align}
Finally, estimating $\|\w\|_{H^1(M^-)}$ using \eqref{eq:omega_estimate} and the trace theorem, we obtain the claim by \eqref{last_line}, \eqref{commutator}, and \eqref{eq:normal_der_estimate}.

\smallskip

\begin{claim}\label{cl_2}
Given $\chi\in C_c^\infty(U;\R)$, $\bu=(u^+,u^-)\in D(P)$ with $\Pi u^-=0$, and $\bv=(v^+,v^-)\in \calH^1_{\div,\rm{tr}}$,
\begin{align}
  \big|\big[q^+(\chi  u^+, v^+)+q^-(\chi {u}^-, v^-)\big]-&\big[(-P^+ u^+,\chi  v^+)_{L^2(M^+,\r_{\rm s}dv_g)}+(- P^- {u}^-, \chi v^-)_{L^2(M^-,\r_{\rm f}dv_g)}\big]\big|
  \\*
  &\leq C\Big(\|u^+ \|_{H^1(M^+)}\big(\|v^+ \|_{L^2(M^+)}+\|\t (v^+) \|_{H^{-1/2}(\G)} \big)\\*
  &\qquad+\|u^- \|_{H^1(M^-)}\big(\|v^- \|_{L^2(M^-)}+\|\t (v^-) \|_{H^{-1/2}(\G)} \big)\Big).\label{boundary_terms}
   \end{align}
\end{claim}
\noindent To prove the claim, note that since  $\bu\in D(P)$ and $\chi \bv\in \calH_{\rm{div,tr}}^1$, %
\begin{equation}\label{eq:ibp}
  \begin{aligned}
    (-P^+ u^+,\chi  v^+)_{L^2(M^+,\r_{\rm s}dv_g)}+(- P^- {u}^-, \chi v^-)_{L^2(M^-,\r_{\rm f}dv_g)}=q^+(  u^+,\chi v^+)+q^-( {u}^-,\chi v^-).
  \end{aligned}
\end{equation}
The result will follow from moving $\chi$ from the second to the first argument of $q^\pm$ and estimating the resulting additional terms:
we have (recall that $\nu$ is inward pointing for $M^-$)
\begin{equation}
  \begin{aligned}
        q^-(u^-, \chi v^-)=&(\Div  u^- ,\Div \chi v^- )_{L^2(M^-,\l_{\rm f} dv_g)} \\
        =&(\chi \Div  u^- ,\Div  v^- )_{L^2(M^-,\l_{\rm f}dv_g)}
         +( \Div  u^- ,\n \chi \cdot  v^- )_{L^2(M^-,\l_{\rm f}dv_g)} \\
        =&q^-(\chi u^-,v^-)-(\l_{\rm f}\n \chi\cdot   u^- ,\Div  v^- )_{L^2(M^-,dv_g)}
        +( \Div  u^- ,\n \chi \cdot  v^- )_{L^2(M^-,\l_{\rm f}dv_g)} \\
        =&q^-(\chi u^-,v^-)+(\n (\l_{\rm f}\n \chi\cdot   u^- ),v^- )_{L^2(M^-,dv_g)}+\lg \t (\l_{\rm f}\n \chi\cdot   u^- ),\nu \cdot\t( v^-) \rg_{L^2(\G,dA)} \\*
        &\qquad +( \Div  u^- ,\n \chi \cdot  v^- )_{L^2(M^+,\l_{\rm f}dv_g)} %
  \end{aligned}
\end{equation}
and,  with $S$ denoting symmetrization,
\begin{equation}
  \begin{aligned}
        q^+(   u^+, \chi v^+)=&(\Div  u^+ ,\Div \chi v^+ )_{L^2(M^+,\l_{\rm s}dv_g)}+(d^{\rm s}  u^+ ,d^{\rm s} \chi v^+ )_{L^2(M^+,2\mu_{\rm s} dv_g)} \\
        =&(\chi \Div  u^+ ,\Div  v^+ )_{L^2(M^+,\l_{\rm s}dv_g)}+(\chi d^{\rm s}  u^+ ,d^{\rm s}  v^+ )_{L^2(M^+,2\mu_{\rm s} dv_g)} \\*
        &\quad +( \Div  u^+ ,\n \chi \cdot  v^+ )_{L^2(M^+,\l_{\rm s}dv_g)}+( d^{\rm s}  u^+ , S(\n\chi \otimes  v^+) )_{L^2(M^+,2\mu_{\rm s} dv_g)} \\
        =&q^+(\chi u^+,v^+)-(\n \chi\cdot   u^+ ,\Div  v^+ )_{L^2(M^+,\l_{\rm s}dv_g)}-(S(\n \chi \otimes   u^+) ,d^{\rm s}  v^+ )_{L^2(M^+,2\mu_{\rm s} dv_g)} \\
        &\quad +( \Div  u^+ ,\n \chi \cdot  v^+ )_{L^2(M^+,\l_{\rm s}dv_g)}+( d^{\rm s}  u^+ , S(\n\chi \otimes  v^+) )_{L^2(M^+,2\mu_{\rm s} dv_g)} \\
        =&q^+(\chi u^+,v^+)+(\n (\l_{\rm s}\n \chi\cdot   u^+ ),v^+ )_{L^2(M^+,dv_g)}-\lg \l_{\rm s}\n \chi\cdot   u^+ ,\nu \cdot v^+ \rg_{L^2(\G,dA)} \\
        &\quad +(\div (2 \mu_{\rm s} S(\n \chi \otimes   u^+)) ,  v^+ )_{L^2(M^+,dv_g)}-\lg  \nu \cdot (2 \mu_{\rm s} S(\n \chi \otimes   u^+)) ,  v^+ \rg_{L^2(\G,dA)} \\
        &\qquad +( \Div  u^+ ,\n \chi \cdot  v^+ )_{L^2(M^+,\l_{\rm s}dv_g)}+( d^{\rm s}  u^+ , S(\n\chi \otimes  v^+) )_{L^2(M^+,2\mu_{\rm s} dv_g)}.
  \end{aligned}
\end{equation}
Hence we find, using Cauchy-Schwartz 
\begin{equation}\label{eq:move_chi}
  \begin{aligned}
    |q^\pm(   u^\pm, \chi v^\pm)-&q^\pm( \chi  u^\pm,  v^\pm)|\\
    &\leq C\left(\|u^\pm\|_{H^1(M^\pm)}\|v^\pm\|_{L^2(M^\pm)} +\|\t u^\pm\|_{H^{1/2}(\G)}\|\t v^\pm\|_{H^{-1/2}(\G)} \right).
  \end{aligned}
\end{equation}
Combining \eqref{eq:ibp} with \eqref{eq:move_chi} and estimating $\|\t u^\pm\|_{H^{1/2}(\G)}$ by $\|u^\pm\|_{H^1(M^\pm)}$ via the trace theorem, we obtain the claim.

\medskip

Now substitute   $v^\pm= \delta_{\ell,-h}\delta_{\ell,h}(\chi u^\pm)$ for $\ell=1,2$ into \eqref{boundary_terms} and use the following estimates:
 \begin{align}
  \|\t \delta_{\ell,-h}\delta_{\ell,h}(\chi u^\pm)\|_{H^{-1/2}(\G)}=\| \delta_{\ell,-h}\t \delta_{\ell,h}(\chi u^\pm)\|_{H^{-1/2}(\G)}& \leq C \| \t \delta_{\ell,h}(\chi u^\pm)\|_{H^{1/2}(\G)}, \\
  &\qquad \leq C \| \delta_{\ell,h}(\chi u^\pm)\|_{H^{1}(M^\pm)}, \\
  \text{and }\quad \|\delta_{\ell,-h}\delta_{\ell,h}(\chi u^\pm)\|_{L^2(M^\pm)}
  &\leq C \|\delta_{\ell,h}(\chi u^\pm)\|_{H^1(M^\pm)}.
 \end{align}
Combining the resulting estimate with  \eqref{eq:coerciveness_1} and Cauchy-Schwartz we obtain
  \begin{align}
    \|\delta_{\ell,h}&(\chi u^+)\|^2_{H^1({M}^+)}+\|\div \delta_{\ell,h}(\chi u^-)\|_{L^2({M}^-)}^2\\
   &\leq C\Big(\|P^+ u^+\|_{L^2({M}^+)}\|\delta_{\ell,h}(\chi u^+) \|_{H^{1}({M}^+)}+\|P^-{u}^-\|_{L^2({M}^-)}\|\delta_{\ell,h}(\chi u^-) \|_{H^{1}({M}^-)} \\
    &\quad+\| u^+\|_{H^{1}(M^+)}\| \delta_{\ell,h}(\chi u^+)\|_{H^{1}(M^+)}+\| u^-\|_{H^{1}(M^-)}\| \delta_{\ell,h}(\chi u^-)\|_{H^{1}(M^-)}+\|u^+\|^2_{H^1(M^+)}\Big).\label{eq:b16}
  \end{align}
Using the inequality $ab\leq\frac{1}{2}( \e a^2 +\frac{1}{\e} b^2 )$ for sufficiently small $\e$ together with  \eqref{eq:H1-div}, \eqref{eq:b16} implies
\begin{align}
  \|\delta_{\ell,h}(\chi u^+)\|^2_{H^1({M}^+)}&+\|\div \delta_{\ell,h}(\chi u^-)\|_{L^2({M}^-)}^2\\*
 &\leq C\Big(\|P^+ u^+\|_{L^2({M}^+)}^2+\|P^- u^-\|_{L^2({M}^-)}^2+\| u^+\|_{H^{1}(M^+)}^2+\| {u}^-\|_{H^{1}(M^-)}^2\Big).
\end{align}
Sending $h\to 0$ we find that for $\ell=1,2$ 
\begin{equation}
  \begin{aligned}\label{eq:H_1_first}
      \|\p_{x_\ell}(\chi u^+)\|^2_{H^1({M}^+)}&+\|\div \p_{x_\ell}(\chi u^-)\|_{L^2({M}^-)}^2\\
      &\leq C\Big(\|P^+ u^+\|_{L^2({M}^+)}^2+\|P^- u^-\|_{L^2({M}^-)}^2+\| u^+\|_{H^{1}(M^+)}^2+\| {u}^-\|_{H^{1}(M^-)}^2\Big), 
  \end{aligned} 
\end{equation}
thus 
\begin{equation}
\begin{aligned}\label{eq:H_1}
    \|\p_{x_\ell}(\chi u^+)\|_{H^1({M}^+)}^2&+\|\p_{x_\ell}\div (\chi u^-)\|_{L^2({M}^-)}^2\\
    &\leq C\Big(\|P^+ u^+\|_{L({M}^+)}^2+\|P^- u^-\|_{L({M}^-)}^2+\| u^+\|_{H^{1}(M^+)}^2+\| {u}^-\|_{H^{1}(M^-)}^2\Big), 
\end{aligned} 
\end{equation}
using that $\|\p_{x_\ell}\div \big(\chi u^-)\|_{L^2({M}^-)}^2\leq C( \|\div \p_{x_\ell}(\chi u^-)\|_{L^2({M}^-)}+\|u^-\|_{H^1(M^-)}\big) $.

\medskip

For the derivatives normal to the interface, $(u^+,u^-)\in D(P)$ implies that $P^+ u^+=f^+\in L^2({M}^+;\C\otimes TM)$.
Since  $\G$ is non-characteristic for $P^+$, we have that
\begin{equation}\label{eq:normal_derivatives}
  a^+(x)\p_{x_3}^2 (\chi u^+) = \td{P}^+ (\chi u^+)+ Q^+(u^+) +\chi f^+,
\end{equation}
where $\det a^\pm\neq 0$, $Q^+$ is an operator of order 1
 and $\td{P}^+$ is a differential operator of order 2 in which the order of normal derivatives appearing is no more than 1. Hence %
\begin{equation}\label{eq:plus_side}
    \|\p_{x_3}(\chi u^+)\|^2_{H^1({M}^+)}
\leq C\Big(\sum_{j=1}^2\|\p_{x_j}(\chi u^+)\|_{H^1(M^+)}^2+\|P^+ u^+\|_{L^2({M}^+)}^2+\|u^+\|_{H^1(M^+)}^2\Big).
\end{equation}
On the other hand,
\begin{equation}
\label{eq:minus_side}
\begin{aligned}
  \|\p_{x_3}\div (\chi u^-)\|_{L^2({M}^-)}^2&\leq C\left(\|\p_{x_3}\l_{\rm f} \div  u^-\|_{L^2({M}^-)}^2+\| \div  u^-\|_{L^2({M}^-)}^2 +\|u^-\|_{H^1(M^-)}^2\right)\\
  &\leq C\left(\|P^-  u^-\|_{L^2({M}^-)}^2+\| \div  u^-\|_{L^2({M}^-)}^2 +\| u^+\|_{H^1(M^+)}^2\right),
\end{aligned}
\end{equation}
where we used \eqref{eq:elliptic_regularity} in the last step.
Adding \eqref{eq:plus_side} and \eqref{eq:minus_side} and using \eqref{eq:H_1} to estimate the terms appearing in the summation in \eqref{eq:plus_side}, we find that \eqref{eq:H_1} also holds for $\ell=3$.

If $\supp \chi \cap \G=\emptyset$ the proof of \eqref{eq:H_1} for $\ell=1,2,3$  can be done in a similar way, though it is simpler.
Using a partition of unity we obtain \eqref{eq:prop_estimate}, finishing the proof of Proposition~\ref{prop:estimate}. 
\end{proof}

\begin{remark}
  Even though the proof is written assuming that $M\subset \R^3$, it would work in exactly the same way for any dimension $\geq 2$.
\end{remark}

We finally have:

\begin{proof}[Proof of Corollary \ref{cor:estimate}]
The estimate \eqref{eq:cor_estimate} is shown for $k=0$ in Proposition \ref{prop:estimate}. 
Suppose it is known for some fixed $k\geq 0$. We will show that it also holds for $k+1$.
Recall the notation $\td{u}^-$ from \eqref{eq:potential_part}.
By our inductive hypothesis and \eqref{eq:elliptic_regularity}, $u^+\in H^{k+2}(M^+;\C\otimes TM)$, $\td{u}^-\in H^{k+2}(M^-;\C\otimes TM)$.
We also have $P^\pm u^\pm\in H^{k+1}(M^\pm;\C\otimes TM)$.
If $V$ is a vector field on $M$ tangent to $\G$ and $\p M$ and $\calL_V$ denotes Lie  derivative,  $[\div, \calL_V]$, $[d^{\rm s}, \calL_V]$ are  operators of order 1 and $[\calL_V,P^\pm]$ are operators of order 2 because the principal symbol of $\calL_V$ is a scalar multiple of the identity. 
Note that $(\calL_V u^+,\calL_V \td{u}^-)$ in general does not satisfy the transmission/boundary conditions in \eqref{eq:transmission_domain} but the fact that those are satisfied for $(u^+,\td{u}^-)$ implies that
\begin{align} %
  \nu \cdot(\calL_V u^+-\calL_V \td{u}^-)\big|_{\G}\in H^{k+3/2}(\G;\C\otimes TM),\quad & N(\calL_Vu^+)\big|_{\p M}\in H^{k+1/2}(\p M;\C\otimes TM),\\
   N(\calL_V u^+)-\l_{\rm f} \div (\calL_V \td{u}^-)\nu& \big|_{\G}\in H^{k+1/2}(\G;\C\otimes TM).
 \end{align}
Thus we can construct suitable extension operators off the boundary and interface (see e.g. \cite[Lemma 3.36]{McLean-book}) to find $w^+\in H^{k+2}(M^+;\C \otimes TM)$ which satisfies 
\begin{equation}
  \begin{cases}
    \nu \cdot w^+ =-\nu\cdot (\calL_Vu^+-\calL_V\td{u}^-)&\quad \text { on }\G, \\
    N(w^+)=-N(\calL_V u^+)+\l_{\rm f} \div (\calL _V \td{u}^-)\nu&\quad \text { on }\G, \\
     N(w^+)=-N(\calL_Vu^+)&\quad \text { on }\p M, 
  \end{cases}
\end{equation}
and
\begin{equation}\label{eq:extension_estimate}
  \|w^+\|_{H^{k+2}(M^+)}^2\leq C\big( \|u^+\|_{H^{k+2}(M^+)}^2+\|\td{u}^-\|_{H^{k+2}(M^-)}^2\big).
\end{equation}

We now wish to use the inductive hypothesis, namely  \eqref{eq:cor_estimate} for our fixed $k\geq 0$.
Notice that $(\calL_Vu^++w^+,
\calL_V\td{u}^-)$ satisfies the transmission and boundary conditions in \eqref{eq:transmission_domain} by construction.
Moreover, for all $k\geq 0$, the inductive hypothesis and the order of the commutators $[\calL_V,P^\pm]$ imply that $P^+(\calL _V u^++w^+)\in H^k(M^+;\C\otimes TM)$, $P^-(\calL _V \td{u}^-)\in H^k(M^+;\C\otimes TM)$.
Using those facts and \eqref{eq:domain_general_P} it can be checked that $(\calL_Vu^++w^+,
\calL_V\td{u}^-)\in D(P)$.
Now use \eqref{eq:cor_estimate} for the second inequality:
    \begin{align}
      \|\calL_Vu^+&\|^2_{H^{k+2}({M}^+)}+\|\calL_V\div u^-\|^2_{H^{k+1}({M}^-)} \\*
      &\leq C\Big( \|\calL_Vu^++w^+\|^2_{H^{k+2}({M}^+)}+\|\div \calL_V \td{u}^-\|^2_{H^{k+1}({M}^-)} \\*
      &\hspace{1 in}+\|[\div ,\calL_V]\td{u}^-\|^2_{H^{k+1}({M}^-)}+\|w^+\|_{H^{k+2}(M^+)}^2\Big)\\*
      &\leq C\Big( \|P^+(\calL_Vu^++w^+)\|_{H^{k}(M^+)}^2+\|P^-\calL_V\td{u}^-\|_{H^{k}(M^-)}^2 \\*
      &\qquad +\|\calL_V u^++w^+\|_{H^1(M^+)}^2+\| \div \calL_V\td{u}^-\|_{L^2(M^-)}^2+\|\td{u}^-\|^2_{H^{k+2}({M}^-)}+\|u^+\|_{H^{k+2}(M^+)}^2\Big)\\
      &\leq C\Big( \|P^+u^+\|_{H^{k+1}(M^+)}^2+\|P^-\td{u}^-\|_{H^{k+1}(M^-)}^2\\*
       &\hspace{ 1 in}+ \| u^+\|_{H^{k+2}(M^+)}^2+\| \td{u}^-\|_{H^{k+2}(M^-)}^2+\|w^+\|_{H^{k+2}(M^+)}^2\Big)\\
      &\leq C\left( \|P^+u^+\|_{H^{k+1}(M^+)}^2+\|P^-{u}^-\|_{H^{k+1}(M^-)}^2 + \| u^+\|_{H^{k+2}(M^+)}^2+\|\div  {u}^-\|_{H^{k+1}(M^-)}^2\right),
    \end{align}
using \eqref{eq:elliptic_regularity} and \eqref{eq:extension_estimate}.
Using the inductive hypothesis to replace the last two terms  we find
  \begin{equation}\label{eq:tangential_derivatives}
    \begin{aligned}
      \|\calL_Vu^+&\|^2_{H^{k+2}({M}^+)}+\|\calL_V\div u^-\|^2_{H^{k+1}({M}^-)} \\
      &\leq C\left( \|P^+u^+\|_{H^{k+1}(M^+)}^2+\|P^-{u}^-\|_{H^{k+1}(M^-)}^2 + \| u^+\|_{H^{1}(M^+)}^2+\|\div  {u}^-\|_{L^2(M^-)}^2\right).
    \end{aligned}
  \end{equation}

For the derivatives normal to the interface and boundary we can use the same method as in the proof of Proposition \ref{prop:estimate} to show that in local coordinates with respect to which $x_3=0$ represents the interface $\G$ or $\p M$, the expression $\|\p_{x_3}^{k+3}(\chi u^+)\|_{L^2(M^+)}^2+\|\p_{x_3}^{k+2} (\chi \div {u}^-)\|_{L^2(M^-)}^2$, where $\chi$ is supported in a neighborhood where the coordinates are valid,  are estimated by the right hand side of \eqref{eq:tangential_derivatives}. 
With a partition of unity we obtain \eqref{eq:cor_estimate} for $k+1$.

\smallskip

The statement regarding $\bu\in D(P^m)$ follows for $m=1$
 by \eqref{eq:cor_estimate}.
If $m\geq 2$, we use \eqref{eq:cor_estimate} for $k+2=2m$.
One would like to estimate the resulting term $\|P^+u^+\|_{H^{k}(M^+)}^2+\|P^-u^-\|_{H^{k}(M^-)}^2$, by replacing $u^\pm$ by $P^\pm u^\pm$ in \eqref{eq:cor_estimate}, and proceed inductively to show the claim.
However such an estimate doesn't follow immediately from \eqref{eq:cor_estimate} since the latter only gives an estimate on $\|P^+u^+\|_{H^{k}(M^+)}^2+\|\div P^-u^-\|_{H^{k-1}(M^-)}^2$.
We can circumvent the issue by means of the following estimate: for any $r\geq 1$ we have, using elliptic regularity estimates for the second inequality below,
\begin{align}
  \|P^- u^-&\|_{H^r(M^-)}=\|\r_{\rm f}^{-1}\n\l_{\rm f}\div u^-\|_{H^r(M^-)}\leq C \|\l_{\rm f}\div u^-\|_{H^{r+1}(M^-)}\\
  \leq & C\left(\|(\div  \r_{\rm f}^{-1}\n)\l_{\rm f}\div u^-\|_{H^{r-1}(M^-)}+\|\l_{\rm f} \div  u^-\|_{H^{1}(M^-)}+\|\nu\cdot \n (\l_{\rm f} \div  u^-)\|_{H^{r-1/2}(\G)}\right)\\
  \leq & C\left(\|\div  P^- u^-\|_{H^{r-1}(M^-)}+\|\l_{\rm f} \div  u^-\|_{H^{1}(M^-)}+\|\nu\cdot \n (\l_{\rm f} \div  u^-)\|_{H^{r-1/2}(\G)}\right)\\
  \leq & C\left(\|\div  P^- u^-\|_{H^{r-1}(M^-)}+\|\l_{\rm f} \div  u^-\|_{H^{1}(M^-)}+\|\nu\cdot \t(P^-  u^-)\|_{H^{r-1/2}(\G)}\right)\\
  \leq & C\left(\|\div  P^- u^-\|_{H^{r-1}(M^-)}+\|\l_{\rm f} \div  u^-\|_{H^{1}(M^-)}+\|\nu\cdot \t(P^+  u^+)\|_{H^{r-1/2}(\G)}\right)\label{eq:minus_to_plus} \\
  \leq & C\left(\|\div  P^- u^-\|_{H^{r-1}(M^-)}+\|\l_{\rm f} \div  u^-\|_{H^{1}(M^-)}+\|P^+  u^+\|_{H^{r}(M^+)}\right),
\end{align}
where in \eqref{eq:minus_to_plus} we used the fact that if $\bu\in D(P^m)$ for $m\geq 2$, then since $(P^+u^+,P^-u^-)\in D(P)$ we have $ \nu \cdot \t(P^+u^+)=\nu\cdot \t(P^-u^-)$.
Hence for $k\geq 1                  $
\begin{equation}
  \begin{aligned}
    \|P^+u^+\|_{H^{k}(M^+)}^2&+\|P^-u^-\|_{H^{k}(M^-)}^2\\
    &\leq C \left(\|P^+ u^+\|_{H^k(M^+)}^2+\|\div  P^- u^-\|_{H^{k-1}(M^-)}^2 +\|\div  u^-\|_{H^{1}(M^-)}^2\right),
  \end{aligned}
\end{equation}
and \eqref{eq:cor_estimate} can be used to push the induction through. This completes the proof of the corollary. 
\end{proof}

%

%

%


\end{document}